\documentclass[final]{siamart190516}

\usepackage[normalem]{ulem}

\usepackage{hyperref}

\hypersetup{
pdftitle={Scattering by finely-layered obstacles:~frequency-explicit bounds and homogenization},
pdfauthor={T. Chaumont-Frelet and E. A. Spence}
}

\newcommand{\hmu}{\widehat \mu}

\usepackage{tikz}
\usepackage{pgf}
\usetikzlibrary{arrows}
\usetikzlibrary{calc}
\usetikzlibrary{decorations.pathmorphing}

\usepackage{rotating} 

\usepackage{mathrsfs}
\DeclareMathAlphabet{\mathpzc}{OT1}{pzc}{m}{it}

\usepackage{color}
\usepackage{amsmath}
\usepackage{graphicx}
\usepackage{amsfonts}
\usepackage{amssymb}%
\usepackage{latexsym}
\usepackage{psfrag}
\usepackage{accents}

\newtheorem{condition}[theorem]{Condition}

\newtheorem{example}[theorem]{Example}
\newtheorem{remark}[theorem]{Remark}

\numberwithin{equation}{section}
\numberwithin{table}{section}
\numberwithin{figure}{section}

\usepackage[a4paper]{geometry}
\newcommand{\noi}{\noindent}

\newcommand{\R}{\mathbb{R}}

\newcommand{\cL}{{\cal L}}

\newcommand{\cZ}{{\cal Z}}
\newcommand{\bx}{\mathbf{x}}
\newcommand{\be}{\mathbf{e}}

\newcommand{\bxi}{\boldsymbol{\xi}}
\newcommand{\bn}{\mathbf{n}}

\newcommand{\bv}{\mathbf{v}}
\newcommand{\bw}{\mathbf{w}}

\newcommand{\by}{\mathbf{y}}
\newcommand{\bz}{\mathbf{z}}

\newcommand{\bQ}{\mathbf{Q}}

\newcommand{\supp}{\mathrm{supp}}

\newcommand{\bZ}{\mathbf{Z}}

\newcommand{\re}{{\rm e}}
\newcommand{\ri}{{\rm i}}
\newcommand{\rd}{{\rm d}}

\newcommand{\beq}{\begin{equation}}
\newcommand{\eeq}{\end{equation}}
\newcommand{\beqs}{\begin{equation*}}
\newcommand{\eeqs}{\end{equation*}}
\newcommand{\bit}{\begin{itemize}}
\newcommand{\eit}{\end{itemize}}
\newcommand{\ben}{\begin{enumerate}}
\newcommand{\een}{\end{enumerate}}
\newcommand{\bal}{\begin{align}}
\newcommand{\eal}{\end{align}}
\newcommand{\bals}{\begin{align*}}
\newcommand{\eals}{\end{align*}}
\newcommand{\bse}{\begin{subequations}}
\newcommand{\ese}{\end{subequations}}
\newcommand{\bpr}{\begin{proposition}}
\newcommand{\epr}{\end{proposition}}
\newcommand{\bre}{\begin{remark}}
\newcommand{\ere}{\end{remark}}
\newcommand{\bpf}{\begin{proof}}
\newcommand{\epf}{\end{proof}}
\newcommand{\ble}{\begin{lemma}}
\newcommand{\ele}{\end{lemma}}
\newcommand{\bco}{\begin{corollary}}
\newcommand{\eco}{\end{corollary}}
\newcommand{\bex}{\begin{example}}
\newcommand{\eex}{\end{example}}
\newcommand{\bth}{\begin{theorem}}
\newcommand{\enth}{\end{theorem}}

\newcommand{\Rea}{\mathbb{R}}
\newcommand{\Com}{\mathbb{C}}

\newcommand{\esssup}{\mathop{{\rm ess} \sup}}

\newcommand{\GR}{{\Gamma_R}}

\newcommand{\eps}{\varepsilon}

\newcommand{\pdiff}[2]{\frac{\partial #1}{\partial #2}}

\newcommand{\nus}{|u|^2}
\newcommand{\ngus}{|\nabla u|^2}

\newcommand{\gu}{\nabla u}

\newcommand{\nvs}{|v|^2}

\newcommand{\gv}{\nabla v}
\newcommand{\vb}{\overline{v}}
\newcommand{\gvb}{\overline{\nabla v}}

\newcommand{\tendi}{\rightarrow \infty}
\newcommand{\tendo}{\rightarrow 0}

\def\XXint#1#2#3{{\setbox0=\hbox{$#1{#2#3}{\int}$}
     \vcenter{\hbox{$#2#3$}}\kern-.5\wd0}}

\usepackage{hyperref}
\definecolor{myblue}{rgb}{0,0,0.5}
\definecolor{myred}{rgb}{0.6,0,0}
\hypersetup{colorlinks=true,
linkcolor=myblue,citecolor=myblue,filecolor=myblue,urlcolor=myblue}
\newcommand*{\N}[1]{\left\|#1\right\|}

\allowdisplaybreaks[4]

\newcommand{\tfa}{\text{ for all }}
\newcommand{\tfor}{\text{ for }}

\newcommand{\tas}{\text{ as }}
\newcommand{\tand}{\text{ and }}
\newcommand{\tst}{\text{ such that }}

\newcommand{\vertiii}[1]{{\left\vert\kern-0.25ex\left\vert\kern-0.25ex\left\vert #1
    \right\vert\kern-0.25ex\right\vert\kern-0.25ex\right\vert}}

\newcommand{\DOmegabar}{C^\infty(\overline{B_R})}

\newcommand{\GammaR}{\GR}

\newcommand{\domain}{{D}}
\newcommand{\domaingen}{{D}}
\newcommand{\domainin}{{\domain_{\rm in}}}
\newcommand{\domainout}{\domain_{\rm out}}

\newcommand{\var}{{n}}
\newcommand{\varmin}{{n_{\min}}}
\newcommand{\varmax}{{n_{\max}}}

\newcommand{\opL}{\cL}
\newcommand{\hatx}{\widehat{\bx}}
\newcommand{\ub}{\overline{u}}

\newcommand{\fout}{{f_{\rm out}}}

\newcommand{\SPD}{{\mathsf{SPD}}}

\newcommand{\Sym}{{\mathsf{Sym}}}

\newcommand*{\conj}[1]{\overline{#1}}

\newcommand{\mymatrix}[1]{\mathsf{#1}}
\newcommand{\MA}{{\mymatrix{A}}}\newcommand{\MB}{{\mymatrix{B}}}\newcommand{\MC}{{\mymatrix{C}}}

\newcommand{\MI}{{\mymatrix{I}}}

\newcommand{\MM}{{\mymatrix{M}}}

\newcommand{\sesqui}{b}

\newcommand{\DtN}{{\rm DtN}_k}

\usepackage{soul}

\definecolor{amcol}{rgb}{0.8,0,0}

\definecolor{escol}{rgb}{0,0,0.8}
\definecolor{estcol}{rgb}{0,0.5,0}
\definecolor{esnewcol}{rgb}{0,0.5,0}

\newcommand{\ZZZ}{\mathbb Z}
\newcommand{\RRR}{\mathbb R}
\newcommand{\CCC}{\mathbb C}

\newcommand{\xx}{\bx}
\newcommand{\yy}{\by}

\newcommand{\al}{{\boldsymbol \alpha}}
\newcommand{\bt}{{\boldsymbol \beta}}

\newcommand{\LL}{\boldsymbol L}
\newcommand{\HH}{\boldsymbol H}
\newcommand{\CC}{\boldsymbol C}

\newcommand{\per}{{\rm per}}

\newcommand{\ddiv}{\operatorname{div}}
\newcommand{\ccurl}{\boldsymbol{\operatorname{curl}}}

\newcommand{\pd}[2]{\frac{\partial #1}{\partial #2}}

\newcommand{\grad}{\nabla}
\renewcommand{\div}{\grad \cdot}
\newcommand{\curl}{\grad \times}

\newcommand{\gradx}{\grad_{\xx}}

\newcommand{\curlx}{\gradx \times}

\newcommand{\grady}{\grad_{\yy}}
\newcommand{\divy}{\grady \cdot}
\newcommand{\curly}{\grady \times}

\newcommand{\scurl}{\operatorname{curl}}
\newcommand{\vcurl}{\boldsymbol{\scurl}}

\newcommand{\eq}{:=}

\newcommand{\pphi}{\boldsymbol \phi}

\newcommand{\ppsi}{\boldsymbol \psi}

\newcommand{\qq}{\boldsymbol q}

\newcommand{\pp}{\boldsymbol p}

\newcommand{\ee}{\boldsymbol e}

\newcommand{\vv}{{\boldsymbol v}}

\newcommand{\hM}{\widehat{\mymatrix{M}}}

\newcommand{\Cemb}{C_{{\rm emb},s}}

\newcommand{\Din}{{D_{\rm in}}}
\newcommand{\Dout}{{D_{\rm out}}}

\newcommand{\nnu}{\boldsymbol \nu}

\newcommand{\EE}{E}

\newcommand{\hn}{\widehat n}
\newcommand{\hMA}{\widehat \MA}
\newcommand{\hMB}{\widehat \MB}
\newcommand{\hMC}{\widehat \MC}

\newcommand{\nH}{n^H}
\newcommand{\MAH}{\MA^H}
\newcommand{\hchi}{\widehat \chi}

\newcommand{\bin }{b_{\rm in} }
\newcommand{\bout}{b_{\rm out}}

\newcommand{\Clayer}{C_{\rm layer}(kR, kR_0,n_{\min})}
\newcommand{\Clayershort}{C_{\rm layer}}
\newcommand{\Clayerh}{C_{\rm layer}(kR, kR_0,\hn_{\min})}

\newcommand{\CsH}{\Csol(\MAH,\nH, k, R,R_0)}
\newcommand{\CsHshort}{C_{{\rm sol}, H}}
\newcommand{\Cse}{\Csol(\MA_\eps, n_\eps, k, R, R_0)}

\newcommand{\Cs}{C_{\rm sol}}

\newcommand{\PY}{\mathscr P}

\newcommand{\Csol}{C_{\rm sol}}
\newcommand{\Cwave}{C_{\rm wave}}
\newcommand{\Cone}{C_1}
\newcommand{\Ctwonew}{C_2}
\newcommand{\Cthreenew}{C_3}
\newcommand{\Cfournew}{C_4}
\newcommand{\Cfivenew}{C_5}

\newcommand{\CTR}{\widetilde{C}_{\rm DtN}}
\newcommand{\CPF}{C_{\rm PF}}

\title{Scattering by finely-layered obstacles:~frequency-explicit bounds and homogenization}

\author{%
T.~Chaumont-Frelet\thanks{
Inria, 2004 Route des Lucioles, 06902 Valbonne, France, and  
Laboratoire J.A.~Dieudonn\'e, Parc Valrose, 28 Avenue Valrose, 06000 Nice, France
(\email{theophile.chaumont@inria.fr})
}
\and
E.~A.~Spence\thanks{
Department of Mathematical Sciences, University of Bath, Bath, BA2 7AY, UK
(\email{E.A.Spence@bath.ac.uk})
}}

\begin{document}

\maketitle

\begin{abstract}
We consider the scalar Helmholtz equation with variable, discontinuous coefficients,
modelling transmission of acoustic waves through an anisotropic penetrable obstacle.
We first prove a well-posedness result and a frequency-explicit bound on the solution
operator, with both valid for sufficiently-large frequency and for a class of coefficients
that satisfy certain monotonicity conditions in one spatial direction, and are only assumed
to be bounded (i.e., $L^\infty$) in the other spatial directions. This class of coefficients
therefore includes coefficients modelling transmission by penetrable obstacles with a
(potentially large) number of layers (in 2-d) or fibres (in 3-d). Importantly, the
frequency-explicit bound holds uniformly for all coefficients in this class; this 
uniformity allows us to consider highly-oscillatory coefficients and study the limiting
behaviour when the period of oscillations goes to zero. In particular, we bound the $H^1$
error committed by the first-order bulk correction to the homogenized
transmission problem, with this bound explicit in both the period of oscillations of
the coefficients and the frequency of the Helmholtz equation; to our knowledge,
this is the first homogenization result for the Helmholtz equation that is explicit
in these two quantities and valid without the assumption that the frequency is small.
\end{abstract}



\section{Introduction}\label{sec:1}

\subsection{Definition of the Helmholtz transmission problem}\label{sec:notation}

For $R > 0$, let $B_R:=\{\bx \in \RRR^d \; | \; |\bx| < R\}$ and let
$\Gamma_R := \partial B_R = \{\bx \in \RRR^d \; | \; |\bx|=R\}$. 
Let $\DtN : H^{1/2}(\Gamma_R) \rightarrow H^{-1/2}(\Gamma_R)$ be the Dirichlet-to-Neumann map for the equation $\Delta u+k^2 u=0$
posed in the exterior of $B_R$ with the Sommerfeld radiation condition
\beq\label{eq:src}
\partial_r u(\bx) - \ri k u(\bx) = o \big(r^{-(d-1)/2}\big) \quad\tas r:=|\bx|\tendi,
\eeq
uniformly in $\hatx:= \bx/r$. The definition of $\DtN $
in terms of Hankel functions and polar coordinates ($d=2$)/spherical polar coordinates
($d=3$) is given in, e.g., 
\cite[Equations 3.7 and 3.10]{MeSa:10}. 

Let $\Sym$ denote the set of $d\times d$ real, symmetric matrices and let $\SPD$ denote the
set of $d\times d$ real, symmetric, positive-definite matrices. Given $\MM_1,\MM_2\in\SPD$
we write $\MM_1\preceq \MM_2$ to indicate inequality in the sense of quadratic forms, namely
$\MM_1\bv\cdot \conj\bv\le \MM_2\bv\cdot \conj\bv$ for all $\bv\in\Com^d$. For a non-negative
scalar $m$ and $\MM\in\SPD$ we write $m\preceq \MM$ if $m\MI\preceq\MM$, and $\MM\preceq m$
if $\MM\preceq m\MI$, where $\MI$ is the identity matrix.

We now give the weak form of the variable coefficient Helmholtz equation
\beq\label{eq:Helmholtz}
\nabla\cdot(\MA \nabla u) + k^2 n u =-f.
\eeq

\begin{definition}[Helmholtz transmission problem]\label{def:transmission}
Given $R > R_0 > 0$, $F\in (H^1(B_R))'$,
\bit
\item 
$n\in L^\infty(\Rea^d,\Rea)$ with $\supp(1-n)\subset B_{R_0}$ and such that
\beq\label{eq:nlimits}
0 < \varmin \leq \var(\bx)\leq\varmax < \infty\,\, \text{ for almost every } \bx \in \Rea^d,
\eeq
\item $\MA \in L^\infty (\Rea^d , \SPD)$ with $\supp(\MI-\MA)\subset B_{R_0}$ and such that
\beq\label{eq:Alimits}
0< A_{\min}\preceq \MA(\bx)\preceq A_{\max}<\infty\quad\text{ for almost every }\bx \in \Rea^d,
\eeq
\eit
and $k>0$, 
we say that $u\in H^1(B_R) $ satisfies the \emph{Helmholtz transmission problem} if 
\beq\label{eq:vf}
\sesqui(u,w)=F(w) \quad \tfa w\in H^1(B_R),
\eeq
where
\beq\label{eq:sesqui}
\sesqui(v,w):= \int_{B_R} 
\Big((\MA \gv)\cdot\overline{\nabla w}
 - k^2 n v\overline{w}\Big) - \big\langle \DtN v,w \big\rangle_{\Gamma_R},
\eeq
where $\langle\cdot,\cdot\rangle_{\Gamma_R}$ is the duality pairing on $\Gamma_R$
that is linear in the first argument and antilinear in the second.
\end{definition}

An important special case of the Helmholtz transmission problem of Definition
\ref{def:transmission} is transmission through a (not necessarily connected) Lipschitz
penetrable obstacle; in this case $\supp(\MI-\MA) = \supp(1-n)= \domainin$, where $\domainin \subset\Rea^d$ is a bounded, Lipschitz open 
set such that $\domainout:=\Rea^d\setminus\overline\domainin$ is connected.

\subsection{Main result I: existence, uniqueness, and a priori bound on particular Helmholtz transmission problem}\label{sec:PDE}

\subsubsection{Statement of the result}

For $\bx\in\Rea^d$, $\bx=(\bx',x_d)$ with $\bx'\in\Rea^{d-1}$ and $x_d\in \Rea$.

\begin{condition}\label{cond:1}
$\MA$ and $n$ satisfy the conditions in Definition \ref{def:transmission}. 
Additionally, $\MA$ is such that 
\beq\label{eq:monotoneA}
x_d \Big[\MA(\bx+h\be_d) -\MA(\bx)\Big] \preceq 0
\quad
\text{ for almost every } h \geq 0 \text{ and almost every } \bx \in \Rea^d.
\eeq
and $(\MA)_{d\ell}(\bx)=(\MA)_{\ell d}(\bx)=0$ for $\ell= 1,\ldots, d-1$ and for all $\bx\in \Rea^d$, 
and $n$ is such that
\beq\label{eq:monotonen}
x_d \Big[n(\bx+h\be_d) -n(\bx)\Big] \geq 0
\quad
\text{ for almost every } h \geq 0 \text{ and almost every } \bx \in \Rea^d.
\eeq
\end{condition}

\noi We make three remarks:
(i) The monotonicity conditions \eqref{eq:monotoneA} and \eqref{eq:monotonen}
(i.e.~that $\MA$ \emph{decreases} with distance from $x_d=0$ and $n$ \emph{increases})
along with the assumptions that $\supp (\MI-\MA)$ and $\supp (1-n)$ are compact imply
that $A_{\min}\geq 1$ and $n_{\max}\leq 1$.
(ii) Condition \ref{cond:1} only imposes constraints on the behaviour of $\MA$ and $n$
in the $x_d$ direction, while the behaviour in the $\bx'$ directions is unconstrained.
(iii) The origin of the condition that $(\MA)_{d\ell}(\bx)=0$ for $\ell=1,\ldots,d-1$, is
discussed in Remark \ref{rem:Adell}.

We are most interested in the following example of $\MA$ and $n$ satisfying Condition \ref{cond:1}.

\begin{example}[Example of $\MA$ and $n$ defined piecewise satisfying Condition \ref{cond:1}]
\label{ex:1}
\bit
\item 
$\domainin$ is as in \S\ref{sec:notation} and such that $\Gamma \cap \{\bx: x_d \geq 0\}$
is the graph of a $C^0$ function $f_+: \bx' \rightarrow \Rea$, and
$\Gamma \cap \{\bx: x_d \leq 0\}$ is the graph of a $C^0$ function $f_-: \bx' \rightarrow \Rea$.
\item $\MA$ is as in Definition \ref{def:transmission} 
and furthermore
\beqs
\MA(\bx)
=
\boldsymbol{1}_{\domainin}(\bx)\,\widetilde{\MA}(\bx) + \boldsymbol{1}_{\domainout}(\bx)\MI,
\eeqs
where $\boldsymbol{1}_D$ denotes the indicator function of a set $D$ and
$\widetilde{\MA} \in L^\infty(\domainin,\SPD)$ satisfies \eqref{eq:monotoneA}
(with $\MA$ replaced by $\widetilde{\MA}$), $(\widetilde{\MA})_{d\ell}(\bx)=0$ for
$\ell= 1,\ldots, d-1$ and for all $\bx\in \domainin$, and $\widetilde{A}_{\min}\geq 1$.
\item $n$ is as in Definition \ref{def:transmission} and furthermore
\beqs
n(\bx) =  \boldsymbol{1}_{\domainin}(\bx) \,\widetilde{n}(\bx) + \boldsymbol{1}_{\domainout}(\bx),
\eeqs
where $\widetilde{n} \in L^\infty(\domainin,\Rea)$ satisfies \eqref{eq:monotoneA}
(with $n$ replaced by $\widetilde{n}$), and $\widetilde{n}_{\max}\leq 1$.
\eit
\end{example}

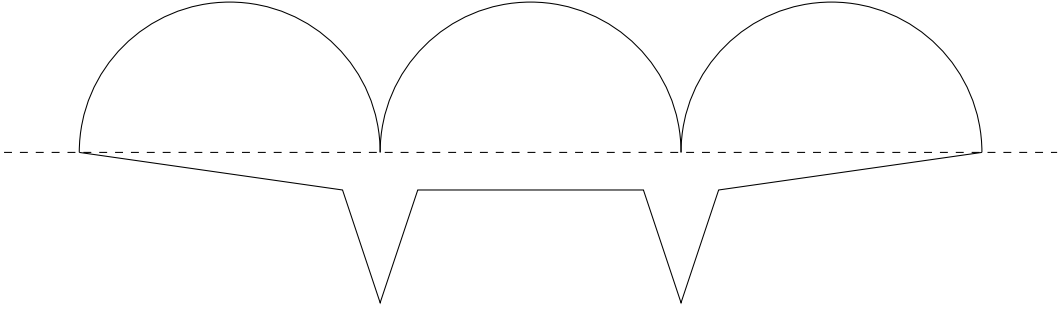
\begin{figure}
\begin{tikzpicture}[scale=2]
\draw ( 0, 0) arc[start angle=0,end angle=180,radius=1cm];
\draw ( 2, 0) arc[start angle=0,end angle=180,radius=1cm];
\draw ( 4, 0) arc[start angle=0,end angle=180,radius=1cm];


\draw[dashed] (-2.5,0) -- (4.5,0);

\draw (-2,0) -- (-0.25,-0.25) -- (0,-1) -- (0.25,-0.25) -- (1.75,-0.25) -- (2,-1) -- (2.25,-0.25) -- (4,0);
\end{tikzpicture}
\caption{An example of a domain $\Din$ satisfying the first bullet point in Example \ref{ex:1}}
\label{figure_domain}
\end{figure}

In Example \ref{ex:1}, if $\Gamma$ is Lipschitz, then the first bullet point can be replaced by
\beq\label{eq:normal}
x_d \be_d \cdot \bn(\bx) \geq 0 \quad\tfa \bx \in \Gamma \tst \bn(\bx) \text{ is defined},
\eeq
and the boundary value problem is then transmission through a Lipschitz penetrable obstacle.

The following theorem bounds the norm of the solution operator of the Helmholtz transmission
problem; i.e., the operator norm of the (linear) mapping $(H^1(B_R))' \ni F \to u \in H^1(B_R)$,
given by
\begin{equation}\label{eq_Csol_def}
\Csol(\MA, n, k, R, R_0) 
\eq \sup_{\substack{F \in (H^{1}_k(B_R))', \,\|F\|_{(H^{1}_k(B_R))'} = 1}} \|u\|_{H^1_k(B_R)},
\end{equation}
where, for a bounded open set $D$, 
\begin{equation} \label{eq:1knorm}
\|u\|^2_{H^1_k(D)} \eq \int_{D} |\grad u|^2+k^2|u|^2 \quad \tand \quad
\|F\|_{(H^{1}_k(D))'}
\eq
\sup_{\substack{v \in H^1(D),\, \|v\|_{H^1_k(D)} = 1}} \big| F(v) \big|.
\end{equation}
Recall that $\Csol$ must grow at least linearly with $k$ as $k\tendi$; indeed, if $\MA=\MI$ and $n=1$, then this follows by considering, e.g., $u(x)= \re^{\ri k x_1} \chi(|x|/R)$ for $\chi \in C^\infty$ supported in $[0,1)$.

\begin{theorem}[Well-posedness under Condition \ref{cond:1}]
\label{main_theorem_stability}
There exists $\Cwave,\Cone>0$ such that the following holds. Given $R>R_0>0$ and $\MA$ and $n$ satisfying Condition \ref{cond:1}, if $kR_0\geq \Cwave$ then the Helmholtz transmission problem of Definition \ref{def:transmission} exists, is unique, and 
\beq\label{main_eq_stability}
\Csol(\MA, n, k, R, R_0)
\leq
\frac{1}{n_{\min}}\Big(1 + 2 (kR)^2(1+ kR_0) \Cone\Big)
=:
\Clayer.
\eeq
\end{theorem}

The key point is that, other than $n_{\min}$, all the terms on the right-hand side of
\eqref{main_eq_stability} are independent of $\MA$ and $n$. Explicit expressions for $\Cwave$
and $\Cone$ are available; see Remark \ref{rem:explicit} below. The notation $\Cwave$ is chosen
because this constant determines the minimum number of wavelengths in $B_{R_0}$ for the results of the theorem to hold.

Theorem \ref{main_theorem_stability} is obtained as a corollary of the following result. 

\begin{theorem}[A priori bound under Condition \ref{cond:1} for $L^2$ data]
\label{main_theorem_stability2}
Under the assumptions of Theorem \ref{main_theorem_stability}, if $F(v) = \int_{B_R} f\,\overline{v}$ and $kR_0 \geq \Cwave$, then 
\begin{equation}
\label{eq:resol}
\N{u}_{H^1_k(B_R)} \leq \Cone (1+ kR_0)(kR) R \N{f}_{L^2(B_R)}.
\eeq
\end{theorem}

\begin{remark}[Plane-wave scattering]
A common Helmholtz problem in applications is scattering of an incident field $u_{\rm inc}$
(such as a plane wave or a point source), creating a scattered field $u_{\rm sca}$, satisfying
the Sommerfeld radiation condition, with the total field $u_{\rm inc} + u_{\rm sca}$ satisfying
the Helmholtz equation with zero right-hand side (see, e.g., \cite[Definition 2.4]{MoSp:19}). 
This problem fits into the framework of Definition \ref{def:transmission} with
$u = u_{\rm sca} + \chi u_{\rm inc}$, where $\chi$ is a smooth cutoff function such that
$\chi = 1$ on $\Din$ and $\chi = 0$ in a neighborhood of $\partial B_R$, and
$f \eq (\Delta \chi) u_{\rm inc} + 2 \grad u_{\rm inc} \cdot \grad \chi$.
\end{remark}

\subsubsection{%
Discussion of the novelty of the well-posedness result in Theorem \ref{main_theorem_stability}%
}
\label{sec:UCP}

The well-posedness result of Theorem \ref{main_theorem_stability}
(i.e., well-posedness for coefficients satisfying Condition \ref{cond:1}) 
is new for $d=3$, but not for $d=2$.  

Recall from Fredholm theory that well-posedness of the Helmholtz equation follows from proving uniqueness. Two established ways of proving uniqueness are to (a) prove a unique continuation principle (UCP) or (b) prove an a priori bound on the solution in terms of the data.

When $d=2$, a UCP for the
Helmholtz equation holds with $\MA \in L^\infty$ and $n \in L^p$ with $p>1$ \cite{Al:12}, thus covering coefficients satisfying Condition \ref{cond:1}.
In constrast, when $d=3$ the UCP holds when $\MA$ is piecewise Lipschitz \cite{BaCaTs:12},
\cite[Proposition 2.11]{LiRoXi:19} (by using the UCP in the Lipschitz case \cite{Pr:60, GaLi:87, Ka:88})
and $n \in L^{3/2}$ \cite{JeKe:85,  Wo:92}. However, examples of $\MA \in C^{0,\alpha}$ for
all $\alpha<1$ for which the solution of the transmission problem is not unique at a particular
$k>0$ (and thus the UCP fails) are given in 
\cite{Pl:63,Mi:74,Fi:01}.

The only other Helmholtz
well-posedness result we are aware of that is valid for $\MA$ with no smoothness assumptions
other than being $L^\infty$ is in \cite[Theorem 2.7]{GrPeSp:19}. There, well-posedness is proved (via proving an a priori bound -- similar to in the present paper) for $\MA \in L^\infty$ satisfying a radial-monotonicity condition (as opposed to the monotonicity
in a single coordinate direction in \eqref{eq:monotoneA}). 

\subsubsection{Discussion of the novel features of the bound in Theorem \ref{main_theorem_stability}}

The bound \eqref{eq:resol} is novel in two ways.

(i) A standard way of obtaining frequency-explicit bounds on solutions of the high-frequency
Helmholtz equation is to consider the billiard flow defined by the (semiclassical) principal
symbol of the Helmholtz equation and use associated results on propagation of singularities
under this flow \cite{MeSj:78}, \cite{MeSj:82}, \cite[Chapter 24]{Ho:85}. However, this flow
not well-defined for the class of $\MA$ and $n$ in Condition \ref{cond:1}.

(ii) 
The constant in the bound \eqref{main_eq_stability} only depends on the coefficients $\MA$ and $n$ via $n_{\min}$ (as noted below Theorem \ref{main_theorem_stability}), and the constant in the bound \eqref{eq:resol} is independent of $\MA$ and $n$.
\footnote{We note that a bound explicit in the coefficients for nontrapping $A$ and $n$ was recently proved in \cite{GaSpWu:20} using the techniques discussed in Point (i).}

Regarding (i): the flow is defined as the solution $(\bx(t),\bxi(t))$ of the Hamiltonian system
\beq\label{eq:Hamilton}
\dot{x_i}(t) = \partial_{\xi_i}p\big(\bx(t), \bxi(t) \big),
\qquad
\dot{\xi_i}(t) = -\partial_{x_i}p\big(\bx(t), \bxi(t) \big),
\eeq
where the Hamiltonian equals the semiclassical principal symbol of the Helmholtz equation, namely
\beqs
p(\bx,\bxi):= \sum_{i=1}^d\sum_{j=1}^{d} \MA_{ij}(\bx)\xi_i \xi_j - n(\bx);
\eeqs
see, e.g., \cite[Page 281]{Zw:12}.
By the Picard--Lindel\"of theorem, the flow exists and is unique if $\MA,n$ are both $C^{1,1}$, since the coefficients of the ODE system \eqref{eq:Hamilton} are then Lipschitz. 
By the well-known examples of ODE non-uniqueness with non-Lipschitz coefficients, if $\MA$ and $n$ are rougher than $C^{1,1}$, then the flow is not guaranteed to be well-defined.
The flow can be defined piecewise, with results about the behaviour of singularities hitting the interface 
given in \cite[Chapter 11]{MeTa}, \cite{Mi:00}, and $k$-explicit bounds on the solution of the Helmholtz transmission problem when $A$ and/or $n$ are discontinuous on a $C^\infty$ strictly-convex interface given in \cite{CaPoVo:99}, \cite{PoVo:99}, \cite{PoVo:99a}. However, it is not possible to defined the flow for the range of coefficients covered by Condition \ref{cond:1}, which need only be $L^\infty$ in the $\bx'$ directions.

Regarding (ii):~this explicitness in the coefficients is a consequence of the method we use to prove \eqref{eq:resol}, namely a Morawetz-type identity -- see the discussion in \S\ref{sec:idea} and the references therein.
The fact that the constant in the bound \eqref{eq:resol} is independent of $\MA$ and $n$ is the basis of the homogenisation results in \S\ref{sec_homo_results}.


\subsubsection{Discussion about the $k$-dependence in the bound \eqref{eq:resol}}
\label{sec:kexplicit}

The key points are the following.
\bit
\item
There exist (fixed) $\MA$ and $n$ such that $\Csol$ grows super-algebraically in $k$ as $k\tendi$
due to trapped rays \cite{Ra:71, PoVo:99, St:99, CaPo:02, Ca:12, CaLePa:12}.
The polynomial bound on $\Csol$ in \eqref{main_eq_stability} shows that this behaviour is ruled out by Condition \ref{cond:1}, which is consistent
with the physical understanding of what causes rays to be trapped for the transmission problem (however, we do not know whether the $k$-dependence in the polynomial bound \eqref{main_eq_stability} is sharp or not).%
\item
The recent results of \cite{SaTo:18} show that, in 1-d, $\Csol$ can grow exponentiallly
through a sequence $n_j, k_j$, with $k_j\tendi$ as $j\tendi$,
through a mechanism \emph{not} involving the trapping of rays.
These examples in \cite{SaTo:18} involve piecewise
constant $n$ (i.e.,~a 1-d layered medium), and the key point is that to get the exponential
growth the width of the pieces need to be tied to $k$ in a delicate way.
This behaviour is ruled out for the layered obstacles included in Condition \ref{cond:1} via 
the polynomial bound on $\Csol$ in \eqref{main_eq_stability}, which is uniform in $\MA$ and $n$ satisfying
Condition \ref{cond:1}.
\eit

\noi We now give more detail on these two points.

\paragraph{Growth of $\Csol$ via trapped rays}
For the Helmholtz transmission problem of Definition \ref{def:transmission}, super-algebraic growth of $\Csol$ through a sequence of $k_j$s has been proved  in the following two cases.
\bit
\item[(i)]
$\MA$ and $n$ are 
such that $\supp(\MI-\MA) = \supp(1-n)= \domainin$
with $\Din$ a $C^\infty$
convex domain with strictly positive curvature, and the jumps of $\MA$ and $n$
have a certain combination of signs 
(e.g., moving outwards, $\MA$ jumps down with $n$ fixed, or $n$ jumps up with $\MA$ fixed) 
so that rays can be totally internally reflected when hitting $\Gamma$ from inside $\Din$
\cite{PoVo:99} (see also \cite{Ca:12}, \cite{CaLePa:12},
\cite[Chapter 5]{AlCa:18} for similar results in the specific case when the obstacle is a ball).
The solutions corresponding to the trapped rays are known as
``whispering-gallery modes''; see, e.g., \cite{BaBu:91} and the references therein.
\item[(ii)]
$\MA=\MI$ and $n$ is $C^\infty$ and spherically symmetry with 
$2n(r_1) + (\partial n/\partial r)(r_1)<0$ for some $r_1$ \cite{Ra:71} (see also \cite[Theorem 7.7(ii)]{GrPeSp:19})
\footnote{Strictly speaking, \cite{Ra:71} proves the existence of
a sequence of resonances exponentially close to the real axis, but then the
``resonances to quasimodes'' result of \cite{St:00} implies super-algebraic
growth through a sequence of real $k_j$s.}; this condition on $n$ causes great circles on $r=r_1$ to be stable trapped rays; see \cite[Page 572]{Ra:71}.
\eit
Neither of the situations in (i) or (ii) are allowed under Condition \ref{cond:1}.
Indeed, Condition \ref{cond:1} implies that $A_{\min}\geq 1$ and $n_{\max}\leq 1$
(as noted just below the condition), and these prevent $\MA$ and $n$ having the ``bad''
jumps in (i). Furthermore, the monotonicity condition on $n$ \eqref{eq:monotonen} prevents $n$ from satisfying the condition in (ii).

\paragraph{Growth of $\Csol$ through a mechanism \emph{not} involving the trapping of rays}
For \eqref{eq:Helmholtz} posed on a 1-d interval with either Dirichlet or impedance boundary
conditions at either end, and at least one of the ends having impedance boundary conditions,
bounds on the Helmholtz solution where $\MA$ and $n$ have finite number of jumps are given in \cite{Ch:16},
\cite{GrSa:20}. These results have $\Csol \leq \exp \big(C(\MA,n)\big)$
(see \cite[Theorem 5.4 and 5.10]{GrSa:20}), but $\Csol \tendi$ if the number of jumps goes to infinity.
The fact that $\Csol$ is independent of $k$ is consistent with the fact that,
for fixed $\MA$ and $n$ geometric-optic rays are not trapped, since although $\MA$ and $n$ can jump
across interfaces, the 1-d nature of the problem means that no rays can get trapped moving
tangent to the interface (like in Point (i) above).

The paper \cite{SaTo:18} considers the 1-d case with $\MA=1$ and $n$ variable and proves 
\bit
\item there exist sequences $n_j, k_j$ (with $k_j\tendi$ as $j\tendi$) such that $\Csol$ grows exponentially as $j\tendi$ \cite[Remark 14]{SaTo:18},
\item if the number of jumps $\lesssim k$ then $\Csol \leq \exp \big(C(n) k\big)$, but $C(n)$ doesn't blow up with number of jumps
\cite[Prop.~18]{SaTo:18}, and
\item for a class of $n$ that oscillate between two values, with an arbitrarily-large number of jumps,
$\Csol \leq \exp \big(C(n) k\big)$ and  $C(n)$ doesn't blow up with number of jumps
\cite[Theorem 22]{SaTo:18}.
\eit
The paper \cite{SaTo:21} generates the results of \cite{SaTo:18} to radially-symmetric $n$ in
3-d (see \cite[Theorem 3.7]{SaTo:21} for the upper bounds and \cite[Lemma 4.3]{SaTo:21} for the
examples of exponential blow-up).

\subsubsection{Summary of the ideas behind the proof of Theorem \ref{main_theorem_stability}}
\label{sec:idea}

Theorem \ref{main_theorem_stability} is proved by the following five steps.

(i)
Observing that Theorem \ref{main_theorem_stability} follows from Theorem
\ref{main_theorem_stability2}  since $\sesqui(\cdot,\cdot)$ satisfies a G\aa rding inequality;
see Lemma \ref{lem:H1} below.

(ii)
Observing that, by Fredholm theory, to prove Theorem \ref{main_theorem_stability2} it is
sufficient to prove the bound \eqref{eq:resol} under the assumption of existence
(see Lemma \ref{lem:Fred}).

(iii) Approximating the coefficients $\MA$ and $n$ in Condition \ref{cond:1} by a sequence of
smooth coefficients $(\MA^\ell)_{\ell=0}^\infty$ and $(n^\ell)_{\ell=0}^\infty$ satisfying
Condition \ref{cond:smooth} below (see \S\ref{sec:approx}); this condition 
is similar to Condition \ref{cond:1}, but involves derivatives of the smooth coefficients $\MA^\ell$ and $n^\ell$.

(iv) Proving the bound \eqref{eq:resol} for these smooth coefficients, importantly with the
constants $\Cwave$ and $\Cone$ independent of $\ell$, using a Morawetz-type identity
(\eqref{eq:morid1} below) and 
making key use of the particular multiplier, vector field, and associated arguments
introduced recently for the constant-coefficient Helmholtz equation in  \cite{ChSpGiSm:20}; see the overview discussion in \S\ref{sec:ideasmooth},
where we recall that in using Morawetz-type identities, the main work is in finding coefficients of the multiplier so that a) the volume terms are sign-definite and control the desired norm of the solution (here the $H^1$ norm) and b) the boundary terms have the correct sign.
The key point for the present paper is that the vector field
(defined by \eqref{eq:Z}) equals $\be_d x_d$ in a neighbourhood of $\supp(\MI-\MA)$ and
$\supp(1-n)$, where $\be_d$ denotes the unit vector in the $x_d$ direction; this vector field is
constant in the $\bx'$ directions, and so ``does not see'' the behaviour of $\MA$ and $n$ in these
directions, hence why no constraint is imposed on this behaviour.
We highlight that our use of a Morawetz-type identity to prove an a priori bound on solutions of the Helmholtz equation satisfying the Sommerfeld radiation condition 
follows, e.g., \cite{MoLu:68, Bl:73, Mo:75, MoRaSt:77, BlKa:77, PeVe:99, NgVo:12, NgVo:12b, NgNg:15, GrPeSp:19, MoSp:19,ChSpGiSm:20} (see also the bibliographic comments in Remark \ref{rem:biblio}), although all these papers apart from \cite{Mo:75, BlKa:77, MoRaSt:77,ChSpGiSm:20} prove their bounds using radial vector fields.

(v) Using approximation arguments from \cite{GrPeSp:19} to prove the bound \eqref{eq:resol}
for the original $\MA$ and $n$ (see Lemma \ref{lem:basicapprox1}); these approximation
arguments were inspired by similar arguments in \cite{Th:06} in the setting of rough-surface
scattering  (with this thesis recently made available as \cite{Ba:19}). 

\subsection{Main result II: homogenization}
\label{sec_homo_results}

An attractive feature of Theorem \ref{main_theorem_stability} is that it does not constrain
the variations of $\MA$ or $n$ along the $\bx'$ hyperplane. In particular, highly-oscillatory
coefficients are allowed, and the results in this section concern the limiting behaviour when
the period of oscillations goes to zero. Theorem \ref{main_theorem_stability} would also allow
one to do homogenization under assumptions other than periodicity (see, e.g.,
\cite[Chapter 13]{cioranescu_donato_1999a}), but for simplicity we only consider
a periodic setting here.

\subsubsection{Statement of the result}

With $\varepsilon > 0$ a (small) oscillation period, we define oscillatory coefficients
$n_\eps$ and $\MA_\eps$ created by repeating functions $\hn$ and $\hMA$ on a grid of
size $\eps$ inside a bounded Lipschitz domain $\Din$.

\begin{definition}[Oscillatory coefficients in $\Din$]
\label{definition_oscillatory_coefficients}
Let $Y \eq (0,1)^d$ (i.e., the unit cube in $\mathbb R^d$).
If $\hn \in L^\infty(Y,\mathbb R)$ and $\hMA \in L^\infty(Y,\SPD)$, let
\begin{equation*}
\hn^\eps(\bx) \eq \hn \left (\left \{ \frac{\bx}{\eps} \right \}\right )
\quad\tand\quad
\hMA^\eps(\bx) \eq \hMA \left (\left \{ \frac{\bx}{\eps} \right \}\right )
\end{equation*}
for all $\bx \in \mathbb R^d$, where $\{\bx/\eps\}_j = (x_j/\eps) \mod  1$.
Then, given a bounded Lipschitz domain $\Din$ satisfying \eqref{eq:normal}, let
\begin{equation}\label{eq_n_eps_A_eps}
n_\eps \eq \mathbf 1_{\Din} \hn^\eps + \mathbf 1_{\Dout}
\quad\tand\quad
\MA_\eps \eq \mathbf 1_{\Din} \hMA^\eps + \mathbf 1_{\Dout} \MI.
\end{equation}
\end{definition}

We now impose constraints on $\hn$ and $\hMA$ so that $n_\eps$ and $\MA_\eps$ satisfy
Condition \ref{cond:1} (and so Theorem \ref{main_theorem_stability} applies) -- these
are the first three points in Condition \ref{condition_periodic_patterns} below.
We also impose additional constraints on $\hn$ and $\hMA$ so that the
correctors in the homogenization argument (see \S\ref{section_homogenization} below)
have sufficient regularity -- these constraints are in the final point in Condition
\ref{condition_periodic_patterns}.


\begin{figure}
\begin{minipage}{.45\linewidth}
\begin{tikzpicture}[scale=3]
\path[draw,clip,rotate=90] plot[smooth cycle]
coordinates{(0, 1) ( 1.2, 0.8) ( .2, .25) ( .2,-.25) ( 1.2,-0.8)
            (0,-1) (-1.2,-0.8) (-.2,-.25) (-.2, .25) (-1.2, 0.8)};

\foreach \i in {-80,...,80}
{
	\draw[solid] (0.10*\i,-2.0) -- (0.10*\i,2.0);
	\draw[solid] (0.03+0.10*\i,-2.0) -- (0.03+0.10*\i,2.0);
}
\end{tikzpicture}
\end{minipage}
\begin{minipage}{.45\linewidth}
\begin{tikzpicture}[scale=5]
\path[draw,clip,rotate=25] (.5,.5) ellipse (0.5 and 0.8);

\foreach \i in {-6,...,16}
\foreach \j in {-2,...,28}
{
   \draw[solid] (0.05*\i,0.05*\j) circle (.015);
}
\end{tikzpicture}
\end{minipage}
\caption{Sketches of admissible periodic patterns, with the black lines representing discontinuities in the coefficients. Left-panel: if the coefficients in this 2-d domain are constant in the vertical direction, then they satisfy (a)
in Condition \ref{condition_periodic_patterns}.
Right-panel: horizontal cross section of a 3-d domain satisfying (b) in Condition \ref{condition_periodic_patterns}.}
\label{figure_periodic_patterns}
\end{figure}
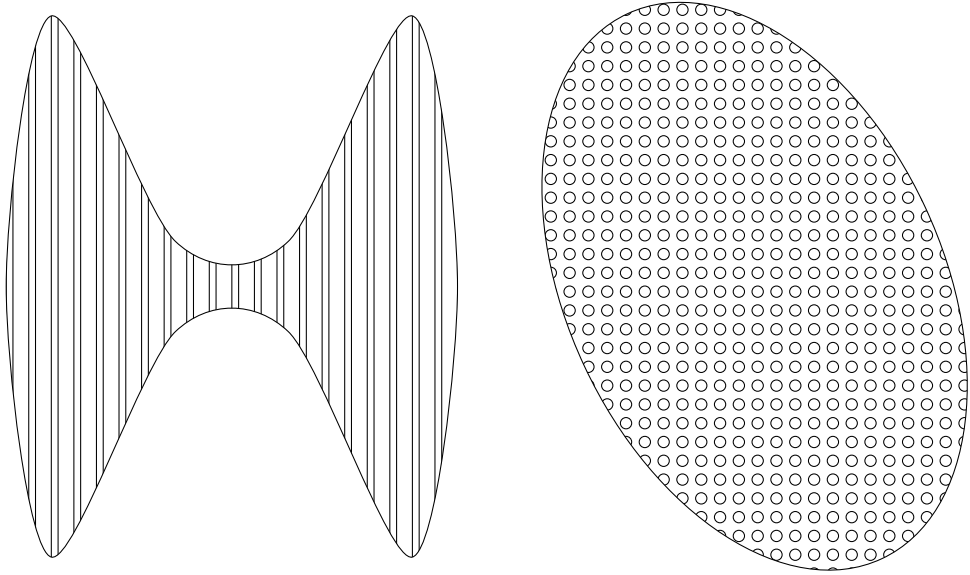

\begin{condition}[Admissible periodic patterns]
\label{condition_periodic_patterns}
$\hn \in L^\infty(Y)$ and $\hMA \in L^\infty(Y,\SPD)$ are
\emph{admissible periodic patterns} if
\begin{itemize}
\item $\hn$ and $\hMA$ are independent of $y_d$
\item $\widehat{A}_{d\ell}= 0$ for $\ell=1,\ldots,d-1$,
\item $\hMA_{\rm min} \geq 1$ and $\hn_{\rm max} \leq 1$,
\item either
\begin{itemize}
\item[(a)] $\hn$ and $\hMA$ only depend on $y_1$ (or,  when $d=3$, only on $y_2$)
and $\hMA$ is piecewise $C^{0,1}$.
\item[(b)] the periodic extension of $\hMA$ is piecewise $C^{1,1}$ and that of $\hn$
is piecewise $H^1$ on a partition of $\Rea^d$ that consists of subdomains with
$C^{2,1}$ boundaries.
\end{itemize}
\end{itemize}
\end{condition}

The class of coefficients covered by Condition \ref{condition_periodic_patterns}
corresponds to obstacles made of thin layers in 2D, and of thin layers or fibers
in 3D.

\bre[Physical relevance of finely-layered and finely-fibred materials]\label{rem:physical}
Many key applications motivating homogenized theory are based on ``composite'' or ``meta''
materials. Such materials are based on blending two (or more) materials in a periodic
manner at a fine scale, and thus are modelled by piecewise-constant periodic
patterns, as covered by Condition \ref{condition_periodic_patterns}. Particular instances
of layered and fibred media arise when considering the propagation of elastic waves
through rocks \cite{rocks} and electromagnetic waves in thin-film coatings \cite{layers} or optical fibres \cite{fibres}.
\ere

\begin{corollary}
Let $\Cwave$ be as in Theorem \ref{main_theorem_stability}. If $\hn$ and $\hMA$ satisfy
Condition \ref{condition_periodic_patterns} and $kR_0\geq \Cwave$, then, for all $\eps>0$,
the solution of the variational problem \eqref{eq:vf} with coefficients $n_\eps$ and $\MA_\eps$
exists and is unique, and, with $\Clayershort$ defined by \eqref{main_eq_stability},
\begin{equation}
\label{main_eq_estimate_Cse}
\Cse \leq \Clayerh.
\end{equation}
\end{corollary}

Crucially, the bound \eqref{main_eq_estimate_Cse} is uniform in $\eps$, and this uniformity
allows us to extract a convergent subsequence $u_\eps$. In fact, by classical homogenization
theory \cite{BeLiPa:78, cioranescu_donato_1999a}, in the limit $\eps \to 0$, the whole sequence
$u_\eps$ to $u_0$ weakly in $H^1(B_R)$ and strongly in $L^2(B_R)$, where $u_0$ is the
solution to the transmission problem of Definition \ref{def:transmission} with constant
``homogenized'' coefficients $\MAH$ and $\nH$. These homogenized coefficients are obtained
from $\hMA$ and $\hn$ by averaging formulas given by \eqref{eq_def_AH} and \eqref{eq_def_nH}
below; note that these are the same formulas for the coercive case
(see, e.g., \cite[Chapter 1]{BeLiPa:78}, \cite[Chapter 6]{cioranescu_donato_1999a}).

Recall that a domain $\domain \subset \RRR^d$ is \emph{star-shaped
with respect to the point $\bx_0 \in \RRR^d$} if, whenever $\bx \in \domain$,
the line segment $[\bx_0,\bx]\in \domain$.

\begin{lemma}[Well-posedness of the homogenized problem]
\label{main_lemma_homogenization}
Let $\Cwave$ be as in Theorem \ref{main_theorem_stability2}.
Given $\hn$ and $\hMA$ satisfying Condition \ref{condition_periodic_patterns},
let $\nH$ and $\MAH$ be defined by \eqref{eq_def_nH} and \eqref{eq_def_AH}, respectively,
and let $u_0 \in H^1(B_R)$ be the solution of the variational problem \eqref{eq:vf} with
coefficients $\nH$ and $\MAH$ (we then say that $u_0$ is the solution of the homogenized problem).

Then, if $kR_0\geq \Cwave$, $u_0$ exists and is unique, and 
\begin{equation}\label{eq_CsolH}
\CsH \leq \Clayerh.
\end{equation}

Furthemore, if $\Din$ is additionally star-shaped with respect to a point,
then, for all $k>0$, $u_0$ exists and is unique, and 
\begin{equation}\label{eq_Cstar}
\CsH \leq 
\frac{1}{\hn_{\min}}
\left(
1
+
kR\frac{4}{\sqrt{\hn_{\min}}}
\sqrt{ 1 + \frac{1}{\hn_{\min}}\left(1 + \frac{d-1}{2kR}\right)^2}
\right).
\end{equation}
\end{lemma}

\begin{theorem}[$k$- and $\eps$-explicit homogenization error estimate in $H^1_k(B_R)$]
\label{main_theorem_homogenization}
Let $\Cwave$ and $\Clayershort$ be as in Theorem \ref{main_theorem_stability}.
Let $\hn$ and $\hMA$ satisfy Condition \ref{condition_periodic_patterns}
and assume that $\Gamma$ is $C^{1,1}$. Given $f\in L^2(B_R)$, let $u_\eps$
be the solution of the variational problem \eqref{eq:vf} with coefficients
$\MA_\eps$ and $n_\eps$ and $F(v) := \int_{B_R}f\overline{v}$. Let $u_0$ be
the solution of the variational problem \eqref{eq:vf} with this $F$ and 
coefficients $\MAH$ and $\nH$, and let $u_1^\eps$ be defined in terms
of $u_0$ by \eqref{eq_def_first_order_corrector} and \eqref{eq_eps_diff} below.

If $kR_0\geq \Cwave$ then $u_\eps, u_0$, and $u_1^\eps$ exist and are unique,
and there exist $\Ctwonew, \Cthreenew>0$ (depending only on $\hMA,\hn$, and $\Din$)
such that
\begin{align}
\nonumber
&
k \|u_\eps-u_0- \eps u_1^\eps\|_{H^1_k(\Din)} + k \|u_\eps-u_0\|_{H^1_k(\Dout)}
\\
&
\hspace{2cm}
\leq
\Ctwonew\,
\Clayerh\,
\Big((k\eps)^{1/2} + k\eps \Big)\Big( \N{u_0}_{H^2(B_R)} + k \N{u_0}_{H^1_k(B_R)}\Big) 
\label{eq_homo_thm_1}
\\
&\hspace{2cm}
\leq
\Cthreenew\,
\Clayerh\,
\Big((k\eps)^{1/2} + k\eps \Big)
\CsH
\N{f}_{L^2(B_R)}.
\label{eq_homo_thm_2}
\end{align}
\end{theorem}

Since
\beqs
\N{u_\eps}_{H^1_k(B_R)}\leq \N{u_\eps-u_0}_{H^1_k(B_R)} + \N{u_0}_{H^1_k(B_R)},
\eeqs
the bound on $u_\eps-u_0$ from Theorem \ref{main_theorem_homogenization}
and the definition of $\CsHshort$ then imply the following bound on $\Cse$.

\begin{corollary}[Improved bound on $C_{{\rm sol}, \eps}$ for $k\eps$ sufficiently small (depending on $k$)]
\label{corollary_stability}
Let $\Cwave$ and $\Clayershort$ be as in Theorem \ref{main_theorem_stability}.
Let $\hn$ and $\hMA$ satisfy Condition \ref{condition_periodic_patterns}
and assume that $\Gamma$ is $C^{1,1}$. 
Then there exists $\Cfivenew$ (depending only on $\hMA,\hn$, and $\Din$) such that 
if $kR_0\geq \Cwave$ then
\begin{equation}\label{eq_newbound1}
\Cse
\leq
\frac{\Cfivenew}{\hn_{\min}} \bigg[ 1 +\left ( (k\eps)^{1/2} + k\eps \right )\Clayerh \bigg]
\CsH.
\end{equation}
\end{corollary}

Corollary \ref{corollary_stability} shows that if $k\eps$ is sufficiently small (depending on $k$),
then the problem with coefficients $\MA_\eps$ and $n_\eps$ inherits the behaviour of the solution
operator for the homogenized problem.

\subsubsection{Discussion of the novelty and context of Theorem \ref{main_theorem_homogenization}}
\label{sec:homo_discuss}

The first thing to highlight about Theorem \ref{main_theorem_homogenization} is that it is
a $k$-explicit analogue of the $H^1$ error bound in \cite[Theorem 2]{CaGuMo:16}.
More precisely, \cite[Theorem 2]{CaGuMo:16} proves the bound 
\beqs
\|u_\eps-u_0- \eps u_1^\eps\|_{H^1_k(\Din)}
+
\|u_\eps-u_0\|_{H^1_k(\Dout)}
\leq
C(k) \eps^{1/2} \|u_0\|_{H^2(B_R)},
\eeqs
for some unspecified $C(k)$, with this bound valid for sufficiently small $kR$ and then arbitrarily
small $\eps/R$. In contrast, in Theorem \ref{main_theorem_homogenization}, $kR$ can be arbitrarily large, and $\eps/R$ can be tied to $kR$. 
We recall that (i) a main novelty of \cite[Theorem 2]{CaGuMo:16} is that it uses boundary correctors for homogenization of a transmission problem (as opposed to using boundary correctors for problems on bounded domains with Dirichlet or Neumann boundary conditions; see, e.g., \cite{BeLiPa:78, MoVo:97, MoVo:97a, Gr:04, ZhPa:05, Gr:06, OnVe:07, JiKoOl:12, KeLiSh:13, Su:13, KeLiSh:14,SeDeSc:18}), and (ii) 
\cite[Theorem 2]{CaGuMo:16} also proves the $L^2$ error estimate $\|u_\eps -u_0\|_{L^2(B_R)}\leq C(k) \eps\|u_0\|_{H^2(B_R)}$, but for brevity we have not pursued here obtaining a $k$-explicit analogue of this result.

The main obstacle for obtaining homogenization results about the Helmholtz equation for
large $kR$ is that the first step of the homogenization procedure relies on a uniform-in-$\eps$
bound on the $H^1$ norm of $u_\eps$, which allows one to extract a weakly-converging subsequence.
Such estimate is immediate for the coercive PDE $-\nabla \cdot(\MA_\eps\nabla u) =f$, since the
coercivity constant only depends on $A_{\min}$, which is independent of $\eps$.
In Appendix \ref{app:small_k}, we recall how the $b(\cdot,\cdot)$ defined by \eqref{eq:sesqui}
is coercive if $kR$ is sufficiently small, with both the coercivity constant (and hence $\Csol$)
and the constant for ``sufficiently small'' independent of $\eps$ and depending on $A_{\min}$
and $n_{\max}$; see Lemma \ref{lem:coercivity} below. Therefore, homogenization results for
the Helmholtz equation can be obtained for $kR$ sufficiently small, and this is implicitly
what \cite{CaGuMo:16} do. 
\footnote{The paper \cite{CaGuMo:16} cites \cite[Theorem 5.26]{CaCo:14} as a reference for $\Csol$
being independent of $\eps$ and depending on $\MA_\eps$ and $n_\eps$ only through
$A_{\max}, A_{\min}, n_{\max},$ and $n_{\min}$ (see \cite[Bottom of Page 2539 and
top of Page 2540]{CaGuMo:16}). Actually, \cite[Theorem 5.26]{CaCo:14} records how Fredholm
theory proves an a priori bound on the solution of the transmission problem; however this
theory (relying on abstract functional analysis arguments) is unable to give the constant
in this bound. The arguments in \cite{CaGuMo:16} hold, however, with the results of
Appendix \ref{app:small_k} replacing reference to  \cite[Theorem 5.26]{CaCo:14}, and under
the explicit assumption that $kR$ is sufficiently small.}
In contrast, if $kR$ is sufficiently large, then  $b(\cdot,\cdot)$ is not coercive (see, e.g., \cite[\S6.1.6]{Sp:15}); furthermore, at least when $\MA_\eps$ and $n_\eps$ are smooth, \cite{GaSpWu:20} shows that $\Csol$ depends on global properties of $\MA_\eps$ and $n_\eps$ (more precisely, $\Csol$ is proportional to the length of the longest ray in $B_R$). 
Therefore, obtaining a uniform-in-$\eps$ bound on $\Csol$ for arbitrary $\MA_\eps$ and $n_\eps$ and arbitrary $kR$ is very challenging.
The present paper bypasses this fundamental problem by considering the restricted class of $\MA_\eps$ and $n_\eps$ of Condition 
\ref{condition_periodic_patterns}, for which Theorem \ref{main_theorem_stability} gives
the required 
bound that is uniform in $\eps$ and valid for large $kR$.
Another place where this problem is bypassed is \cite{BoFe:04}:~this paper considers homogenization of particular 2-d Helmholtz transmission problem (where the contrast in the coefficients depends on $\eps$). The result of \cite{BoFe:04} is that, via a contradiction argument, the sequence $u_\eps$ is bounded in $L^2$ and two-scale converges to the homogenized solution (albeit without an error estimate); this result is valid for fixed $k$, excluding a countable set, but crucially not assumed to be small.  

Finally, we note that another difference between the analysis in the present paper and that of \cite{CaGuMo:16} is that, for simplicity, \cite{CaGuMo:16} assume that $\hMA$, $\hn$, and $\Din$ are $C^\infty$, but the present paper makes much weaker regularity assumptions on $\hMA$, $\hn$, and $\Din$.
In particular, we allow for piecewise smooth $\hMA$ and $\hn$; recall from Remark \ref{rem:physical} that these are particularly important in applications. 

\subsubsection{Discussion of the $k$- and $\eps$-dependence of the bounds in Theorem \ref{main_theorem_homogenization} and Corollary \ref{corollary_stability}}

The bound \eqref{eq_homo_thm_1} shows that the ``relative error"
\beq\label{eq_rel_error}
\frac{k \|u_\eps-u_0- \eps u_1^\eps\|_{H^1_k(\Din)} + k \|u_\eps-u_0\|_{H^1_k(\Dout)}
}{
 \N{u_0}_{H^2(B_R)} + k \N{u_0}_{H^1_k(B_R)}
 }
\eeq
is controllably small if $\Clayerh (k\eps)^{1/2}$ is small, independently
of $k$ and $\eps$, i.e., if $k\eps \leq c (kR)^{-6}$ (with $c$ independent of $k$ and $\eps$) by the definition of
$\Clayer$ \eqref{main_eq_stability}. While this is certainly pessimistic, to our knowledge it is the first
homogenization result for the Helmholtz equation that is explicit in both $\eps$ and $k$. Note that $\Clayer \gtrsim kR$ (see the discussion after
\eqref{eq:1knorm}), and therefore even if we obtained a sharper upper bound on $\Clayer$,
the bound \eqref{eq_homo_thm_1} would still only show that \eqref{eq_rel_error} is small when
$k\eps \leq c (kR)^{-2}$. (Note that the factor of $\Clayer$ in the bound on \eqref{eq_rel_error} from \eqref{eq_homo_thm_1} arises from the standard treatment of $u_\eps-u_0$ minus the correctors as a solution of the PDE -- in our case the Helmholtz equation; see \eqref{tmp_duality_argument} and \eqref{eq_Weps} below.)

The bound \eqref{eq_newbound1} shows that if $\Clayerh (k\eps)^{1/2}$ is bounded independently
of $k$ and $\eps$, then 
\beq\label{eq:compare0}
\Cse \lesssim \CsH.
\eeq
From \eqref{main_eq_estimate_Cse} and \eqref{eq_CsolH}, both $\Cse$ and $\CsH$ are $\leq \Clayer$.
Therefore, the bound \eqref{eq:compare0} is only interesting when $\CsH \ll \Clayer$. 
A concrete case when this is true is when $\Din$ is star-shaped with respect to a point,
since then $\CsH\lesssim kR$ by \eqref{eq_Cstar} (which is the smallest possible
growth of $\Csol$ with $kR$), whereas $\Clayer \lesssim (kR)^3$ by its definition
\eqref{main_eq_stability}.

\section{Preliminary results} \label{sec:prelim}

\subsection{Morawetz-type identities and associated results}\label{sec:4}

When writing these identities, it is convenient to use the notation that $\langle \textbf{a},\textbf{b}\rangle :=\sum_{j=1}^d a_j \overline{b_j}$ for $\textbf{a}, \textbf{b}\in \Com^d$.
Here and in the rest of the paper, we use the convention that all indices are lowered, and repeated indices $i,j,$ and $\ell$, are summed over, but \emph{not} repeated indices $d$.

\ble[Morawetz-type identity]
Let $\domaingen\subset \Rea^d$.
Let $v\in C^2(\domaingen)$, $\MA\in C^1(\domaingen, \Sym)$, $n\in C^1(\domaingen,\Rea)$,
$\bZ\in C^1(\domaingen, \Rea^d)$, $\alpha\in C^2(\domaingen,\Rea)$, and
$\beta\in C^1(\domaingen,\Rea)$. Let 
\beq
\opL_{\MA,n} v:= \nabla\cdot (\MA \gv) + k^2 n\, v
\quad\tand\quad
\label{eq:cZ}
\cZ v:= \bZ\cdot \gv - \ri k \beta v + \alpha v.
\eeq
Then, in $D$, and where $\partial \bZ$ is the derivative matrix of $\bZ$, i.e.~$(\partial \bZ)_{ij} = \partial_i Z_j$,
\begin{align}\nonumber 
2 \Re \big\{\overline{\cZ v } \,\opL_{\MA,n} v \big\} = &\, \nabla \cdot \bigg[ 2 \Re\big\{\overline{\cZ v}\, A \gv\big\} + \bZ\Big(k^2n \nvs - 
\langle \MA\gv, \gv\rangle
\Big)- \nabla\alpha \nvs\bigg] 
- (2\alpha -\nabla\cdot\bZ) 
\langle \MA\gv, \gv\rangle 
\\
&\qquad+ 
\Big\langle \big((\bZ\cdot\nabla)\MA\big)\gv,\gv\Big\rangle
 - \Big( (\nabla\cdot\bZ-2\alpha) n + \bZ\cdot \nabla n\Big) k^2 \nvs -2 \Re\big\langle \MA \gv, \partial \bZ\, \gv\big\rangle\nonumber\\
&\qquad -2 \Re\big\{\ri k \vb
\langle \MA\gv, \nabla \beta\rangle
\big\}-2\Re \big\{\vb \langle (\MA-\MI) \gv, \nabla\alpha\rangle\big\} + \Delta \alpha\nvs.
\label{eq:morid1}
\end{align}
\ele

\bpf
Splitting $\cZ v$ up into its component parts, we see that the identity \eqref{eq:morid1}
is the sum of the following four identities:
\begin{align}\nonumber
2 \Re \big\{\bZ\cdot \overline{\nabla v} \,\opL_{\MA,n} v \big\} &= \nabla \cdot \bigg[ 2 \Re\big\{ \bZ\cdot \overline{\nabla v}\,\MA\nabla v \big\} + \bZ\big(k^2n\, \nvs - \langle \MA\gv,\gv\rangle\big) \bigg]+ 
\nabla\cdot\bZ\langle \MA\gv,\gv\rangle \\
&\quad
-2 \Re\big\langle \MA \gv, \partial \bZ\, \gv\big\rangle
+ \big\langle \big((\bZ\cdot\nabla)\MA\big)\gv,\gv\big\rangle - (\nabla\cdot\bZ\,n  +\bZ\cdot\nabla n)k^2 \nvs,
\label{A}
\end{align}
\beq\label{B}
2 \Re \big\{ \ri k \beta \overline{v} \,\opL_{\MA,n} v\big\} = \nabla \cdot \big[ 2 \Re \big\{\ri k \beta \overline{v}\, \MA\nabla v\big\}\big]-2 \Re\big(\ri k \vb\langle \MA\gv, \nabla \beta\rangle\big) ,
\eeq
\beq\label{C}
2 \Re \big\{\alpha \overline{v} \,\opL_{\MA,n} v \big\} = \nabla \cdot \big[ 2 \Re \{\alpha \overline{v}\, \MA\nabla v \}\big] + 2\alpha k^2 n\nvs
-2\alpha\langle \MA\gv,\gv\rangle- 2\Re \big\{\vb \langle \MA \gv, \nabla\alpha\rangle\big\}, \quad\tand
\eeq
\beq\label{D}
0 = -\nabla\cdot\big[\nabla\alpha \nvs\big] + 2 \Re \big\{\vb \nabla\alpha\cdot \nabla v\big\} + \Delta \alpha \nvs.
\eeq

To prove \eqref{B}, \eqref{C}, and \eqref{D}, expand the divergences on the right-hand sides (remembering that $\alpha$ and $\beta$ are real and that $\MA$ is symmetric, so $\langle \MA \bxi,\bxi\rangle$ is real for any $\bxi\in \Com^d$).

The basic ingredient of \eqref{A} is the identity
\beq\label{basic}
(\overline{\bZ\cdot \nabla v})\nabla\cdot(\MA\gv) = \nabla \cdot \big[(\bZ\cdot \overline{\nabla v})\MA\nabla v\big] - 
\langle \MA\gv,\partial \bZ\,\gv\rangle- \big((\bZ\cdot\nabla)\gvb\big)\cdot \MA\gv.
\eeq
To prove this, expand the divergence on the right-hand side and use the fact that the second derivatives of $v$ commute.
We would like each term on the right-hand side of \eqref{basic} to either be single-signed or be the divergence of something.
To deal with the final term we use the identity
\beq\label{Melenktrick}
2 \Re \big\{ (\bZ\cdot \nabla)\gvb \cdot \MA\gv\big\}= \nabla \cdot \big[ \bZ\langle \MA\gv,\gv\rangle \big] - (\nabla\cdot\bZ) \langle \MA\gv,\gv\rangle -
\big\langle \big((\bZ\cdot\nabla)\MA\big)\gv,\gv\big\rangle, 
\eeq
which can be proved by expanding the divergence on the right-hand side and using the fact that $\MA$ is symmetric.
Therefore, taking twice the real part of \eqref{basic} and using \eqref{Melenktrick} yields
\begin{align}\nonumber
2 \Re \big\{(\bZ\cdot \gv)\nabla\cdot(\MA\gv)\big\} =&  \nabla \cdot \bigg[ 2 \Re\big\{ (\bZ\cdot \overline{\nabla v})\,\MA\nabla v \big\} -\bZ\langle \MA\gv,\gv\rangle\bigg] + (\nabla\cdot\bZ)\langle \MA\gv,\gv\rangle \\&\qquad\qquad\qquad+ \big\langle \big((\bZ\cdot\nabla)\MA\big)\gv,\gv\big\rangle
-2\Re\langle \MA\gv,\partial \bZ\,\gv\rangle.  \label{A2}
\end{align}
Now add $k^2$ times
\beqs
2 \Re \left\{ (\bZ\cdot \overline{\nabla v})nv  \right\}= \nabla \cdot \big[ \bZ n \nvs \big] -(\nabla\cdot\bZ)n\nvs - \bZ\cdot\nabla n \nvs
\eeqs
(which is the analogue of \eqref{Melenktrick} with the vector $\nabla v$ replaced by the scalar $v$ and the matrix $\MA$ replaced by the scalar $n$) to \eqref{A2} to obtain \eqref{A}. 
\epf

\bre[Bibliographic remarks on Morawetz-type identities for $\opL_{\MA,n}$]\label{rem:biblio}
Multiplying $\Delta v$ by a derivative of $v$ goes back to Rellich \cite{Re:40, Re:43}, and multiplying $\nabla\cdot(\MA\gv)$ by a derivative of $v$ goes back to H\"ormander \cite{Ho:53} and Payne and Weinberger \cite{PaWe:58} (e.g., the identity \eqref{eq:morid1} with $\bZ= \bx$ and $n$, $\alpha$, and $\beta$ all equal zero appears as \cite[Equation 2.4]{PaWe:58}).

In the context of the Helmholtz equation, the identity \eqref{eq:morid1} with $\MA=\MI$, $n= 1$, $\bx$ replaced by a general vector field, and $\alpha$ and $\beta$ replaced by general scalar fields was the heart of Morawetz's paper \cite{Mo:75} (following the earlier work \cite{Mo:61, MoLu:68}); hence why we call \eqref{eq:morid1} a ``Morawetz-type" identity.
The identity \eqref{eq:morid1} with $\bZ=\bx$, $\MA$ variable, and $n= 1$ was used by Bloom in \cite{Bl:73}, and the identity \eqref{eq:morid1} with $\bZ=\bx$, $\MA=\MI$, and variable $n$ was used in Bloom and Kazarinoff in \cite{BlKa:77}. 
\ere

\bre[The origin of the condition $(\MA)_{d \ell}=0$ for $\ell= 1,\ldots, d-1$ in Condition \ref{cond:1}]\label{rem:Adell}
Our argument below requires that the term $2\Re\langle \MA\gv,\partial \bZ\,\gv\rangle$ be bounded below by a multiple of $|\partial_d v|^2$ when $\bZ= \be_d x_d$. Under this choice of $\bZ$,
\beq\label{eq:keyZA}
2\Re\big\langle \MA\gv,\partial \bZ\,\gv\big\rangle = 2 \Re\left\{\conj{\partial_d v}\left(\sum_{\ell=1}^d (\MA_{d\ell}) \partial_\ell v\right)\right\},
\eeq
hence we impose that $(\MA)_{d \ell}=0$ for $\ell= 1,\ldots, d-1$.
\ere

For notational convenience, given $\MA, n, \bZ, \beta,$ and $\alpha$, we let 
\beqs
\bQ(v):=2 \Re\big(\overline{\cZ v}\, \MA \gv\big) + \bZ\big(k^2n \nvs - \big\langle \MA\gv, \gv\big\rangle\big) - \nabla\alpha \nvs,
\eeqs
and 
\begin{align}\nonumber 
P(v) :=& 2 \Re \big(\overline{\cZ v } \,\opL_{\MA,n} v \big) + (2\alpha -\nabla\cdot\bZ) \langle \MA\gv, \gv\rangle 
\\
&\qquad- 
\Big\langle \big((\bZ\cdot\nabla)\MA\big)\gv,\gv\Big\rangle
 + \Big( (\nabla\cdot\bZ-2\alpha) n + \bZ\cdot \nabla n\Big) k^2 \nvs +2 \Re\big\langle \MA \gv, \partial \bZ\, \gv\big\rangle\nonumber\\
&\qquad +2 \Re\big(\ri k \ub
\langle \MA\gv, \nabla \beta\rangle
\big)+2\Re \big(\vb \langle (\MA-\MI) \gv, \nabla\alpha\rangle\big) - \Delta \alpha\nvs,\label{eq:P}
\end{align}
so that the identity \eqref{eq:morid1} becomes
$\nabla \cdot \bQ(v)= P(v).$

\begin{lemma}[The Morawetz-type identity \eqref{eq:morid1} integrated over a ball]
\label{lem:morid1int}
Recall the notation that $B_R:= \{\bx : |\bx|<R\}$ and $\Gamma_R := \partial B_R$. 
On $\Gamma_R$, let $\partial_r $ denote the normal derivative and let $\nabla_S$ denote the surface gradient. 
Let $v \in H^2(B_R)$, $\MA\in C^{1}(\overline{B_R},\SPD)$, $n\in C^{1}(\overline{B_R},\Rea)$, $\bZ\in C^1(\overline{B_R}, \Rea^d)$, $\beta\in C^1(\overline{B_R},\Rea)$, and $\alpha\in C^2(\overline{B_R},\Rea)$. Suppose further that, in a neighbourhood of $\Gamma_R$, $\MA=\MI$, $n=1$, $\bZ= \bx$, and $\alpha$ is constant.
Then
\beq
\int_{B_R} P(v)
=\int_{\GammaR} R
\left( \left|\partial_r v\right|^2 - |\nabla_S u|^2 + k^2 |u|^2\right)   - 2 k \, \Im \int_{\GammaR} \beta\vb \,\partial_r v + 2\alpha\,\Re \int_{\GammaR} \vb\, \partial_r v
\label{eq:morid1int}
\eeq
where $P(v)$ is defined by \eqref{eq:P}. 
\ele

\bpf
Under the assumption that $v\in \DOmegabar$, \eqref{eq:morid1int} follows from integrating the identity \eqref{eq:morid1} over $B_R$ and using the divergence theorem; the result then follows from the density of $\DOmegabar$ in $H^2(B_R)$ and the fact that \eqref{eq:morid1int} is continuous in $v$ with respect to the topology of $H^2(B_R)$. The integrated versions of similar Morawetz-type identities can be found in, e.g., \cite[Lemma 1]{NgVo:12} and \cite[Lemma 4.2]{GrPeSp:19}.
\epf

The following lemma deals with the contribution from $\GR$ in \eqref{eq:morid1int} when $v$ equals a solution of the Helmholtz equation satisfying the Sommerfeld radiation condition.

\ble[Inequality on $\GammaR$ for outgoing Helmholtz solutions]\label{lem:2.1}
Let $u$ be a solution of the homogeneous Helmholtz equation in
$\Rea^d\setminus \overline{B_{R_0}}$, for some $R_0>0$, satisfying the
Sommerfeld radiation condition \eqref{eq:src}.
Let $\alpha\in \Rea$ with $2\alpha\geq d-1$. Then, for $R>R_0$, 
\beqs
\int_{\GammaR} R\left( \left|
\partial_r u
\right|^2 - |\nabla_S u|^2 + k^2 |u|^2\right)   - 2 k R\, \Im \int_{\GammaR} \bar{u}\,\partial_r u 
+ 2\alpha\,\Re \int_{\GammaR}\bar{u}\,\partial_r u
 \leq 0,
\eeqs
where $\nabla_S$ is the surface gradient on $r=R$.
\ele

\bpf[References for proof of Lemma \ref{lem:2.1}]
See \cite[Lemma 2.1]{ChMo:08} or \cite[Lemma 2.4]{SpChGrSm:11} for the proof when $d=2,3$; the proof for general $d\geq 2$ is very similar. 
This result is also essentially contained in \cite[proof of Proposition 1]{NgVo:12}.
\epf


\subsection{Background results for the Helmholtz transmission problem}

Since $\DtN: H^{1/2}(\Gamma_R)\to H^{-1/2}(\Gamma_R)$ is continuous, the sesquilinear form $b(\cdot,\cdot)$ of the transmission problem defined by \eqref{eq:sesqui} is continuous on $H^1(B_R)$.

\ble[Well-posedness from an a priori bound]\label{lem:Fred}
Suppose that, under the assumption of existence, the solution of the transmission problem \eqref{eq:vf} with 
$F(v) = \int_{B_R}f \,\overline{v}$ and $f\in L^2(B_R)$ satisfies 
\beq\label{eq:bound_lem1}
\|u\|_{H^1_k(B_R)}
\leq C(\MA,n, k, R) \N{f}_{L^2(B_R)} \quad \tfa k\geq k_0,
\eeq
for some $C(\MA,n,k,R)>0$ and $k_0>0$.
Then the Helmholtz transmission problem of Definition \ref{def:transmission} is well-posed: i.e.,
given $F\in (H^1(B_R))'$, let $\widetilde{u}$ satisfy the variational problem \eqref{eq:vf};
then, if $k\geq k_0$, $\widetilde{u}$ exists, is unique, and the map $F\mapsto \widetilde{u}$
is continuous (i.e., $\Cs<\infty$).
\ele

\bpf
The inequalities
\beq\label{eq:TR}
\Re \big\{ - \langle \DtN  \phi,\phi\rangle_{\Gamma_R}\big\} \geq 0 \quad\tfa \phi \in H^{1/2}(\Gamma_R)
\eeq
(see, e.g., \cite[Lemma 3.3]{MeSa:10}),
\eqref{eq:nlimits}, and \eqref{eq:Alimits} imply that, for any $v\in H^1(B_R)$, 
\begin{align}
\Re \sesqui(v,v)
\geq
A_{\min} \N{\gv}^2_{L^2(B_R)} - k^2 \varmax \N{v}^2_{L^2(B_R)}
\label{eq:Garding}
\geq A_{\min} \N{v}^2_{H^1_k(B_R)} - k^2 (\varmax + A_{\min})\N{v}^2_{L^2(B_R)};
\end{align}
i.e.~$\sesqui(\cdot,\cdot)$ satisfies a G\aa rding inequality.
A bound on the solution in terms of the data such as \eqref{eq:bound_lem1}, under the assumption of existence, shows that the solution of the boundary value problem (if it exists) is unique. Since $\sesqui(\cdot,\cdot)$ is continuous and satisfies a G\aa rding inequality, the result then follows from Fredholm theory; see, e.g., \cite[Theorems 2.27 and 2.34]{Mc:00}, \cite[\S6.2.3]{Ev:98}. 
\epf

\begin{lemma}
[Well-posedness and bound for $f\in L^2$ implies well-posedness and bound for $F\in (H^1)'$]
\label{lem:H1}
Assume that $\MA$ and $n$ are such that, given $f\in L^2(B_R)$, the solution of the transmission problem of Definition \ref{def:transmission} with $F(v) = \int_{B_R}f \,\overline{v}$  exists, is unique, and satisfies the bound \eqref{eq:bound_lem1}
for some $C(\MA,n,k,R)>0$ and $k_0>0$.
Given $F\in (H^1(B_R))'$, let $\widetilde{u}$ satisfy the variational problem \eqref{eq:vf}.
Then $\widetilde{u}$ exists, is unique, and satisfies the bound 
\beqs
\N{\widetilde{u}}_{H^1_k(B_R)} \leq \frac{1}{\min(A_{\min},n_{\min})}\Big( 1 + 2n_{\max}  k \, C(\MA,n,k, R)\Big) \N{F}_{(H^1_k(B_R))'} \quad\tfa k\geq k_0.
\eeqs
\end{lemma}

\bpf[References for the proof]
This well-known result is proved in, e.g., \cite[Proof of Lemma 5.1]{GrPeSp:19} using the fact that $\sesqui(\cdot,\cdot)$ satisfies the G\aa rding inequality \eqref{eq:Garding}. 
\epf

In the rest of the paper, $|\cdot|_2$ denotes \emph{both} the Euclidean vector norm on $\Com^d$ \emph{and} the induced matrix norm on $\Com^{d\times d}$, and $\|\MA\|_{L^2(B_R, \Com^{d\times d})}:= \| |\MA(\cdot)|_2 \|_{L^2(B_R)}$ for $\MA\in \Com^{d\times d}$.

\ble[Solution of variational problems by approximation]\label{lem:basicapprox1}
Given coefficients $\MA$ and $n$ satisfying the conditions in Definition \ref{def:transmission}, assume that there exist sequences of coefficients $(\MA^\ell)_{\ell=0}^\infty$ and $(n^\ell)_{\ell=0}^\infty$ satisfying, for every $\ell$, the conditions in Definition \ref{def:transmission} (possibly with different limits $n_{\min}, n_{\max}, A_{\min},$ and $A_{\max}$ in the inequalities \eqref{eq:nlimits} and \eqref{eq:Alimits}), and such that (i)
\beqs
\N{\MA- \MA^\ell}_{L^2(B_R, \Com^{d\times d})}\leq \frac{1}{\ell} \quad\tand\quad \N{n-n^\ell}_{L^2(B_R)}\leq \frac{1}{\ell} \quad\tfa \ell,
\eeqs
and (ii) there exists an $\ell_0>0$ and $C(k,R)>0$ such that, for every $\ell\geq \ell_0$, the solution of the variational problem \eqref{eq:vf} with 
$F(v) = \int_{B_R}f \,\overline{v}$
and coefficients $\MA^\ell$ and $n^\ell$ exists, is unique, and satisfies the bound 
\beq\label{eq:basicapprox1}
\N{u}_{H^1_k(B_R)}\leq C(k,R) \N{f}_{L^2(B_R)}.
\eeq

Then the solution of the variational problem \eqref{eq:vf} with 
$F(v) = \int_{B_R}f \,\overline{v}$
and with coefficients $\MA$ and $n$ exists, is unique, and satisfies the bound \eqref{eq:basicapprox1}.
\ele

\bpf[References for the proof]
See, e.g., \cite[Pages 2901 and 2902]{GrPeSp:19}:~indeed, the proof goes through ad verbatim from the second paragraph of 
 \cite[Page 2901]{GrPeSp:19} to the end-of-proof symbol on \cite[Page 2902]{GrPeSp:19} with $A_\delta,n_\delta$ replaced by $\MA^\ell, n^\ell$.
\epf

Lemma \ref{lem:basicapprox1} holds with the $\|f\|_{L^2(B_R)}$ in the bound \eqref{eq:basicapprox1} replaced by $\|F\|_{(H^1_k(B_R))'}$ (with a different $C(k,R)$), thereby covering more general variational problems. In the rest of the paper, however,  we only use Lemma \ref{lem:basicapprox1} as stated, i.e.~applied to variational problems with right-hand sides of the form $\int_{B_R} f\vb$.

\subsection{Traces and weighted Sobolev spaces}
\label{sec:trace}

We repeatedly use the following standard trace result (see, e.g., \cite[Theorem 3.37]{Mc:00}).

\ble
\label{lem:trace}
If $D$ is a bounded $C^{\ell-1,1}$ domain and $1/2<s\leq \ell$, then the trace operator is bounded $H^s(D)\rightarrow H^{s-1/2}(\partial D)$.
\ele
We use also the multiplicative trace inequality for Lipschitz $D$, i.e.~that there exists
$C_{\rm tr}>0$ such that
\beq
\label{eq:multtrace}
\N{v}^2_{L^2(\partial \domain)}
\leq
C_{\rm tr}\left(
\frac{1}{\ell_D} \N{v}^2_{L^2(\domain)}
+
\N{v}_{L^2(\domain)}|v|_{H^1(\domain)}
\right) \quad\tfa v\in H^1(D),
\eeq
where $\ell_D$ is the diameter of $D$ (see, e.g.,
\cite[Theorem 1.5.1.10, last formula on Page 41]{Gr:85}
or \cite[Lemma 5.2]{HaJoNg:05}) and thus,
with $\|\cdot\|_{H^1_k(\domain)}$ defined by \eqref{eq:1knorm},
\beq
\label{eq:multtrace2}
k^{1/2}\N{v}_{L^2(\partial \domain)}
\lesssim
\N{v}_{H^1_k(\domain)} \quad\tfa v\in H^1(D).
\eeq
We define $\|\cdot\|_{H_k^s(\partial \domain)}$ for $s=1$ by
\beqs
 \|\phi\|_{H^1_k(\partial \domain)}^2 := \|\nabla_{\partial \domain} \phi\|^2_{L^2(\partial \domain)} +k^2\|\phi\|^2_{L^2(\partial \domain)},
\eeqs
and for $0<s<1$ by interpolation, choosing the specific norm given by the complex interpolation method (equivalently, by real methods of interpolation appropriately defined and normalised; see, e.g., \cite[Remark 3.6]{ChHeMo:15}). We then define the norms on $H^s(\partial \domain)$ and $H^s_k(\partial \domain)$ for $-1\leq s<0$ by duality,
\beqs
\|\phi\|_{H^s(\partial \domain)} := \sup_{0\neq\psi\in H^{-s}(\partial \domain)}\, \frac{|\langle\phi,\psi\rangle_{\partial \domain}|}{\|\psi\|_{H^{-s}(\partial \domain)}} \quad \mbox{and} \quad \|\phi\|_{H_k^s(\partial \domain)} := \sup_{0\neq\psi\in H^{-s}(\partial \domain)}\, \frac{|\langle\phi,\psi\rangle_{\partial \domain}|}{\|\psi\|_{H_k^{-s}(\partial \domain)}},
\eeqs
for $\phi\in H^s(\partial \domain)$,
where $\langle\phi,\psi\rangle_{\partial \domain}$ denotes the standard duality pairing that reduces to $(\phi,\psi)_{\partial \domain}$, the inner product on $L^2(\partial \domain)$, when $\psi\in L^2(\partial \domain)$. In the terminology of \cite[Remark 3.8]{ChHeMo:15}, with the norms we have selected, 
$\{H^s_k(\partial \domain):-1\leq s\leq 1\}$ are {\em exact interpolation scales}.

Finally, we need the weighted analogue of Lemma \ref{lem:trace} with $s=1$; see, e.g., \cite[Theorem 5.6.4]{Ne:01}.

\ble[Trace and extension in $k$-weighted spaces]\label{lem:trace_weight}
If $D$ is Lipschitz, then given $k_0>0$ there exists $C$, depending on $k_0$ but independent of $k$, such that 
\begin{equation*}
\N{ v}_{H^{1/2}_k(\partial D)} \leq C \|v\|_{H^1_k(D)}\quad \tfa v \in H^1(D) \tand k\geq k_0.
\end{equation*}
Furthermore, there exists an extension operator $\EE: H^{1/2}(\partial D) \rightarrow H^1(D)$ such that, given $k_0>0$ 
there exists $C'$, depending on $k_0$ but independent of $k$, such that 
\begin{equation}
\label{eq_stability_extension}
\|\EE(\phi)\|_{H^1_k(D)} \leq C' \|\phi\|_{H^{1/2}_k(\partial D)} \quad\tfa \phi \in H^{1/2}(\Gamma).
\end{equation}
\ele

\section{Proof of Theorem \ref{main_theorem_stability} (the well-posedness result)}

We follow the steps outlined in \S\ref{sec:idea}. In \S\ref{sec:smooth} we prove the
analogue of Theorem \ref{main_theorem_stability2} for smooth coefficients satisfying
Condition \ref{cond:smooth}. In \S\ref{sec:approx} we show how coefficients satisfying
Condition \ref{cond:1} can be approximated by coefficients satisfying Condition \ref{cond:smooth}.
Theorem \ref{main_theorem_stability2} then follows from using Lemma \ref{lem:basicapprox1}.
Having proved Theorem \ref{main_theorem_stability2}, Theorem \ref{main_theorem_stability}
then follows from Lemma \ref{lem:H1}.

\subsection
{The analogue of Theorem \ref{main_theorem_stability2} for a class of smooth coefficients}
\label{sec:smooth}

\subsubsection{Statement of the result}

\begin{condition}[A particular class of smooth coefficients]\label{cond:smooth}

\bit
\item
$\MA \in C^\infty(\Rea^d,\SPD)$, $n\in C^\infty(\Rea^d, \SPD)$, 
\item
there exist $A_{\min}, A_{\max}, n_{\min},$ and $n_{\max}$ such that $\MA$ and $n$ 
satisfy \eqref{eq:Alimits} and \eqref{eq:nlimits}, respectively,
\item
there exists $0<R_0<R$ such that $\supp (\MI-\MA)\subset B_{R_0}$ and $\supp (1-n)\subset B_{R_0}$,
\item 
\begin{equation}
\label{eq:monotone}
x_d \pdiff{\MA}{x_d}(\bx) \preceq 0
\quad\tand\quad 
x_d \pdiff{n}{x_d}(\bx) \geq 0 \quad\tfa \bx \in \Rea^d,
\end{equation}
\item
$(\MA_{d\ell})(\bx)=0$ for $\ell= 1,\ldots, d-1$ and for all $\bx\in \Rea^d$.
\eit
\end{condition}

Observe that the monotonicity conditions in \eqref{eq:monotone} are the continuous versions
of \eqref{eq:monotoneA} and \eqref{eq:monotonen} from Condition \ref{cond:1}. In addition,
the conditions \eqref{eq:monotone} along with the assumptions
$\supp (\MI-\MA)\subset B_{R_0}$ and $\supp (1-n)\subset B_{R_0}$ imply that $A_{\min}\geq 1$
and $n_{\max}\leq 1$.

\ble[Bound for coefficients satisfying Condition \ref{cond:smooth}]
\label{lem:smooth}
There exist $\Cwave,\Cone>0$ such that the following holds.
Given $R>R_0>0$, $\MA$ and $n$ satisfying Condition \ref{cond:smooth}, and $f\in L^2(B_R)$,
if $kR_0\geq \Cwave$, then the variational problem \eqref{eq:vf} with
$F(v)= \int_{B_R}f\,\overline{v}$ and coefficients $\MA$ and $n$ has a unique solution
that satisfies the bound  \eqref{eq:resol}.
\ele

\subsubsection{Overview of the ideas behind the proof of Lemma \ref{lem:smooth}}
\label{sec:ideasmooth}

The basic idea is to use the integrated Morawetz identity \eqref{eq:morid1int}
in $B_R$ with $v=u$,
\bit
\item
$\bZ$ given by the vector field in Definition \ref{def:Z} below;
the key point is that $\bZ$ transitions from  being equal to 
$x_d \be_d$ in $B_{R_0}$ (and hence in a neighbourhood of
$\supp(\MI-\MA)$ and $\supp(1-n)$ to being equal to $\bx$ in
$\Rea^d\setminus B_{R_1}$ for some $0<R_0<R_1<R$, with this
transition controlled by a radial function $\chi$ transitioning
from $0$ (in $B_{R_0}$) to $1$ (in $\Rea^d\setminus B_{R_1}$);
\item
$\beta=R$; and 
\item
$2\alpha = \nabla \cdot \bZ -q\,\chi$
for some $q\in[0,1]$ (to be fixed later in the proof).
\eit

These choices allow us to use the inequality of Lemma \ref{lem:2.1} in \eqref{eq:morid1int}
and obtain that $\int_{B_R}P(u) \leq 0$. The expression for $P(u)$ \eqref{eq:P} can then be
simplified, and bounded below using the monotonicity assumptions \eqref{eq:monotone} on
$\MA$ and $n$, with the result that 
\begin{align} \nonumber
&
\int_{B_R} \Big(
2|\partial_d u|^2\big(1-\chi\big)
+
|\nabla u|^2(2-q)\chi
+
qk^2|u|^2\chi
+
2r|\partial_r u|^2\chi^\prime
\Big)\, \rd x
-
2\Re\int_{B_R} x_d\partial_d \bar u\partial_r u\chi^\prime
\\
&
\hspace{0.5cm}
\leq
-2kR\Im \int_{B_R} f\bar u
+
\Re\int_{B_R} f\Big(
2 x_d\partial_d \bar u \big(1-\chi\big) + 2r\partial_r \bar u\chi +2\alpha \bar u
\Big)
+
\int_{B_R}  \Delta \alpha \nus.
\label{eq:rellich4}
\end{align}
Observe that when $\chi=0$ (i.e.~in $B_{R_0}$) we only have control of $|\partial_d u|^2$ in the
first integral on the left-hand side, but in $\supp(\chi)$ we have control of $|\gu|^2+k^2|u|^2$.

The rest of the argument consists of 
\ben
\item
getting rid of the sign-indefinite ``cross" term in the second integral on the
left-hand side of \eqref{eq:rellich4};
\item
using the following Poincar\'e-Friedrichs-type inequality (see, e.g., \cite[Lemma 2.7]{ChSpGiSm:20})
to put $|u|^2$ back in $B_{R_0}$ (i.e.~where $\chi=0$),
\beqs
\int_{B_{2R}}|v|^2
\leq
8\int_{B_{\sqrt{13} R}\setminus B_{2R}}
|v|^2  + 4 R^2 \int_{B_{\sqrt{13} R}} |\partial_d v|^2
\quad\tfa R>0 \tand v\in H^1(\R^d);
\eeqs
\item
using the following consequence of Green's identity and the inequality \eqref{eq:TR}
to put $|\gu|^2$ back in $B_{R_0}$,
\beq
\label{eq:Green_ineq}
\min\big\{ A_{\min},1\big\} \int_{B_R}\ngus
\leq
\max\big\{ n_{\max},1\big\} \int_{B_R} \nus + \Re\int_{B_R} f \ub;
\eeq
\item
bounding the term on the right-hand side of \eqref{eq:rellich4}
involving the Laplacian of $\alpha$;
\item
using the inequality $2 a b \leq \delta a^2 + \delta^{-1}b^2$ (for $a, b,\delta>0$)
on the other terms on the right-hand side of \eqref{eq:rellich4}.
\een

Regarding 1: this involves imposing a constraint on the growth of $\chi$
(see Point (ii) in Definition \ref{def:Z} and the discussion below this definition).
Regarding 2: here is the place where we lose powers of $k$ compared to the nontrapping estimate,
since, in the Poincar\'e-Friedrichs-type inequality, $|\partial_d u|^2$ is bounded below
by $|u|^2$ without a corresponding factor of $k^2$.

The paper \cite{ChSpGiSm:20} considers the constant-coefficient Helmholtz equation posed
outside a class of obstacles, the prototypical example of which is two aligned cubes
(so that there exist weakly-trapped rays between the two components of the obstacle).
The arguments in \cite{ChSpGiSm:20} use the same vector field $\bZ$ we use here, and,
in fact, \cite{ChSpGiSm:20} obtains exactly the inequality \eqref{eq:rellich4} in the
course of its arguments. The details of the Steps 1-5 above are therefore exactly as
in \cite{ChSpGiSm:20}, and so we do not repeat the details below, instead making precise
reference to the relevant results in \cite{ChSpGiSm:20}.

\subsubsection{The proof of Lemma \ref{lem:smooth}}

\begin{definition}[The vector field $\bZ$]
\label{def:Z}
Given $0<R_0<R_1<\infty$ such that $\supp (\MI-\MA)\subset B_{R_0}$ and
$\supp (1-n)\subset B_{R_0}$ and $\chi\in C^3[0,\infty)$ with
\begin{enumerate}
\item[(i)]
$\chi(r)=0$ for $0\leq r\leq R_0$, $\chi(r) = 1$,
for $r\geq R_1$, $0<\chi(r)<1$, for $R_0<r<R_1$; and
\item[(ii)]
$0\leq r\chi^\prime(r) < 4$, for $r>0$;
\end{enumerate}
let 
\begin{equation}
\label{eq:Z}
\bZ(\bx) := \be_d x_d\big(1-\chi(r)\big) + \bx \,\chi(r), \quad \bx\in \R^d.
\end{equation}
\end{definition}

The requirement (ii) on $\chi$ is needed to 
control the term $2\Re\int_{B_R} x_d\partial_d \bar u\partial_r u\chi^\prime(r)\, \rd x$
on the left-hand side of \eqref{eq:rellich4}, ensuring that this left-hand side is bounded
below by a multiple of $\int_{B_R}(|\partial_d u|^2 + \chi |\gu|^2$)\rd x.
This requirement imposes a constraint on $R_1/R_0$; indeed, \cite[Remark 1.5]{ChSpGiSm:20}
shows that if $\chi$ satisfies Points (i) and (ii) in Definition \ref{def:Z}, then
$R_1/R_0>\re^{1/4}\approx 1.284$, and, conversely, if $R_1/R_0>\re^{1/4}$, then there exists
a $\chi\in C^3[0,\infty)$ satisfying Points (i) and (ii) in Definition \ref{def:Z}.

\begin{lemma}
\label{lem:E0}
Let $u$ be the solution of the transmission problem of Definition \ref{def:transmission},
with $F(v)= \int_{B_R} f\,\overline{v}$ and coefficients $\MA$ and $n$ satisfying Condition
\ref{cond:smooth}. Let $\alpha$ be defined by 
\beq\label{eq:alpha}
2\alpha := \nabla \cdot \bZ -q\,\chi(r),
\eeq
for some $q\in[0,1]$, with $\chi$ as in Definition \ref{def:Z}.
Then $u$ exists and the inequality \eqref{eq:rellich4} holds.
\end{lemma}

\bpf[Proof of Lemma \ref{lem:E0}]
Since $\MA$ and $n$ satisfying Condition \ref{cond:smooth} are $C^\infty$, 
the existence of $u$ follows from the unique continuation principle (see
the references in \S\ref{sec:UCP}) and Fredholm theory (see the proof of Lemma \ref{lem:Fred}).
Since $\MA \in C^\infty$, $u \in H^2(B_R)$ for any $R>0$ by $H^2$ regularity for the operator
$\nabla\cdot(\MA\nabla )$; see, e.g., \cite[Theorem 4.16]{Mc:00}. By Lemmas \ref{lem:morid1int}
and \ref{lem:2.1}, 
\beq\label{eq:E01}
\int_{B_R} P(u) \leq 0,
\eeq
where $P(u)$ is defined by \eqref{eq:P}.
We first claim that, in $B_R$,
\beq\label{eq:P1}
P(u) \geq - 2 \Re (\overline{\cZ u}\, f ) - q \chi \ngus + q \chi k^2 \nus + 2 \Re \big\langle \MA\gu, \partial \bZ \,\gu\big\rangle - \Delta \alpha \nus.
\eeq
Indeed, this follows from using (i) the conditions on $\partial\MA/\partial x_d$ and $\partial n/\partial x_d$ \eqref{eq:monotone}, noting that when $\MA\neq \MI$ and $n\neq 1$, $\bZ= x_d \be_d$, (ii) the definitions of $\alpha$ \eqref{eq:alpha} and $\beta=R$, (iii) the fact that when $\chi \neq 0$, $\MA=\MI$ and $n=1$, (iv) the fact that $\alpha$ is constant when $\MA\neq \MI$.

To deal with the term $ \Re \big\langle \MA\gu, \partial \bZ \,\gu\big\rangle$,  we first observe that (with the summation convention for the indices $i$ and $j$ but not $d$)
\beq\label{eq:DZ1}
\langle \gv, \partial \bZ\, \gv\big\rangle=\partial_iZ_j\partial_i v\overline{\partial_j v}= |\partial_d v|^2\big(1-\chi(r)\big) + |\nabla v|^2\chi(r) + \left(r|\partial_r v|^2-x_d\conj{\partial_d v}\partial_r v\right)\chi^\prime(r).
\eeq
Next, when $\MA\neq \MI$, $\chi=0$ so $\bZ= x_d \be_d$; therefore $((\partial \bZ) \nabla u)_i = \delta_{id} \partial_d u$, so that \eqref{eq:keyZA} above holds. The assumption that $(\MA_{d\ell})=0$ for $\ell=1,\ldots, d-1$ implies that, when $\MA\neq \MI$, $\langle \MA \gu, \partial \bZ \gu\rangle = (\MA)_{dd} |\partial_d u|^2$. The fact that $A_{\min}\geq 1$ (noted after Condition \ref{cond:smooth}) implies that $(\MA)_{dd}\geq 1$, and so 
\beq\label{eq:DZ2}
\text{when } \MA\neq \MI, \quad \langle \MA \gu, \partial \bZ \,\gu\rangle \geq |\partial_d u|^2.
\eeq
Since $\chi=0$ when $\MA\neq \MI$, we can combine \eqref{eq:DZ1} and \eqref{eq:DZ2} to obtain that
\beq\label{eq:DZ3}
\langle \MA\gv, \partial \bZ\, \gv\big\rangle\geq  |\partial_d v|^2\big(1-\chi(r)\big) + |\nabla v|^2\chi(r) + \left(r|\partial_r v|^2-x_d\conj{\partial_d v}\partial_r v\right)\chi^\prime(r).
\eeq
Combining \eqref{eq:P1} and \eqref{eq:DZ3} we have 
\begin{align*}
P(u) \geq&-2 \Re \big(\overline{\cZ u }\, \fout \big) +2|\partial_d u|^2\big(1-\chi(r)\big) + |\nabla u|^2(2-q)\chi(r) +qk^2|u|^2\chi(r)\\
 & \qquad + 2r|\partial_r u|^2\chi^\prime(r)  
 -2 \Re\big(x_d\conj{\partial_d u}\partial_r u\,\chi^\prime(r)\big)
 -  \Delta \alpha |u|^2.
\end{align*}
 and \eqref{eq:rellich4} then follows from \eqref{eq:E01} by expanding the $\overline{\cZ u }\, f$ term using  that
\beqs
\bZ\cdot \nabla u =  x_d\partial_d u \big(1-\chi(r)\big) + r\partial_r u\chi(r).
\eeqs
\epf

\bpf[Proof of Lemma \ref{lem:smooth}]
The inequality \eqref{eq:rellich4} in Lemma \ref{lem:E0} is identical to the inequality
that is the result of \cite[Lemma 3.1]{ChSpGiSm:20}. The results
\cite[Lemmas 3.2, 3.3, 3.4, 3.5]{ChSpGiSm:20} (discussed in \S\ref{sec:ideasmooth}) then go
through ad verbatim (with the value of $q$ fixed in \cite[Lemma 3.5]{ChSpGiSm:20});
in seeing this, observe that, since $A_{\min}\geq 1$ and $n_{\max}\leq 1$, the factors in
front of the integrals of $\ngus$ and $\nus$ in \eqref{eq:Green_ineq} are both one, as they
are in the analogous inequality \cite[Equation 2.11]{ChSpGiSm:20}. Lemma \ref{lem:smooth}
is then a consequence of \cite[Lemma 3.5]{ChSpGiSm:20} (see \cite[Proof of Theorem 1.10 from
Lemma 3.5]{ChSpGiSm:20}).
\epf

\bre[Explicit expressions for $\Cwave$, and $k_0$]
\label{rem:explicit}
$\Cone$ is given in terms of $\chi$ in \cite[Equation 3.22]{ChSpGiSm:20}
(see the discussion in \cite[Proof of Theorem 1.10 from Lemma 3.5]{ChSpGiSm:20}).

$\Cwave$ is given in terms of the function $\chi$ in Definition \ref{def:Z} in
\cite[Equation 3.20]{ChSpGiSm:20}.
Indeed, to see that \cite[Equation 3.20]{ChSpGiSm:20} is a condition of the form $kR_0\geq \Cwave$
with $\Cwave$ dependent only on $\chi$, we need to show that $ (R_0)^2 m_\alpha(R)$ is a
dimensionless quantity depending only on $\chi$, where
$m_\alpha(R):= \sup_{\bx\in B_R}\Delta\alpha(\bx)$ and $2\alpha = \nabla \cdot \bZ -q\,\chi$
(as in \S\ref{sec:ideasmooth}), with $\bZ$ defined by \eqref{eq:Z}. By definition, $\alpha$ is
constant for $R\geq R_1$ (see \cite[Equation 3.15]{ChSpGiSm:20}), and thus
$m_\alpha(R)= m_\alpha(R_1)$. The only constraint on $R_1$ is that $R_1/R_0$ is sufficiently
large (see the discussion below Definition \ref{def:Z}); therefore, we can choose $R_1$ to be
proportional to $R_0$, and thus $(R_0)^2 m_\alpha(R)$ is a dimensionless quantity depending
only on $\chi$ as claimed.
\ere

\subsection
{
Approximation of $\MA$ and $n$ satisfying Condition \ref{cond:1} by 
$\MA_\delta$ and $n_\delta$ satisfying Condition \ref{cond:smooth}
}
\label{sec:approx}

\ble[Approximation of $\MA$ and $n$ satisfying Condition \ref{cond:1}]
\label{lem:approx}
Given $\MA$ and $n$ satisfying Condition \ref{cond:1}, there exist $(\MA^\ell)_{\ell=0}^\infty,
(n^\ell)_{\ell=0}^\infty$ such that, for $\ell$ sufficiently large, (i) $\MA^\ell$ and $n^\ell$
satisfy Condition \ref{cond:smooth}, and (ii) 
\beqs
\N{\MA^\ell-\MA}_{L^2(B_R, \Com^{d\times d})}
\leq
\frac{1}{\ell} \quad\tand\quad \N{n^\ell -n}_{L^2(B_R)} \leq \frac{1}{\ell}.
\eeqs
\ele

\bpf
The idea of the proof is to mollify $\MA$ and $n$ to approximate them by families of $C^\infty$ functions  $\MA_\delta$ and $n_\delta$  that  satisfy Condition \ref{cond:smooth}.
However, before mollification, we first must approximate $\MA$ and $n$ by functions that are constant in the $x_d$ direction in a neighbourhood of $x_d=0$. To see why this is necessary, suppose that, in some compact set, $n$ is only a function of $x_d$, and equals $C$ for $0\leq x_d< c$ and $C-x_d$ for $-c<x_d<0$ (observe that this function satisfies the monotonicity assumption \eqref{eq:monotonen} in $-c<x_d<c$). 
Then, with $n_\delta$ the standard mollification of $n$, given $\delta>0$, there exists an $x_d^*>0$ such that $(\partial n_\delta/\partial x_d)|_{x_d=x_d^*}\leq 0$, violating the requirement on $n_\delta$ in \eqref{eq:monotone}.

Let $0<R_0<R$ be such that $\supp(\MI-\MA)$ and $\supp(1-n)$ are both $\subset B_{R_0}$. Given $\delta_1>0$, let
\beq\label{eq:tilde_A_delta_1}
\widetilde{\MA}_{\delta_1}(\bx):= 
\begin{cases}
A_{\max} \MI, & \tfor \bx \in \big\{ (\bx',x_d): \, |\bx'|\leq R_0, \, |x_d|\leq \delta_1\big\},\\
A(\bx), & \text{ otherwise},
\end{cases}
\eeq
and let 
\beqs
\widetilde{n}_{\delta_1}(\bx):= 
\begin{cases}
n_{\min} \MI, & \tfor \bx \in \big\{ (\bx',x_d): \, |\bx'|\leq R_0, \, |x_d|\leq \delta_1\big\},\\
n(\bx), & \text{ otherwise},
\end{cases}
\eeqs
Observe that (i), $\widetilde{\MA}_{\delta_1}$ and $\widetilde{n}_{\delta_1}$ satisfy Condition \ref{cond:1}, (ii)
$\supp(\MI-\widetilde{\MA}_{\delta_1})$ and $\supp(1-\widetilde{n}_{\delta_1})$ are both $\subset B_{R_1(\delta_1)}$, where $R_1(\delta_1):= \sqrt{R_0^2+\delta^2_1}$, and 
(iii) 
there exists $C_d>0$ (depending only on the dimension $d$) such that 
\beq\label{eq:approx1}
\big\|\MA-\widetilde{\MA}_{\delta_1}\big\|_{L^2(B_R,\Com^{d\times d})}\leq C_d (R_0)^{d-1}\delta_1 A_{\max}
\quad\tand\quad
\N{n-\widetilde{n}_{\delta_1}}_{L^2(B_R)}\leq C_d (R_0)^{d-1}\delta_1 n_{\max}.
\eeq
Let $\psi\in C_{0}^\infty(\Rea^{d})$ be defined by 
\beqs
\psi(\bx) := 
\left\{
\begin{array}{ll}
C \exp\big((|\bx|^2 -1)^{-1}\big) & \text{ if } |\bx| <1,\\
0 & \text{ if } |\bx|\geq 1,
\end{array}
\right.
\eeqs
where $C$ is chosen so that $\int_{\Rea^{d}} \psi(\bx) \rd \bx =1$. Define $\psi_\delta(\bx) := \delta^{-d}\psi(\bx/\delta)$, so that $\psi_\delta(\bx)=0$ if $|\bx|>\delta$ and $\int_{\Rea^d} \psi_\delta(\bx) \rd \bx =1$.
Let
\beqs
(\widetilde{\MA}_{\delta_1})_\delta (\bx) := (\widetilde{\MA}_{\delta_1} * \psi_\delta)(\bx)= \int_{|\by|<\delta}\widetilde{\MA}_{\delta_1}(\bx-\by)\psi_\delta(\by)\, \rd \by,
\eeqs
where the convolution is understood element-wise, and similarly
\beqs
(\widetilde{n}_{\delta_1})_\delta (\bx) := (\widetilde{n}_{\delta_1} * \psi_\delta)(\bx)=\int_{|\by|<\delta}\widetilde{n}_{\delta_1}(\bx-\by)\psi_\delta(\by)\, \rd \by.
\eeqs
Standard properties of mollifiers (see, e.g., \cite[\S C.4 Theorem 6]{Ev:98}) imply that 
\bit
\item
$(\widetilde{\MA}_{\delta_1})_\delta, (\widetilde{n}_{\delta_1})_\delta \in C^\infty(\Rea^{d})$,
\item
given $R>0$, both
$\|(\widetilde{\MA}_{\delta_1})-(\widetilde{\MA}_{\delta_1})_\delta\|_{L^2(B_R;\Com^{d\times d})}$
and
$\|(\widetilde{n}_{\delta_1})-(\widetilde{n}_{\delta_1})_\delta\|_{L^2(B_R)}\tendo$
as $\delta\tendo$, and
\item
$n_{\min}\leq (\widetilde{n}_{\delta_1})_\delta \leq n_{\max}$ and
$A_{\min}\preceq (\widetilde{\MA}_{\delta_1})_\delta \preceq A_{\max}$
for all $\bx \in \Rea^{d}$.
\eit
Since $(\widetilde{\MA}_{\delta_1})_\delta$ is formed from $(\widetilde{\MA}_{\delta_1})$ by elementwise convolution, the condition $((\widetilde{\MA}_{\delta_1})_\delta)_{d\ell}(\bx)=0$ for $\ell= 1,\ldots, d-1$ and for all $\bx\in \Rea^d$ in Condition \ref{cond:smooth} follows from the corresponding condition on $\widetilde{\MA}_{\delta_1}$, which follows from the corresponding condition on $\MA$ in 
Condition \ref{cond:1}. 
Furthermore, since $\supp (\MI-\widetilde{\MA}_{\delta_1})$ and $\supp (1-\widetilde{n}_{\delta_1})\subset B_{R_1(\delta_1)}$, the definitions of $\MA_\delta$ and $n_\delta$ imply that that $\supp(\MI-\MA_\delta)$ and $\supp (1-n)\subset B_{R_1(\delta_1)+\delta}$ which is $\subset B_{R}$ for $\delta< R - R_1(\delta_1)= R- \sqrt{R_0^2 + \delta_1^2}$.

We now show that, when $\delta$ is sufficiently small, $(\widetilde{\MA}_{\delta_1})_\delta$ and $(\widetilde{n}_{\delta_1})_\delta$ satisfy the monotonicity conditions \eqref{eq:monotone} (with $\MA$ replaced by $(\widetilde{\MA}_{\delta_1})_\delta$ and $n$ replaced by
$(\widetilde{n}_{\delta_1})_\delta$).
Observe that
\beq\label{eq:diffA}
(\widetilde{\MA}_{\delta_1})_\delta(\bx+h\be_d) -(\widetilde{\MA}_{\delta_1})_\delta(\bx) = \int_{|\by|<\delta} \Big[
\widetilde{\MA}_{\delta_1}(\bx+ h \be_d - \by) - \widetilde{\MA}_{\delta_1}(\bx-\by)\Big] \psi_\delta(\by) \rd \by.
\eeq
We claim that, when $\delta\leq \delta_1$, (i) when $x_d>0$, the integrand on the right-hand side of \eqref{eq:diffA} is $\preceq 0$ for all $h\geq 0$, and (ii) when $x_d<0$, the integrand on the right-hand side of \eqref{eq:diffA} is $\succeq 0$ for all $h\geq 0$.
To see (i), first recall that, from its definition \eqref{eq:tilde_A_delta_1}, $\widetilde{\MA}_{\delta_1}(\bz)$ is constant in the $z_d$ direction when $|z_d|\leq \delta_1$.
If $x_d\geq \delta_1$ then $x_d-y_d>0$, and the integrand on the right-hand side of \eqref{eq:diffA} is $\preceq 0$ by \eqref{eq:monotoneA}. If $0<x_d<\delta_1$, $x_d-y_d$ is no longer $>0$ for all $|\by|<\delta$, but $x_d-y_d> -\delta_1$, and the fact that $\widetilde{\MA}_{\delta_1}(\bz)$ is constant in the $z_d$ direction when $|z_d|\leq \delta_1$ implies that  the integrand on the right-hand side of \eqref{eq:diffA} is either $=0$ (when $x_d - y_d + h\leq \delta_1$) or $\preceq 0$ by \eqref{eq:monotoneA} (when $x_d-y_d+h>\delta_1$).
The proof of (ii) is similar.
We have therefore shown that, for $\delta$ sufficiently small,
\beqs
x_d \Big[(\widetilde{\MA}_{\delta_1})_\delta(\bx+h\be_d) -(\widetilde{\MA}_{\delta_1})_\delta(\bx) \Big] \preceq 0 \quad\tfa h\geq 0;
\eeqs
since $(\widetilde{\MA}_{\delta_1})_\delta \in C^\infty(\Rea^d)$ this implies that $x_d (\partial (\widetilde{\MA}_{\delta_1})_\delta/\partial x_d)(\bx) \preceq 0$. 
In an essentially-identical way (with inequality in the sense of quadratic forms replaced by standard inequality), we find that $x_d (\partial 
(\widetilde{n}_{\delta_1})_\delta
/\partial x_d)(\bx) \geq 0$.

We have therefore shown that $(\widetilde{\MA}_{\delta_1})_\delta$ and $(\widetilde{n}_{\delta_1})_\delta$ satisfy Condition \ref{cond:smooth} for $\delta \leq\min\{ \delta_1, R- \sqrt{R_0^2 + \delta_1^2}\}$, where the first term in the minimum ensures that $(\widetilde{\MA}_{\delta_1})_\delta$ and $(\widetilde{n}_{\delta_1})_\delta$ satisfy the monotonicity condition \eqref{eq:monotone}, and the second term in the minimum ensures that $\supp(\MI- (\widetilde{\MA}_{\delta_1})_\delta)$ and $\supp(1- (\widetilde{n}_{\delta_1})_\delta)$ are both $\subset B_R$.

We now define $\MA^\ell$ to be $(\widetilde{\MA}_{\delta_1})_\delta$ for specific $\delta_1$ and $\delta$, and similarly for $n^\ell$. 
Indeed, from the properties of mollifiers above we have that, given $\delta_1>0$, $\eps>0$, there exists $\delta^*= \delta^*(\MA,d,\delta_1,\eps)>0$ such that
\beq\label{eq:approx2}
\big\|(\widetilde{\MA}_{\delta_1})_\delta -\widetilde{\MA}_{\delta_1}\big\|_{L^2(B_R,\Com^{d\times d})} \leq \frac{\eps}{2} \quad\tfa \delta\leq\delta^*(\MA,d,\delta_1,\eps).
\eeq
Set
$\delta_1:= (2 C_d\, (R_0)^{d-1} A_{\max}\ell)^{-1}$;
 then, if $\delta \leq \delta^*(\MA,d,\delta_1, \ell^{-1})$, the inequalities \eqref{eq:approx1} and \eqref{eq:approx2} imply that
\begin{align*}
\big\|
\MA - (\widetilde{\MA}_{\delta_1})_\delta
\big\|_{L^2(B_R;\Com^{d\times d})}
&\leq \big\|
\MA - \widetilde{\MA}_{\delta_1}
\big\|_{L^2(B_R;\Com^{d\times d})}
+
\big\|
\widetilde{\MA}_{\delta_1} - (\widetilde{\MA}_{\delta_1})_\delta
\big\|_{L^2(B_R;\Com^{d\times d})},\\
&\leq C_d \,(R_0)^{d-1} \delta_1A_{\max} + \frac{1}{2\ell}\, \leq \,\frac{1}{2\ell}+\frac{1}{2\ell}= \frac{1}{\ell}.
\end{align*}
Recall from above that $(\widetilde{\MA}_{\delta_1})_\delta$ and $(\widetilde{n}_{\delta_1})_\delta$
satisfy Condition \ref{cond:smooth} when $\delta \leq\min\{ \delta_1, R- \sqrt{R_0^2 + \delta_1^2}\}$.
We therefore set
\beqs
\delta := \min\left\{ \delta_1, \,\,R- \sqrt{R_0^2 + \delta_1^2}, \,\,\delta^*(\MA,d,\delta_1, \ell^{-1})\right\}\quad\tand\quad \MA^\ell:=(\widetilde{\MA}_{\delta_1})_\delta;
\eeqs
then $\|\MA-\MA^\ell\|_{L^2(B_R;\Rea^{d\times d})}\leq \ell^{-1}$ and $\MA^\ell$ satisfies Condition \ref{cond:smooth}. The definition of $n^\ell$ follows in an essentially-identical way.
\epf

\section{Proofs of Lemma \ref{main_lemma_homogenization}, Theorem \ref{main_theorem_homogenization}, and Corollary \ref{corollary_stability}}
\label{section_homogenization}

If $v: \Din \times Y \to \CCC$, let  $v^\eps: \Din \to \CCC$ denote the function
\begin{equation*}
v^\eps(\xx) \eq v\left (\xx, \left \{\frac{\xx}{\eps} \right \}\right ),
\end{equation*}
where $\{\yy\} = \yy \mod 1$. (As a special case of this, if $v: Y \to \CCC$,
then $v^\eps(\xx):= v(\{\xx/\eps\})$.) Given such a $v$,
\begin{equation}
\label{eq_eps_diff}
\pd{v^\eps}{x_j} = \left (\pd{v}{x_j} + \frac{1}{\eps} \pd{v}{y_j} \right )^\eps.
\end{equation}
From here on, when we use the notation $\lesssim$, the omitted constant only depends
on $\hn$, $\hMA$, and $\Din$.

\subsection{The homogenized problem}

\subsubsection{The homogenized coefficients}

The homogenized coefficient $\nH$ is the mean value of
$\hn$ over the periodic cell $Y$, i.e., (recalling that $|Y| = 1$)
\begin{equation}
\label{eq_def_nH}
\nH \eq \int_Y \hn.
\end{equation}
The definition of $\MAH$ requires the auxiliary functions
$\hchi_j$, $j=1,2,3$, and matrices $\hMC$ and $\hMB$.

Let $H^1_\sharp(Y)$ be the subspace of $H^1_\per(Y)$
consisting of functions with zero mean. The function $\hchi_j \in H^1_\sharp(Y)$
is defined as the solution of
\begin{equation}\label{eq_hchi_def}
-\divy \left (\hMA \grady \hchi_j\right ) = -\pd{(\hMA)_{j \ell}}{y_\ell}
\end{equation}
where the index $\ell$ on the right-hand side is summed over,
and the derivative on the right-hand side is understood in a distributional sense.
Lemma \ref{lemma_condition_patterns} shows that $\hchi_j$, $j=1,2,3,$ exists, is unique,
and is in $H^{1+s}(Y) \cap C^1(\PY)$ (i.e., $C^1$ on each element of the partition $\PY$) for some $s > 0$,
with
\begin{equation}
\label{eq_hchi}
\|\hchi_j\|_{H^{1+s}(Y)} \lesssim 1
\quad \tand \quad
\|\hchi_j\|_{W^{1,\infty}(\PY)} \lesssim 1,
\end{equation}
where both the omitted  constant and $s$ depend only on $\hMA$. The matrix $\hMC$ is defined by
\begin{equation}
\label{eq_def_C}
(\hMC)_{j \ell} \eq \pd{\hchi_\ell}{y_j}, \quad 1 \leq j,\ell \leq d.
\end{equation}
Because of \eqref{eq_hchi}, $(\hMC)_{j \ell} \in H^s(Y) \cap C^0(\PY)$ for some $s>0$ with
\beqs
\|\hMC\|_{H^{s}(Y)} \lesssim 1
\quad \tand \quad
\|\hMC\|_{L^\infty(Y)} \lesssim 1.
\eeqs
With
\begin{equation}
\label{eq_B}
\hMB := \hMA(\MI - \hMC),
\end{equation}
the homogenized matrix coefficient $\MA^H$
is the mean value of $\hMB$ over the periodic cell $Y$, i.e., 
\begin{equation}
\label{eq_def_AH}
(\MAH)_{j\ell} = \int_Y (\hMB)_{j\ell}, \qquad 1 \leq j,\ell \leq d;
\end{equation}
note that $\MAH \in \SPD$ by, e.g., \cite[Remarks 2.6 and 2.7]{BeLiPa:78},
\cite[\S6.3]{cioranescu_donato_1999a}.
We now claim that $\hMB \in H^s(Y) \cap C^0(\PY) \cap L^\infty(Y)$ for some $0<s<1/2$; indeed, this follows from the facts that $\hMC\in H^s(Y)$, 
$\hMA$ is (at least) piecewise $C^{0,1}$ (by Condition \ref{condition_periodic_patterns}), and a piecewise $H^s$ function for $s<1/2$ is in $H^s(Y)$ (by the definition of the Slobodeckij seminorm). 

\subsubsection{Basic properties of the homogenized problem}

We now prove Lemma \ref{main_lemma_homogenization}, i.e., well-posedness
of the homogenized problem. 

\begin{proof}[Proof of Lemma \ref{main_lemma_homogenization}]
We first show that if $\hn$ and $\hMA$ are admissible periodic patterns
(in the sense of Definition \ref{condition_periodic_patterns}), then $\nH$ and $\MAH$
satisfy Condition \ref{cond:1}; once this is established, the bound \eqref{eq_CsolH}
follows from Theorem \ref{main_theorem_stability}.

Since both $\nH$ and $\MAH$ are constant in $\Din$,
they are clearly independent of $x_d$; the monotonicity
conditions \eqref{eq:monotoneA} and \eqref{eq:monotonen}
therefore hold if $\nH \leq 1$ and $\MAH \succeq \MI$.
Recall from the discussion below Condition \ref{cond:1}
that $n_{\max}\leq 1$ and $A_{\min}\geq 1$. Since $\nH$ is
obtained from $\hn$ through simple averaging \eqref{eq_def_nH}, $\nH \leq 1$. 
While the averaging process is more complicated for $\MAH$,
\cite[Theorem 13.7]{cioranescu_donato_1999a} implies that 
$\MAH \succeq \MI$.

Therefore, to show that $\nH$ and $\MAH$ satisfy Condition \ref{cond:1}, it only remains to show that $\MAH_{d,\ell} = 0$ for $\ell=1,\dots,d-1$. By \eqref{eq_def_AH},
\begin{equation*}
\MAH = \int_Y \hMB = \int_Y \hMA - \int_Y \hMA \hMC,
\end{equation*}
it is sufficient to show that $(\hMC)_{\ell d} = 0$ for $\ell=1,\dots,d-1$. By the definition of $\hMC$ \eqref{eq_def_C},
\begin{equation*}
(\hMC)_{d\ell} \eq \frac{\partial \hchi_{\ell}}{\partial y_d} = 0
\end{equation*}
since $\hchi$ is independent of $y_d$, and
\begin{equation*}
(\hMC)_{\ell d} \eq \frac{\partial \hchi_d}{\partial y_\ell};
\end{equation*}
however, $\hchi_d = 0$ since the right-hand side of \eqref{eq_hchi} is 
\begin{equation*}
-\sum_{\ell=1}^d\frac{\partial (\hMA)_{d\ell}}{\partial y_\ell}
=
-\frac{\partial (\hMA)_{dd}}{\partial y_d}
=
0,
\end{equation*}
because $\hMA$ is independent of $y_d$.

If $\Din$ is star-shaped, since $\MAH \succeq \MI$ and $\nH\leq 1$,
$\MAH$ and $\nH$ satisfy \cite[Condition 2.6]{GrPeSp:19} with
$\mu_1 = A_{\min}=1$ and $\mu_2 = n_{\min}$ (see the example in \cite[Condition 2.10]{GrPeSp:19})
and so, by \cite[Theorem 2.7]{GrPeSp:19}, if $F(v) = \int_{B_R} f\,\overline{v}$, then 
\beqs
\N{u_{0}}_{H^1_k(B_R)}
\leq
\frac{2}{\sqrt{n_{\min}}}
\sqrt{R^2 + \frac{1}{n_{\min}} \left( R + \frac{d-1}{2k}\right)^2}
\N{f}_{L^2(B_R)}
\quad \tfa k>0;
\eeqs
the bound \eqref{eq_Cstar} follows from combining this with Lemma \ref{lem:H1}.
\end{proof}

\begin{lemma}[Shift property for the homogenized problem]
\label{lemma_shift}
If $\Gamma$ is $C^{1,1}$, then for all $g \in L^2(B_R)$, 
if $v \in H^1(B_R)$ satisfies
$-\div \left (\MAH \grad v\right ) = g \text{ in } B_R$
then $v \in H^2(\Din)$ with
$|v|_{H^2(\Din)} \lesssim \|g\|_{L^2(B_R)}.$
\end{lemma}

\bpf
Since $\MAH$ is constant in both $\Din$ and $\Dout$, Lemma \ref{lemma_shift}
holds if $\Gamma$ is $C^{1,1}$ by, e.g., \cite[Theorem 4.18]{Mc:00}.
(If $\Gamma$ has edges or corners, similar shift properties hold but with
the angles of edges/corners constrained by the values of $\MAH$, see, e.g., \cite{costabel_dauge_nicaise_1999a}.)
\epf

We abbreviate $\CsH$ by $\CsHshort$ in the rest of the paper.

\begin{lemma}[Bounds on $u_0$]
\label{lemma_estimates_u0}
Assume that $\hn$ and $\hMA$ satisfy Condition \ref{condition_periodic_patterns},
and $\Gamma$ is $C^{1,1}$. Then, for all $f \in L^2(B_R)$, the solution $u_0$ of
\eqref{eq:vf} with coefficients $\nH$ and $\MAH$ and $F(v)= \int_{B_R} f\, \overline{v}$
satisfies
\begin{equation}
\label{eq_u0_H2}
\N{u_0}_{H^2(B_R)} + k\|u_0\|_{H^1_k(B_R)} \lesssim \CsHshort \|f\|_{L^2(B_R)}.
\end{equation}
\end{lemma}

\bpf
The bound on $\|u\|_{H^1_k(B_R)}$ in \eqref{eq_u0_H2} follows from the definition of $\CsHshort$.
The bound on $\|u\|_{H^2(B_R)}$ in  \eqref{eq_u0_H2} follows from Lemma \ref{lemma_shift} with
$v=u$, $g= \nH k^2 u +f$, and $\alpha=0$, and the fact that $\CsHshort \gtrsim 1$
(as noted below \eqref{eq:1knorm}).
\epf

\subsection{The first-order, second-order, and boundary correctors}

Periodic homogenization relies on the formal asymptotic expansion
\begin{equation*}
u_\eps(\xx)
= 
u_0(\xx) + \eps u_1^\eps(\xx) + \eps^2 u_2^\eps(\xx) + \dots,
\end{equation*}
where $u_0$ is the homogenized solution and
$u_1,u_2$ are ``two scale'' functions usually called ``correctors''.
Since, in our case, the homogenization process only takes place in
$\Din$, we need an additional corrector, called the ``boundary corrector'',
to deal with the interface $\Gamma$. 

\subsubsection{The first-order corrector}

The first-order corrector
$u_1: \Din \times Y \to \CCC$ is defined by
\begin{equation}
\label{eq_def_first_order_corrector}
u_1(\bx,\by) \eq -\hchi_j(\by) \pd{u_0}{x_j}(\bx).
\end{equation}

\begin{lemma}[Key properties of the first-order corrector]
\label{lemma_u1_eps}
\begin{equation}
\label{eq_grad_u1}
\grady u_1 = -\hMC \grad u_0 \quad\text{ for a.e.~$\yy \in Y$},
\end{equation}
and
\begin{equation}
\label{eq_u1_L2}
k\|u_1^\eps\|_{L^2(\Din)} + \|(\gradx u_1)^\eps\|_{L^2(\Din)} \lesssim
k \N{u_0}_{H^1_k(B_R)} + |u_0|_{H^2(\Din)}.
\end{equation}
Furthermore, 
\begin{equation}
\label{eq_u1_H1}
\N{u_1^\eps}_{H^1_k(B_R)}
\lesssim
\left (1+(k\eps)^{-1}
\right )
\left (|u_0|_{H^2(\Din)} + k \N{u_0}_{H^1_k(B_R)}\right ),
\end{equation}
and, if $kR_0 \gtrsim 1$, 
\begin{equation}
\label{eq_u1_H12}
\N{u_1^\eps}_{H^{1/2}_k(\Gamma)}  
\lesssim
\big(1 + (k\eps)^{-1/2} \big) \left (k\N{u_0}_{H^1(\Din)} + \N{u_0}_{H^2(\Din)}\right ).
\end{equation}
\end{lemma}

\begin{proof}
By the definitions of $\hMC$ \eqref{eq_def_C} and of $u_1$ \eqref{eq_def_first_order_corrector},
\begin{equation*}
\frac{\partial u_1}{\partial y_\ell}
=
-
\frac{\partial \hchi_j}{\partial y_\ell} \frac{\partial u_0}{\partial x_j}
=
\left (-\hMC \grad u_0\right )_\ell,
\end{equation*}
which is \eqref{eq_grad_u1}. By the definition of $u_1$ \eqref{eq_def_first_order_corrector}
and the bound on $\hchi_j$ \eqref{eq_hchi}, \begin{equation*}
\|u_1^\eps\|_{L^2(\Din)}
\leq
\|\hchi_j\|_{L^\infty(Y)}
\left \|\frac{\partial u_0}{\partial x_j}\right \|_{L^2(\Din)}
\lesssim
|u_0|_{H^1(\Din)}
\end{equation*}
and
\begin{equation*}
\|(\gradx u_1)^\eps\|_{L^2(\Din)}
\leq
\|\hchi_j\|_{L^\infty(Y)}
\left \|\frac{\partial u_0}{\partial x_j}\right \|_{H^1(\Din)}
\lesssim
|u_0|_{H^2(\Din)},
\end{equation*}
and then the bound \eqref{eq_u1_L2} follows.
By \eqref{eq_eps_diff} and \eqref{eq_grad_u1},
\beqs
\grad (u_1^\eps)
=
\left (\gradx u_1\right )^\eps + \frac{1}{\eps} \left (\grady u_1\right )^\eps
= \left (\gradx u_1\right )^\eps  - \frac{1}{\eps}\hMC^\eps \grad u_0.
\eeqs
Therefore
\begin{equation*}
\N{\grad(u_1^\eps)}_{L^2(\Din)}
\lesssim
|u_0|_{H^2(\Din)} + \eps^{-1} |u_0|_{H^1(\Din)}
\lesssim
(k\eps)^{-1} \left (|u_0|_{H^2(\Din)} + k \N{u_0}_{H^1_k(B_R)}\right ).
\end{equation*}
Then, the bound \eqref{eq_u1_H1} follows from combining this with the bound on
$\|u_1^\eps\|_{L^2(\Din)}$ in \eqref{eq_u1_L2}.

Let the operator $T: \LL^2(\Gamma) \to L^2(\Gamma)$ be defined by
$T\vv \eq - \hchi_j^\eps \vv_j$ for all $\vv \in \LL^2(\Gamma)$; then
\begin{equation*}
\|T\vv\|_{L^2(\Gamma)} \lesssim \|\vv\|_{\LL^2(\Gamma)} \quad \tfa \vv \in \LL^2(\Gamma).
\end{equation*}
In addition,  $T: \HH^1(\Gamma) \to H^1(\Gamma)$ with
\begin{equation*}
\|T\vv\|_{H^1(\Gamma)}
\lesssim
\eps^{-1} \|\vv\|_{\LL^2(\Gamma)} + \|\vv\|_{\HH^1(\Gamma)} \quad \tfa \vv \in \HH^1(\Gamma).
\end{equation*}
By interpolation (see, e.g., \cite[Appendix B]{Mc:00}), $T$ then maps $\HH^{1/2}(\Gamma)\to H^{1/2}(\Gamma)$ with
\begin{equation*}
\|T\vv\|_{H^{1/2}(\Gamma)}
\lesssim
\eps^{-1/2} \|\vv\|_{\LL^2(\Gamma)} + |\vv|_{\HH^{1/2}(\Gamma)}.
\end{equation*}
We now observe that $u_1^\eps|_\Gamma = T ((\grad u_0)|_\Gamma)$, so that
\begin{equation}\label{eq:Friday1}
\|u_1^\eps\|_{H^{1/2}(\Gamma)}
\lesssim
\eps^{-1/2} \|\grad u_0\|_{\LL^2(\Gamma)} + |\grad u_0|_{\HH^{1/2}(\Gamma)}.
\end{equation}
By Lemma \ref{lem:trace},
\begin{equation}\label{eq:Friday2}
\N{\nabla u_0}_{\HH^{1/2}(\Gamma)} \lesssim \|u_0\|_{H^2(\Din)},
\end{equation}
and, by \eqref{eq:multtrace},
\begin{align}\nonumber
\N{\grad u_0}_{\LL^2(\Gamma)}^2
&\lesssim
\frac{1}{R_0} \|u_0\|_{H^1(\Din)}^2 + \|u_0\|_{H^1(\Din)}\|u_0\|_{H^2(\Din)}
\\
&\lesssim
\frac{1}{k} \left ( 1 + \frac{1}{kR_0} \right )
\left (k^2\|u_0\|_{H^1(\Din)}^2 + \|u_0\|_{H^2(\Din)}^2\right ).\label{eq:Friday3}
\end{align}
Using \eqref{eq:Friday2} and \eqref{eq:Friday3} in \eqref{eq:Friday1},
and recalling the assumption that $kR_0 \gtrsim 1$, the result \eqref{eq_u1_H12} follows.
\end{proof}

\subsubsection{The second-order corrector}

The second-order corrector arises since $\hn^\eps$ varies (in addition to $\hMA^\eps$ 
varying); that corrector is formally defined by
\begin{equation*}
u_2(\xx,\yy) \eq k^2 \hmu(\yy) u_0(\xx),
\end{equation*}
where $\hmu$ is the unique element of $H^1_\sharp(Y)$ such that
\begin{equation}\label{eq_mu}
-\divy \left (\hMA \grady \hmu\right ) = \hn - \nH.
\end{equation}
By Lemma \ref{lemma_condition_patterns}, $\hmu \in C^1(\PY)$ with
$\|\hmu\|_{W^{1,\infty}(\PY)} \lesssim 1.$

We can actually bypass referring explicitly to this second-order corrector in the rest
of proof of Theorem \ref{main_theorem_homogenization}, and directly establish the following
result (where, both here and in the rest of the paper, $(\cdot,\cdot)_{\Din}$ denotes the $L^2(\Din)$ inner product, and similarly for inner products over other domains).

\begin{lemma}[Bound on the contribution from the second-order corrector]
\label{lemma_L2_term}
\begin{equation}
\label{eq_L2_term}
k^2 \big|\big((\hn^\eps - \nH) u_0,v\big)_\Din\big|
\lesssim
\eps \left (k\N{u_0}_{H^1_k(\Din)}\right ) \|v\|_{H^1_k(\Din)} \quad\tfa v \in H^1(\Din).
\end{equation}
\end{lemma}

\begin{proof}
Since both $\hMA$ and $\hmu$ are independent of $\xx$, by \eqref{eq_mu},
\begin{align*}
-\div \left [(\hMA \grady \hmu)^\eps\right ]
=
-\frac{1}{\eps} \left[
\divy (\hMA \grady \hmu)
\right ]^\eps
=
\frac{1}{\eps} (\hn^\eps - \nH) \in L^\infty(\Din).
\end{align*}
Therefore,
\begin{align*}
-\div (u_0 (\hMA \grady \hmu)^\eps)
&=
-u_0 \div (\hMA \grady \hmu)^\eps
-\grad u_0 \cdot (\hMA \grady \hmu)^\eps
\\
&=
\frac{1}{\eps} (\hn^\eps - \nH) u_0
-\grad u_0 \cdot (\hMA \grady \hmu)^\eps.
\end{align*}
By Condition \ref{condition_periodic_patterns} and Lemma \ref{lemma_condition_patterns},
$\hMA \in C^0(\PY)$ and $\hmu \in C^1(\PY)$. Therefore, the function $(\hMA \grady \hmu)^\eps$
is piecewise continuous on a partition of $\Din$ consisting of a finite number of open sets,
so that its trace on $\partial \Din$ is well-defined. In addition,
\begin{equation*}
\|(\hMA\grady \hmu)^\eps\|_{L^\infty(\Din)}
\leq
\|\hMA\|_{L^\infty(Y)}\|\grady \hmu\|_{L^\infty(Y)}
\lesssim 1
\end{equation*}
and similarly
\begin{equation*}
\|(\hMA\grady \hmu)^\eps\|_{L^\infty(\partial \Din)}
\lesssim 1.
\end{equation*}
By Green's first identity in $\Din$, 
\begin{align*}
\int_\Din u_0 (\hMA \grady \hmu)^\eps \cdot \grad v
&=
\int_\Gamma u_0 (\hMA \grady \hmu)^\eps \cdot \nnu v
-\int_\Din \div (u_0 \hMA \grady \hmu)^\eps v
\\
&=
\int_\Gamma u_0 (\hMA \grady \hmu)^\eps \cdot \nnu v
+\frac{1}{\eps} \int_\Din (\hn^\eps - \nH) u_0 v
-
\int_\Din \grad u_0 \cdot (\hMA \grady \hmu)^\eps v,
\end{align*}
and thus
\begin{equation*}
\frac{1}{\eps} \int_\Din (\hn^\eps-n^H) u_0 v
=
\int_\Gamma u_0 (\hMA \grady \hmu)^\eps \cdot \nnu v
-
\int_\Din \grad u_0 \cdot (\hMA \grady \hmu)^\eps v
-
\int_\Din u_0 (\hMA \grady \hmu)^\eps \cdot \grad v.
\end{equation*}
To estimate each term in the right-hand side, we use the Cauchy-Schwarz
inequality on each $L^2$ inner-product, so that
\begin{align*}
|((\hn^\eps-\nH)u_0,v)_{\Din}|
&\leq
\eps
\|(\hMA \grady \hmu)^\eps\|_{L^\infty(\Din)}\|u_0\|_{L^2(\Din)}|v|_{H^1(\Din)}
\\
&\,+
\eps\|(\hMA \grady \hmu)^\eps\|_{L^\infty(\Din)}|u_0|_{H^1(\Din)}\|v\|_{L^2(\Din)}
+
\eps\|(\hMA \grady \hmu)^\eps\|_{L^\infty(\Gamma)}\|u_0\|_{L^2(\Gamma)}\|v\|_{L^2(\Gamma)}
\\
&\lesssim
\eps \Big(
\|u_0\|_{L^2(\Din)}|v|_{H^1(\Din)}
+
|u_0|_{H^1(\Din)}\|v\|_{L^2(\Din)}
+
\|u_0\|_{L^2(\Gamma)}\|v\|_{L^2(\Gamma)}
\Big).
\end{align*}
Multiplying by $k^2$ and then using  \eqref{eq:multtrace2}, we have
\begin{align*}
k^2 |((\hn^\eps-\nH)u_0,v)_{\Din}|
&\lesssim
k\eps
\Big(
k\|u_0\|_{L^2(\Din)}|v|_{H^1(\Din)}
+
k|u_0|_{H^1(\Din)}\|v\|_{L^2(\Din)}
+
k\|u_0\|_{L^2(\Gamma)}\|v\|_{L^2(\Gamma)}
\Big)
\\
&\lesssim
k\eps \|u_0\|_{H^1_k(\Din)} \|v\|_{H^1_k(\Din)},
\end{align*}
which is the result \eqref{eq_L2_term}.
\end{proof}

\subsubsection{Boundary corrector}

The following lemma is proved using Theorem \ref{theorem_vector_potential_main} from
Appendix \ref{appendix_vector_potentials}. A similar result (without the explicit $k$-dependence)
appears in \cite[Lemma 2.2]{OnVe:07} and \cite[Page 2536]{CaGuMo:16}.

\begin{lemma}[Existence of vector potential]
\label{lemma_vector_potential}
There exists a vector potential $\qq: \Din \times Y \to \CCC^d$ such that
\begin{equation}\label{eq_vpq} 
\curly \qq = (\hMB - \MAH) \grad u_0,
\end{equation}
\begin{equation}
\label{eq_curl_qq_L2}
\|(\curlx \qq)^\eps\|_{L^2(\Din)}
\lesssim
|u_0|_{H^2(\Din)}
\end{equation}
and, if $kR_0\gtrsim 1$,
\beq\label{eq_curl_qq_trace_H12}
\N{\curl \qq^\eps \cdot \nnu}_{\HH^{-1/2}(\Gamma)}
\lesssim
(1+(k\eps)^{-1/2}) \left ( k\N{u_0}_{H^{1}(\Din)}+ \N{u_0}_{H^2(\Din)}\right ).
\eeq
\end{lemma}

\bpf[Proof of Lemma \ref{lemma_vector_potential} using Theorem \ref{theorem_vector_potential_main}]
To apply Theorem \ref{theorem_vector_potential_main}, we need to show that 
(i) $ (\hMB - \MAH) \in H^s(Y) \cap L^\infty(Y)$ for some $s>0$, (ii) $(\hMB - \MAH)$ does not depend on $y_d$, and 
(iii) $\divy ((\hMB - \MAH) \bw)=0$ for all $\bw \in \Com^d$ independent of $\by$.

Regarding (i): this is shown under \eqref{eq_def_AH}.
Regarding (ii): $\MAH$ is constant, and therefore independent of $y_d$. $\hMA$ is independent of $y_d$ by Condition \ref{condition_periodic_patterns}, therefore so are the  $\hchi_j$ by Lemma \ref{lemma_condition_patterns}, therefore so is $\hMC$ by \eqref{eq_def_C}, and therefore so is $\hMB$ by \eqref{eq_B}.
Regarding (iii): by the definition of $\hMB$ \eqref{eq_B}, for $\bw \in \Com^d$ independent of $\by$,
\beqs
\divy\big( (\hMB -\MAH) \bw\big) = \divy\big( \hMB\bw\big) = \divy(\hMA\bw) -\divy(\hMA\hMC\bw),
\eeqs
which equals zero by \eqref{eq_hchi_def} and the definition of $\hMC$ \eqref{eq_def_C}. 
\epf

Let $b^\eps(\cdot,\cdot)$ denote the sesquilinear form $b(\cdot,\cdot)$ \eqref{eq:sesqui} with coefficients $\MA_\eps$ and $n_\eps$. Let $\bout(\cdot,\cdot)$ denote the sesquilinear form $b(\cdot,\cdot)$ with integration over $B_R$ replaced by integration over $\Dout$ (and recall that, in $\Dout$, $\MA_\eps=\MI$ and $n_\eps=1$ by \eqref{eq_n_eps_A_eps}).

We now define the boundary corrector $\theta_\eps$ as the unique element
of $H^1(B_R)$ such that
\begin{equation}
\label{eq_boundary_corrector}
b^\eps(\theta_\eps,v)
=
\bout(\EE(u_1^\eps),v) + \big(\curl \qq^\eps \cdot \nnu,v\big)_\Gamma \quad\tfa v\in H^1(B_R),
\end{equation}
where $E$ is the extension operator in Lemma \ref{lem:trace}.
Since the right-hand side of \eqref{eq_boundary_corrector} is an element of
$\left (H^1(B_R)\right )'$, the problem \eqref{eq_boundary_corrector} is indeed well-posed;
the rationale for choosing this particular right-hand side becomes clear in the proof of
Lemma \ref{lemma_homogenization_with_theta} below.

\begin{lemma}[Bound on the norm of the boundary corrector]
\label{lemma_boundary_corrector}
If $\theta_\eps \in H^1(B_R)$ is the boundary corrector defined by
\eqref{eq_boundary_corrector} and $kR_0\gtrsim 1$, then
\begin{align}
\label{eq_estimate_theta}
\|\theta_\eps\|_{H^1_k(B_R)}
&\lesssim
\Clayerh
\big(\N{u_1^\eps}_{H^{1/2}_k(\Gamma)}
+
\N{\curl \qq^\eps  \cdot \nnu}_{H^{-1/2}(\Gamma)}
\big)
\\
\nonumber
&\lesssim
\Clayerh \big(1+(k\eps)^{-1/2}\big) \big( k\N{u_0}_{H^1(\Din)}  + \N{u_0}_{H^2(\Din)} \big).
\end{align}
\end{lemma}

\begin{proof}
Let $\psi \in \left (H^1(B_R)\right )'$ be defined such that
\begin{equation*}
\langle \psi, v \rangle
=
b^\eps(\theta_\eps,v) \quad\tfa v \in H^1(B_R).
\end{equation*}
By \eqref{main_eq_estimate_Cse},
\beqs
\|\theta_\eps\|_{H^1_k(B_R)} \leq \Clayerh \|\psi\|_{(H^{1}_k(B_R))'},
\eeqs
so that it remains to estimate the dual norm of $\psi$.
By the definition of $\theta_\eps$ \eqref{eq_boundary_corrector}, continuity of $\bout$, and Lemma \ref{lem:trace}, for all $v \in H^1(B_R)$, 
\begin{align*}
|\langle \psi, v\rangle |
&
\leq
|b_{\rm out}(\EE(u_1^\eps),v)| + |((\curl \qq^\eps) \cdot \nnu,v)_\Gamma|_.
\\
&
\lesssim
\|\EE(u_1^\eps)\|_{H^1_k(\Dout)}\|v\|_{H^1_k(\Dout)}
+
\|(\curl \qq^\eps) \cdot \nnu\|_{H^{-1/2}(\Gamma)}\|v\|_{H^{1/2}(\Gamma)}
\\
&\lesssim
\left (
\|u_1^\eps\|_{H^{1/2}_k(\Gamma)}
+
\|(\curl \qq^\eps) \cdot \nnu\|_{H^{-1/2}(\Gamma)}
\right )
\|v\|_{H^1_k(B_R)},
\end{align*}
and the  inequality \eqref{eq_estimate_theta} follows. The second inequality then follows from using the bounds \eqref{eq_u1_H12} and \eqref{eq_curl_qq_trace_H12}.
\end{proof}

\subsection{Proof of Theorem \ref{main_theorem_homogenization}}

The proof relies on the following error esimate involving the boundary corrector $\theta_\eps$,
whose proof is postponed until Section \ref{section_homogenization_proof}.

\begin{lemma}[Error estimate with a boundary corrector]
\label{lemma_homogenization_with_theta}
Let 
$u_{1,\eps} \in H^1(B_R)$ be defined by
\begin{equation}
u_{1,\eps} \eq \left \{
\begin{array}{ll}
u_1^\eps & \text{ in } \Din,
\\
\EE(u_1^\eps) & \text{ in } \Dout.
\end{array}
\right .
\label{eq_u1eps_def}
\end{equation}
where the extension operator $E$ is as in Lemma \ref{lem:trace_weight}.
Then,
\begin{equation}
\label{eq_error_estimate_u}
\big\|u_\eps - \big(u_0 + \eps (u_{1,\eps} - \theta_\eps)\big)\big\|_{H^1_k(B_R)}
\lesssim
\Clayerh \eps \left ( \N{u_0}_{H^2(B_R)}+ k \N{u_0}_{H^1_k(B_R)} \right ).
\end{equation}
\end{lemma}

\bre
[Difference between our first-order and boundary correctors and those in \cite{CaGuMo:16}]
In \cite{CaGuMo:16}, the first-order corrector is defined by $u_1^\eps$ in $\Din$ and zero
in $\Dout$, and is therefore not in $H^1(B_R)$ because of the jump on $\Gamma$. Analogously,
the boundary corrector $\theta_\eps$ defined by \cite[Equation 18]{CaGuMo:16} has a jump on
$\Gamma$ and is not in $H^1(B_R)$. In contrast, both our first-order corrector $u_{1,\eps}$
\eqref{eq_u1eps_def} and our boundary corrector $\theta_\eps$ \eqref{eq_boundary_corrector}
are defined so that they are in $H^1(B_R)$.
\ere

\begin{proof}[Proof of Theorem \ref{main_theorem_homogenization}
using Lemma \ref{lemma_homogenization_with_theta}]
By (in this order) the triangle inequality, the definition of
$u_{1,\eps}$ \eqref{eq_u1eps_def}, the triangle inequality again and
the bound \eqref{eq_stability_extension}, the bounds \eqref{eq_error_estimate_u}
and \eqref{eq_estimate_theta}, and the fact that $\Clayerh\gtrsim 1$,
\begin{align*}
&\hspace{-1cm}
k \|u_\eps-u_0- \eps u_1^\eps\|_{H^1_k(\Din)} +k \|u_\eps-u_0\|_{H^1_k(\Dout)}
\\
&\leq
k \|u_\eps-u_0- \eps u_{1,\eps}\|_{H^1_k(B_R)} + k\eps \N{E(u_1^\eps)}_{H^1_k(B_R)},
\\
&\lesssim
k \|u_\eps-u_0- \eps (u_{1,\eps}-\theta_\eps)\|_{H^1_k(B_R)}
+
k\eps \N{\theta_\eps}_{H^1_k(B_R)}
+
k\eps \N{ u_1^\eps}_{H^{1/2}_k(\Gamma)},
\\
&\lesssim
\Clayerh k\eps \left ( \N{u_0}_{H^2(B_R)} + k\N{u_0}_{H^1_k(B_R)} \right )\\
&\qquad\quad+
\Clayerh k\eps
\left (
\N{ u_1^\eps}_{H^{1/2}_k(\Gamma)}
+
\N{\curl \qq^\eps  \cdot \nnu}_{H^{-1/2}(\Gamma)}
\right ).
\end{align*}
The bound \eqref{eq_homo_thm_1} then follows from the bounds \eqref{eq_curl_qq_trace_H12}
and \eqref{eq_u1_H12}; the bound \eqref{eq_homo_thm_2} then follows from the bound
\eqref{eq_u0_H2}.
\end{proof}

\subsection{Proof of Lemma \ref{lemma_homogenization_with_theta}}
\label{section_homogenization_proof}

Since 
\beqs
\|\phi\|_{H^1_k(B_R)}
=
\sup_{\substack{\psi \in (H^1_k(B_R))' \\ \psi \neq 0}}
\frac{\big|\langle \phi,\psi\rangle\big|}{\N{\psi}_{(H^1_k(B_R))'}}
\eeqs
(see, e.g., \cite[Page 76 and Theorem 3.30]{Mc:00}), if we can show that 
\begin{equation}
\label{tmp_duality_argument}
\big|(u_\eps-(u_0+\eps(u_{1,\eps}-\theta_\eps),\psi)_{B_R}\big|
\lesssim
\mathscr M
\|\psi\|_{(H^1_k(B_R))'}
\end{equation}
for all $\psi \in L^2(B_R)$, with
\begin{equation*}
\mathscr M := \Clayerh \eps \left (\N{u_0}_{H^2(B_R)}+k \N{u_0}_{H^1_k(B_R)} \right ),
\end{equation*}
then the bound \eqref{eq_error_estimate_u} follows since  $L^2(B_R)$ is dense in $(H^1(B_R))'$.
With $\psi$ an arbitrary element of $L^2(B_R)$, let $W_\eps \in H^1(B_R)$ be the solution of
the adjoint problem
\begin{equation}
\label{eq_Weps}
b^\eps(v,W_\eps) = (v,\psi)_{B_R} \quad \tfa v \in H^1(B_R).
\end{equation}
We now claim that
\begin{equation*}
b^\eps(\overline{W_\eps},v) = (\overline{\psi},v)_{B_R} \quad \tfa v \in H^1(B_R);
\end{equation*}
indeed, this follows from that fact that $b^\eps(\cdot,\cdot)$ is the sesquilinear form $b(\cdot,\cdot)$ \eqref{eq:sesqui} with coefficients $\MA_\eps$ and $n_\eps$, and the fact that $\langle \DtN \phi, \overline{\psi}\rangle_\Gamma = \langle \DtN \psi, \overline{\phi}\rangle_\Gamma$ (from Green's second identity and the Sommerfeld radiation condition \eqref{eq:src}; see, e.g., \cite[Lemma 6.13]{Sp:15}).
Then, by \eqref{main_eq_estimate_Cse}, 
\begin{equation}
\label{eq_stab_Weps}
\|W_\eps\|_{H^1_k(B_R)} \leq \Clayerh \|\psi\|_{(H^1_k(B_R))'}.
\end{equation}
The following lemmas involve different components of the inner product on the left-hand side
of \eqref{tmp_duality_argument}; we highlight that the first lemma is the point where the
existence of the vector-potential \eqref{eq_vpq} is used.

\begin{lemma}[The terms in the inner product \eqref{tmp_duality_argument} involving $u_\eps-u_0$]
\begin{align}
\nonumber
(u_\eps-u_0,\psi)_{B_R}
=
&-
k^2 ((n^H -\hn^\eps) u_0,W_\eps)_\Din
-
\big((\hMA\hMC)^\eps\grad u_0,\grad W_\eps\big)_\Din
\\
&\quad
+
\eps \big((\curlx \qq )^\eps,\grad W_\eps\big)_\Din
-
\eps \big((\curl \qq^\eps) \cdot \nnu,W_\eps\big)_\Gamma
\label{eq_lemma_ueps_u0}
\end{align}
\end{lemma}

\begin{proof}
By the definition of $W_\eps$ \eqref{eq_Weps},
\begin{equation*}
(u_\eps-u_0,\psi)_{B_R} = b^\eps(u_\eps,W_\eps) - b^\eps(u_0,W_\eps).
\end{equation*}
Since $u_\eps$ and $u_0$ are the solutions to the oscillatory
and homogenized problems, respectively, 
\begin{equation*}
b^\eps(u_\eps,v) = (f,v)_{B_R} = b^H(u_0,v) \quad \tfa v \in H^1(B_R),
\end{equation*}
where $b^H(\cdot,\cdot)$ denotes the sesquilinear form $b(\cdot,\cdot)$ \eqref{eq:sesqui} with coefficients $\MAH$ and $\nH$ in $\Din$ and $\MA=\MI$ and $n=1$ in $\Dout$.
Therefore, since $b^H$ and $b^\eps$ are identical on $\Dout$,
\beq
\label{eq_48_tmp1}
(u_\eps-u_0,\psi)
=
b^H(u_0,W_\eps) - b^\eps(u_0,W_\eps)
=
\big((\MAH - \hMA^\eps) \grad u_0,\grad W_\eps\big)_\Din-k^2 \big((\nH - \hn^\eps) u_0,W_\eps\big)_\Din.
\eeq
Lemma \ref{lemma_L2_term} bounds the second term on the right-hand side; we therefore focus on the first term.
Since $\hMB \eq \hMA(I-\hMC)$, 
$(\MAH - \hMA^\eps) + (\hMA\hMC)^\eps = \MAH - \hMB^\eps$, and thus
\begin{align*}
\big((\MAH - \hMA^\eps) \grad u_0,\grad W_\eps\big)_\Din
+
\big((\hMA \hMC)^\eps \grad u_0,\grad W_\eps\big)_\Din
&=
\big((\MAH - \hMB^\eps) \grad u_0,\grad W_\eps\big)_\Din
\\
&=
\big( (\MAH - \hMB) \grad u_0 )^\eps,\grad W_\eps\big)_\Din.
\end{align*}
By the definition of $\qq$ \eqref{eq_vpq}, $((\MAH-\hMB) \grad u_0)^\eps = -(\curly \qq)^\eps$, and thus
\begin{equation*}
\big((\MAH - \hMA^\eps) \grad u_0,\grad W_\eps\big)_\Din
=
- \big((\hMA \hMC)^\eps \grad u_0,\grad W_\eps\big)_\Din
- \big((\curly \qq)^\eps ,\grad W_\eps\big)_\Din;
\end{equation*}
combining this with \eqref{eq_48_tmp1} yields
\begin{equation}\label{eq_48_tmp2}
(u_\eps-u_0,\psi)_{B_R}
=
-k^2 \big((n^H-\hn^\eps) u_0,W_\eps\big)_\Din 
-\big((\hMA\hMC)^\eps \grad u_0,\grad W_\eps\big)_\Din
- \big((\curly \qq)^\eps,\grad W_\eps\big)_\Din.
\end{equation}
We now focus on the term involving $(\curly \qq)^\eps$.
By \eqref{eq_eps_diff},
\begin{equation*}
\left (\curly \qq\right )^\eps = \eps
\big(
\curl \qq^\eps - \left (\curlx \qq\right )^\eps
\big)
\end{equation*}
and thus, by the divergence theorem, 
\begin{align*}
\big((\curly \qq)^\eps,\grad W_\eps\big)_\Din
&=
\eps \big(\curl \qq^\eps,\grad W_\eps\big)_\Din
-
\eps \big((\curlx \qq )^\eps,\grad W_\eps\big)_\Din
\\
&=
\eps \big((\curl \qq^\eps) \cdot \nnu,W_\eps\big)_{\Gamma}
-
\eps \big( (\curlx \qq)^\eps,\grad W_\eps\big)_\Din,
\end{align*}
where we have crucially used that $\div \curl \qq^\eps =0$;
the result \eqref{eq_lemma_ueps_u0} then follows by combining this last equality with \eqref{eq_48_tmp2}.
\end{proof}

\begin{lemma}[The term in the inner product \eqref{tmp_duality_argument} involving $\eps u_{1,\eps}$]
\begin{align}
\nonumber
-(\eps u_{1,\eps},\psi)_{B_R}
=
&
-\eps\, \bout\big(\EE(u_1^\eps),W_\eps\big)
+
\big((\hMA\hMC)^\eps \grad u_0,\grad W_\eps\big)_{\Din}
\\
\label{eq_lemma_tu1}
&\quad
+
k^2 \eps (\hn^\eps u_1^\eps,W_\eps)_{\Din}
-
\eps \big((\hMA \gradx u_1)^\eps,\grad W_\eps\big)_{\Din}.
\end{align}
\end{lemma}

\begin{proof}
By the definitions of $W_\eps$ \eqref{eq_Weps} and $u_{1,\eps}$ \eqref{eq_u1eps_def}, 
\begin{align*}
-(\eps u_{1,\eps},\psi)_{B_R}=-\eps\, b^\eps(u_{1,\eps},W_\eps)
=
-\eps\, \bout(\EE(u_1^\eps),W_\eps)
-\eps \,\bin^\eps(u_1^\eps,W_\eps),
\end{align*}
where $\bin^\eps(\cdot,\cdot):= b^\eps(\cdot,\cdot)- \bout(\cdot,\cdot)$.
By this definition,
\beqs
-\eps\,\bin^\eps(u_1^\eps,W_\eps) =
-\eps\, (\hMA^\eps \grad u_1^\eps,\grad W_\eps)_\Din + \eps\,k^2 (\hn^\eps u_1^\eps,W_\eps)_\Din.
\eeqs
and so to prove the result \eqref{eq_lemma_tu1}, we only need to prove that 
\beq\label{eq_48_tmp3}
-\eps \big(\hMA^\eps \grad u_1^\eps,\grad W_\eps\big)_\Din =
\big((\hMA\hMC)^\eps \grad u_0,\grad W_\eps\big)_{\Din}
- \eps \big((\hMA \gradx u_1)^\eps,\grad W_\eps\big)_{\Din}.
\eeq
By \eqref{eq_eps_diff} and \eqref{eq_grad_u1}, 
\begin{equation*}
\grad (u_1^\eps)
=
(\gradx u_1)^\eps + \frac{1}{\eps} (\grady u_1)^\eps
=
(\gradx u_1)^\eps - \frac{1}{\eps} \hMC^\eps \grad u_0;
\end{equation*}
therefore,
\begin{equation*}
\hMA^\eps \grad (u_1^\eps)
=
(\hMA\gradx u_1)^\eps - \frac{1}{\eps} (\hMA\hMC)^\eps \grad u_0,
\end{equation*}
and \eqref{eq_48_tmp3} (and hence \eqref{eq_lemma_tu1}) follows. 
\end{proof}

\begin{proof}[Proof of Lemma \ref{lemma_homogenization_with_theta}]
By the discussion at the start of the subsection, it is sufficient to prove
\eqref{tmp_duality_argument}. By the definitions of $W_\eps$ \eqref{eq_Weps}
and $\theta_\eps$ \eqref{eq_boundary_corrector}
\begin{equation}
\label{tmp_boundary_corrector}
\eps (\theta_\eps,\psi)
=
\eps \,b^\eps(\theta_\eps,W_\eps)
=
\eps \,\bout(\EE(u_1^\eps),W_\eps)
+
\eps ((\curl \qq^\eps) \cdot \nnu,W_\eps)_\Gamma.
\end{equation}
We now rewrite 
$(u_\eps - (u_0 + \eps (u_{1,\eps}-\theta_\eps)),\psi)_{B_R}$
using \eqref{eq_lemma_ueps_u0},
\eqref{eq_lemma_tu1}, and \eqref{tmp_boundary_corrector}, to obtain that
\begin{align*}
&(u_\eps - (u_0 + \eps (u_{1,\eps}-\theta_\eps)),\psi)_{B_R}\\
&=
-
k^2 \big((n^H -\hn^\eps) u_0,W_\eps\big)_\Din
-
\big((\hMA\hMC)^\eps\grad u_0,\grad W_\eps\big)_\Din
+
\eps \big( (\curlx \qq)^\eps,\grad W_\eps\big)_\Din
\\
&\quad-
\eps \big((\curl \qq^\eps) \cdot \nnu,W_\eps\big)_\Gamma
- \eps \,\bout(E(u_1^\eps),W_\eps) + \big((\hMA\hMC)^\eps \grad u_0,\grad W_\eps\big)_{\Din}
+k^2 \eps \big(\hn^\eps u_1^\eps,W_\eps\big)_{\Din}\\
&\quad - \eps \big((\hMA \gradx u_1)^\eps,\grad W_\eps\big)_{\Din} +\eps \,\bout(\EE(u_1^\eps),W_\eps) + \eps \big((\curl \qq^\eps) \cdot \nnu,W_\eps\big)_\Gamma
\\
&=
-
k^2 ((n^H -\hn^\eps) u_0,W_\eps)_\Din
+
\eps (\left (\curlx \qq\right )^\eps,\grad W_\eps)_\Din
-k^2 \eps (\hn^\eps u_1^\eps,W_\eps)_{\Din} + \eps ((\hMA \gradx u_1)^\eps,\grad W_\eps)_{\Din}.
\end{align*}
In particular, we see that the terms $\eps ((\curl \qq^\eps) \cdot \nnu,W_\eps)_\Gamma$ 
and $\eps\, \bout(\EE(u_1^\eps),W_\eps)$, present in the right-hand sides of
\eqref{eq_lemma_ueps_u0} and \eqref{eq_lemma_tu1}, respectively, have been removed thanks
to the definition of the boundary corrector $\theta_\eps$ \eqref{eq_boundary_corrector} and
its consequence \eqref{tmp_boundary_corrector}.

We now bound each of the four terms on the right-hand side. By \eqref{eq_L2_term}, 
\begin{equation*}
k^2\big|\big((\nH-\hn^\eps) u_0,W_\eps\big)_\Din\big|
\lesssim
k\eps \N{u_0}_{H^1_k(B_R)}\|W_\eps\|_{H^1_k(\Din)},
\end{equation*}
and by \eqref{eq_curl_qq_L2},
\begin{equation*}
\eps \big|\big( (\curlx \qq)^\eps,\grad W_\eps\big)_\Din\big|
\lesssim
\eps \N{u_0}_{H^2(B_R)} \|W_\eps\|_{H^1_k(\Din)}.
\end{equation*}
Using the fact that $\|\hn^\eps\|_{L^\infty(Y)} \lesssim 1$
(by Definition \ref{definition_oscillatory_coefficients}), 
the Cauchy--Schwarz inequality, and the bound \eqref{eq_u1_L2}, we have
\begin{equation*}
k^2 \eps \big|(\hn^\eps u_1^\eps,W_\eps)_{\Din}\big|
\lesssim
\eps \big(k\|u_1^\eps\|_{L^2(\Din)}\big)\big(k\|W_\eps\|_{L^2(\Din)}\big)
\lesssim
\eps\big(\N{u_0}_{H^2(B_R)}+k\N{u_0}_{H^1_k(B_R)}\big)\|W_\eps\|_{H^1_k(\Din)}.
\end{equation*}
Similarly, using that $\|\hMA\|_{L^\infty(Y)} \lesssim 1$
(again by Definition \ref{definition_oscillatory_coefficients}),
and the bound \eqref{eq_u1_L2}, we have
\begin{align*}
\eps |((\hMA \gradx u_1)^\eps,\grad W_\eps)_\Din|
&\leq
\eps \|\hMA^\eps\|_{L^\infty(\Din)}
\|(\gradx u_1)^\eps\|_{L^2(\Din)}|W_\eps|_{H^1(\Din)}\\
&\lesssim
\eps \big(\N{u_0}_{H^2(B_R)}+k\N{u_0}_{H^1_k(B_R)}\big) \|W_\eps\|_{H^1_k(\Din)}.
\end{align*}
Combining these bounds, we have
\begin{equation*}
|(u_\eps-(u_0+\eps(u_{1,\eps}-\theta_\eps)),\psi)_{B_R}|
\lesssim
\eps \big(\N{u_0}_{H^2(B_R)}+k\N{u_0}_{H^1_k(B_R)}\big) \|W_\eps\|_{1,k,\Din}
\end{equation*}
and then using \eqref{eq_stab_Weps} we obtain \eqref{tmp_duality_argument}.
\end{proof}

\subsection{Proof of Corollary \ref{corollary_stability}}

By Lemma \ref{lem:H1}, it is sufficient to prove that,
if $F(v)= \int_{B_R}f\, \overline{v}$, then
\begin{equation}
\label{eq_star_cor1}
\N{u_\eps}_{H^1_k(B_R)}
\lesssim
\frac{\CsHshort}{k}
\bigg[ 1 +\Clayer\Big( (k\eps)^{1/2} + k\eps\Big)\bigg]
\N{f}_{L^2(B_R)}.
\end{equation}
By the triangle inequality, Theorem \ref{main_theorem_homogenization}, and the definition
of $\CsHshort$ \eqref{eq_Csol_def}, 
\begin{align}
\nonumber
&\N{u_\eps}_{H^1_k(B_R)}
\leq
\N{u_\eps-u_0}_{H^1_k(B_R)} + \N{u_0}_{H^1_k(B_R)},
\\
\label{eq_star_cor2}
&\quad\leq
\frac{ \CsHshort}{k} \Cfournew\Clayer\Big( (k\eps)^{1/2} + k\eps\Big) \N{f}_{L^2(B_R)}
+
\eps\N{u_1^\eps}_{H^1_k(\Din)}
+
\frac{\CsHshort}{k} \N{f}_{L^2(B_R)},
\end{align}
The bound on $u_1^\eps$ \eqref{eq_u1_H1} and the bound \eqref{eq_u0_H2} imply that
\begin{equation}
\label{eq_star_cor3}
\eps\N{u_1^\eps}_{H^1_k(\Din)}
\lesssim
\eps \big( 1 + (k\eps)^{-1}\big)\Big(|u_0|_{H^2(\Din)} + k \N{u_0}_{H^1_k(B_R)}\Big)
\lesssim
\frac{\CsHshort}{k} \big(1 + k\eps\big)\N{f}_{L^2(B_R)};
\end{equation}
the bound  \eqref{eq_star_cor1} then follows by combining \eqref{eq_star_cor2} and \eqref{eq_star_cor3}.

\appendix

\section{Regularity of periodic cell problems}
\label{appendix_cell_problems}

The goal of this section is to show that the functions $\hchi_j$ and $\hmu$
appearing in the homogenization section have the required regularity, namely
that $\hchi_j \in H^{1+s}(Y) \cap C^1(\PY)$, for some $s>0$, and that $\hmu \in C^1(\PY)$.
Recall that $H^1_\sharp(Y)$ is the subspace of $H^1_\per(Y)$ consisting of
functions with zero mean.

\begin{lemma}
\label{lemma_cell_problems}
Given $f \in \left (H^1(Y)\right )'$ with $\langle f, 1 \rangle = 0$,
there exists a unique element $u \in H^1_\sharp(Y)$ such that, in a distributional sense,
\begin{equation}
\label{eq_poisson_Y}
-\div \left (\hMA \grad u \right ) = f \text{ in } Y.
\end{equation}
Furthermore, if $\hMA(\by)$ does not depend on $y_\ell$ for some $\ell \in \{1,\dots,d\}$ and
\begin{equation}
\label{eq:appB2}
\langle f, v \rangle = \int_Y f_0 v + \sum_{j=1}^d \int_Y f_j \frac{\partial v}{\partial y_j}
\quad
\tfa v \in H^1(Y)
\end{equation}
where $f_0,f_j \in L^2(Y)$ are functions that are independent of $y_\ell$, then
$u$ does not depend on $y_\ell$.
\end{lemma}

\begin{proof}
Given $f \in \left (H^1(Y)\right )'$, by the Poincar\'e-Wirtinger inequality and the Lax-Milgram
lemma applied in the space $H^1_\sharp(Y)$, there exists a unique $u \in H^1_\sharp(Y)$ such that
\begin{equation}\label{eq:A1}
\big(\hMA \grad u,\grad v\big)_Y = \big\langle f, v \big\rangle \quad \tfa v \in H^1_\sharp(Y).
\end{equation}
Now let $\phi \in C^\infty_c(Y)$, with $\phi_0 \in \mathbb C$ the mean value
of $\phi$. Then $\widetilde \phi = \phi-\phi_0 \in H^1_\sharp(Y)$, so that, by  \eqref{eq:A1} and the fact that 
$\langle f, \phi_0 \rangle = 0$,
\begin{equation*}
\big(\hMA \grad u,\grad \phi\big)_Y
=
\big(\hMA \grad u,\grad \widetilde \phi\big)_Y
=
\big\langle f, \widetilde \phi \big\rangle
=
\big\langle f, \phi \big\rangle;
\end{equation*}
the result \eqref{eq_poisson_Y} follows.
Since functions in $H^1_\sharp(Y)$ can be periodically extended, the translation
operation $\tau_{h,\ell}: v(\by) \to v(\by+h \be_\ell)$ is well-defined.
Then, with the summation convention,
\begin{align*}
\big(\hMA \grad (\tau_{h,\ell}u),\grad v\big)_Y
=
\big(\hMA \grad u,\grad (\tau_{-h,\ell}v)\big)_Y
&=
\big(f_0,\tau_{-h,\ell}v\big)_Y + \big(f_j,\partial_j (\tau_{-h,\ell}v)\big)_Y
\\
&=
\big(\tau_{h,\ell}f_0,v\big)_Y + \big(\tau_{h,\ell}f_j,\partial_jv\big)_Y
=
\big(f_0,v\big)_Y + \big(f_j,\partial_jv\big)_Y,
\end{align*}
so that $u = \tau_{h,\ell} u$ for all $h \in \mathbb R$; therefore $u$
does not depend on $y_\ell$.
\end{proof}

\begin{lemma}
\label{lemma_condition_patterns}
If Condition \ref{condition_periodic_patterns} holds, then $\hchi_j$, $j=1,2,3,$
defined by \eqref{eq_hchi_def} exists, is unique, and satisfies
$\hchi_j \in H^{1+s}(Y) \cap C^1(\PY)$, for some $s>0$, and $\hmu$ defined by \eqref{eq_mu}
exists, is unique, and satisfies $\hmu \in C^1(\PY)$. Furthermore,
if $\hMA$ and $\hn$ do not depend on $y_{\ell}$, then neither do $\hchi_j$ and $\hmu$.
\end{lemma}

\begin{proof}
We first claim that both $\hchi_j$ and $\mu$ satisfy equations of the form  \eqref{eq_poisson_Y},
\eqref{eq:appB2} with $\langle f,1\rangle=0$. For $\mu$ this is immediate from \eqref{eq_mu} and
the fact that, by the definition of $\nH$ \eqref{eq_def_nH}, $(\nH-\hn,1)_Y = 0$. For $\hchi_j$,
observe that, by its definition \eqref{eq_hchi_def}, for all $v\in H^1(Y)$,
\begin{equation*}
-\left\langle \frac{\partial (\hMA)_{j \ell}}{\partial y_\ell},v \right\rangle
=
\left((\hMA)_{j \ell},\frac{\partial v}{\partial y_\ell}\right)_Y, 
\qquad
\text{ so that }
-\left\langle \frac{\partial (\hMA)_{j \ell}}{\partial y_\ell},1 \right\rangle = 0.
\end{equation*}
Therefore, by Lemma \ref{lemma_cell_problems}, both $\hchi_j$ defined by \eqref{eq_hchi_def} 
and $\mu$ defined by \eqref{eq_mu} exist and are unique. The lemma also implies
that if $\hMA$ and $\hn$ do not depend on $y_{\ell}$, then neither do $\hchi_j$ and $\hmu$.

We now give the details of the proof of the regularity for $\hchi_j$;
the proof of the regularity of $\hmu$ is very similar. Our goal is to use
\ben
\item that, for functions depending on $r$ variables, then, if $t>0$, there exists $C>0$ such that 
\beq\label{eq_Sobolev}
\N{v}_{C^{0,t}(Y)}\lesssim \N{v}_{H^{r/2+t}(Y)}
\eeq
i.e., Sobolev embedding (see, e.g., \cite[Theorem 3.26]{Mc:00}), and 
\item an elliptic regularity shift.
\een
Regarding 2.: such a shift is well-known, at least in the non-periodic case.
To go from the periodic case to the non-periodic case, we extend the functions $u$,
$f_0$ and $f_j$ in \eqref{eq_poisson_Y} and \eqref{eq:appB2} by periodicity and multiply
by a cutoff function, and obtain a problem posed in a bounded domain in
$\Omega \subset \mathbb R^d$ with $Y \subset \subset \Omega$. 
We can then just use the standard interior regularity shift result; see, e.g., \cite[Theorem 4.20]{Mc:00}.

We need to check that \emph{both} the coefficient $\hMA$ of the PDE \eqref{eq_hchi_def} \emph{and} the right-hand side
$-\partial (\hMA)_{j \ell}/\partial y_\ell$ have sufficient regularity under
either (a) or (b) in Condition \ref{condition_periodic_patterns}.

If (a) holds, then both $\hMA$ and $\hchi_j$ only depend on $y_1$, so that \eqref{eq_hchi_def} becomes
\begin{equation*}
-\frac{\rd}{\rd y_1} \left ((\hMA)_{11}\frac{\rd \hchi_j}{\rd y_1}\right ) = -\frac{\rd (\hMA)_{11}}{\rd y_1}.
\end{equation*}
Therefore, if $\hMA$ is piecewise $C^{0,1}$, then $\rd(\hMA)_{11}/\rd y_1$ is
piecewise $L^2$, and $\hchi_j$ is piecewise $H^2$ by elliptic regularity (see, e.g., \cite[Theorem 4.20]{Mc:00}). 
By \eqref{eq_Sobolev}, $\hchi_j$ is $C^0$ and piecewise $C^1$.
By the definition of the Slobodeckii seminorm, if $0\leq s< 1/2$, then a function that is locally $H^s$ is in $H^s(Y)$. 
Therefore, $\grad\hchi_j \in H^s(Y)$ for $s< 1/2$; thus $\chi_j \in H^{1+s}(Y)$ for $s<1/2$.

Finally, for (b), 
by integrating by parts piecewise  the right-hand side of the variational problem defining $\hchi_j$, we find that, for all $v\in H^1_\sharp(Y)$,
\begin{align*}
(\hMA \grad \hchi_j,\grad v)_{Y}
=
\sum_{\omega \in \PY}
\left ((\hMA)_{j\ell},\frac{\partial v}{\partial y_\ell}\right )_{\omega}
&=
\sum_{\omega \in \PY}
\left \{
\left ((\hMA)_{j\ell} \nnu_{\omega,\ell},v\right )_{\partial \omega}
-
\left (\frac{\partial (\hMA)_{j\ell}}{\partial y_\ell},v\right )_{\omega}
\right \}
\\
&=
\sum_{\Gamma \in \partial\PY}
\left ([\![(\hMA)_{j\ell}]\!] \nnu_{\Gamma,\ell},v\right )_{\Gamma}
-
\sum_{\omega \in \PY}
\left (\frac{\partial (\hMA)_{j\ell}}{\partial y_\ell},v\right )_{\omega}
\end{align*}
where $\partial \PY$ denotes the set of interfaces of the partition, and $[\![\cdot ]\!]$ denotes the jump of a function across an interface $\Gamma$.
By elliptic regularity (see, e.g. \cite[Theorem 4.20]{Mc:00}), since each $\Gamma$ is $C^{2,1}$, the periodic extension of $\hMA$ is piecewise $C^{1,1}$ and hence piecewise $H^2$, we have
\begin{equation*}
\|\hchi_j\|_{H^3(\PY)}
\lesssim
\left \|\frac{\partial (\hMA)_{j\ell}}{\partial y_\ell}\right \|_{H^1(\PY)}
+
\sum_{\Gamma \in \partial \PY} \|[\![(\hMA)_{j\ell}]\!] \cdot \nnu_{\ell,\Gamma}\|_{H^{3/2}(\Gamma)}
\lesssim
\|\hMA\|_{H^2(\PY)},
\end{equation*}
by continuity of the trace map $H^2(D)\to H^{3/2}(D)$ for a $C^{1,1}$ domain $D$ (see, e.g., \cite[Theorem 3.37]{Mc:00}).
By \eqref{eq_Sobolev} with $r=2$, $\grad\hchi_j \in C^{0,t}(\PY)$ for all $0<t<1$, and thus certainly $\hchi_j \in C^1(\PY)$.
Finally, exactly as in (a), since $\grad\hchi_j$ is locally $H^s$ for $s\leq 2$, it is in $H^s(Y)$ for $s< 1/2$; thus $\chi_j \in H^{1+s}(Y)$ for $s<1/2$.
\end{proof}

\section{Periodic vector potentials}
\label{appendix_vector_potentials}

The goal of this appendix is to prove Theorem \ref{theorem_vector_potential_main} below,
from which Lemma \ref{lemma_vector_potential} follows.

\subsection{Periodic Sobolev spaces}

In this section $\CC^\infty(\RRR^d)$ is the space of smooth complex vector-valued functions;
i.e., $\CC^\infty(\RRR^d) \eq C^\infty(\Rea^d;\Com^d)$. $Y = (0,1)^d$ is the $d$-dimensional unit
cube. The space $\CC^\infty_\per(Y)$ consists of those functions in $\CC^\infty(\RRR^d)$ that
are $Y$-periodic, i.e. $\pphi \in \CC^\infty(\RRR^d)$ and $\pphi(\yy + \ee_j) = \pphi(\yy)$
for all $\yy \in \RRR^d$, where, for $1 \leq j \leq d$, $\ee_j$ denotes the $j^{\rm th}$ vector
in the canonical basis of $\RRR^d$. For $\pphi \in \CC_\per^\infty(Y)$, 
\begin{align*}
&\|\pphi\|_{\ddiv,Y}^2   \eq \|\pphi\|_{\LL^2(Y)}^2 + \|\divy \pphi\|_{\LL^2(Y)}^2, \qquad
\|\pphi\|_{\ccurl,Y}^2  \eq \|\pphi\|_{\LL^2(Y)}^2 + \|\curly \pphi\|_{\LL^2(Y)}^2,
\\
&\hspace{3.5cm}\|\pphi\|_{\HH^{1}(Y)}^2       \eq \|\pphi\|_{\LL^2(Y)}^2 + \|\grady \pphi\|_{\LL^2(Y)}^2,
\end{align*}
and $\HH_\per(\ddiv,Y)$, $\HH_\per(\ccurl,Y)$
and $\HH^1_\per(Y)$ are the closures of $\CC^\infty_\per(Y)$ in $\LL^2(Y)$ with respect to the $\|\cdot\|_{\ddiv,Y}$, $\|\cdot\|_{\ccurl,Y}$
and $\|\cdot\|_{\HH^{1}(Y)}$ norms, respectively.

The spaces $\CC^\infty_\sharp(Y)$, $\HH_\sharp(\ddiv,Y)$,
$\HH_\sharp(\ccurl,Y)$ and $\HH^1_\sharp(Y)$ are the subspaces of 
$\CC^\infty_\per(Y)$, $\HH_\per(\ddiv,Y)$, $\HH_\per(\ccurl,Y)$ and $\HH^1_\per(Y)$, respectively,
consisting of functions with zero mean. Since the map  $\LL^2(Y) \rightarrow  \mathbb C^d$ defined by 
$\pphi \mapsto \frac{1}{|Y|}\int_Y \pphi $
is bounded, $\CC^\infty_\sharp(Y)$ is dense in $\HH_\sharp(\ddiv,Y)$, $\HH_\sharp(\ccurl,Y)$,
and $\HH^1_\sharp(Y)$. 

\subsection{Fourier series}\label{sec:FS}

For $\al \in \ZZZ^d$, let
\begin{equation*}
|\al|^2 \eq \sum_{j=1}^d \alpha_j^2
\quad \tand \quad
\re^\al(\yy) \eq \re^{2\pi \ri\al \cdot \yy} \quad \tfa \yy \in \RRR^d.
\end{equation*}
For $\pphi \in \LL^2_\per(Y)$, we  define $\pphi^\al \in \ZZZ^d$ by
\begin{equation*}
(\pphi^\al)_j \eq \big(\phi_j,\re^\al\big)_{\LL^2(Y)}\quad\tfa \al \in \ZZZ^d.
\end{equation*}
Then, by the $\LL^2$ theory of Fourier series (see, e.g., \cite[\S4.26, Page 91]{Ru:86}
for the results in one dimension, and, e.g., \cite[Prop.~3.1.15 and 3.1.16, Page 170]{Gr:08}
for the results in arbitrary dimensions) $\sum_{\al \in \ZZZ^d} |\pphi^\al|^2 <\infty$ and 
\begin{equation*}
\pphi = \sum_{\al \in \ZZZ^d} \pphi^\al \re^\al,
\end{equation*}
where the sum converges in $\LL^2(Y)$.
Conversely, if $\{\tilde{\ppsi}^\al\}_{\al \in \ZZZ^d}$ is a set of coefficients
satisfying $\sum_{\al} |\tilde{\ppsi}^\al|^2 <\infty$ then there exists a unique $\ppsi\in \LL^2(Y)$ such that $\ppsi^\al= \tilde{\ppsi}^\al$ and 
\begin{equation*}
\ppsi \eq \sum_{\al \in \ZZZ^d} \ppsi^\al \re^\al,
\end{equation*}
where the sum converges in $\LL^2(Y)$.
If $\pphi \in \HH_\per(\ddiv,Y)$ then, by the divergence theorem (see, e.g., \cite[Part 2 of Theorem 3.24]{Mo:03})
\begin{equation}\label{eq_Fourier_div}
(\divy \pphi)^\al  = 2 \pi \ri \,\al \cdot \pphi^\al \quad\text{ so that }\quad
\divy  \pphi = 2 \pi \ri  \sum_{\al \in \ZZZ^d} (\al \cdot  \pphi^\al) \re^\al.
\eeq
Similarly if $\pphi\in \HH_\per(\ccurl,Y)$ then integration by parts (see, e.g., \cite[Theorem 3.29]{Mo:03}) gives 
\begin{equation}\label{eq_Fourier_curl}
(\curly \pphi)^\al = 2 \pi \ri \,\al \times \pphi^\al\quad\text{ so that }\quad
\curly \pphi = 2 \pi \ri  \sum_{\al \in \ZZZ^d} (\al \times \pphi^\al) \re^\al.
\end{equation}

\subsection{Sobolev norms defined by Fourier coefficients}

For $s \in \mathbb R_+$, we define the Sobolev semi-norms on vector- and scalar-valued functions by
\begin{equation}
\label{eq_fourier_sobolev_norm}
|\pphi|_{\HH^s(Y)}^2 := (2\pi)^s \sum_{\al \in \ZZZ^d} |\al|^{2s} |\pphi^\al|^2
\quad\tand\quad
|\phi|_{H^s(Y)}^2 := (2\pi)^s \sum_{\al \in \ZZZ^d} |\al|^{2s} |\phi^\al|^2;
\end{equation}
where $|\cdot|_2$ denotes the Euclidean norm on $\Com^d$ (and the modulus on $\Com$);
observe that then
\beq
\label{eq_norm_scalar_vector}
|\pphi|^2_{\HH^s(Y)} = \sum_{j=1}^d |\phi_j|^2_{H^s(Y)}
\eeq
Let 
\beq
\N{\pphi}^2_{\HH^m(Y)} := \sum_{s=0}^m |\pphi|^2_{\HH^s(Y)},
\text{ and } 
\label{eq_norm_not_int}
\N{\pphi}^2_{\HH^t(Y)} := \sum_{s=0}^{\lfloor t \rfloor} |\pphi|^2_{\HH^s(Y)} + |\pphi|^2_{\HH^t(Y)},
\eeq
for $m \in \ZZZ^+$ and $t \in \Rea^+\setminus \ZZZ^+$,
and similarly for scalar-valued functions. We use below the particular consequence of \eqref{eq_norm_scalar_vector} and \eqref{eq_norm_not_int} that
\beq\label{eq_norm_scalar_vector2}
\N{\pphi}^2_{\HH^{1+s}(Y)} =\sum_{j=1}^d \N{\phi_j}^2_{H^{1+s}(Y)}.
\eeq

\begin{lemma}
\label{lem:A1}
Let $0\leq s \leq 1$. If $\pphi \in \HH_\per(\ddiv,Y)$ with $\divy \pphi = 0$
and $\curly \pphi \in \HH^{s}(Y)$, then $\pphi \in \HH^{1+s}(Y)$ with
\begin{equation}
\label{eq_Hs_divy_free}
|\pphi|_{\HH^{1+s}(Y)}
\leq
|\curly \pphi|_{\HH^s(Y)}.
\end{equation}
\end{lemma}

An immediate consequence of Lemma \ref{lem:A1} is that
\beq\label{eq_H1}
\text{ if } \,\,\pphi \in \HH_\per(\ddiv, Y) \cap \HH_\per(\ccurl, Y) \,\,\text{  then  }\,\, \pphi \in \HH^1_\per(Y).
\eeq

\begin{proof}[Proof of Lemma \ref{lem:A1}]
The identity
\begin{equation*}
|\al|^2 \vv = (\al \cdot \vv) \al + (\al \times \vv) \times \al,
\end{equation*}
holds for all $\al \in \ZZZ^d$ and $\vv \in \CCC^d$.
Using this with $\vv = \pphi^\al$, along with \eqref{eq_Fourier_div} and \eqref{eq_Fourier_curl}, we have
\begin{align*}
2 \pi \ri |\al|^2 \pphi^\al
=
2 \pi \ri (\al \cdot \pphi^\al) \al + 2 \pi \ri (\al \times \pphi^\al) \times \al
=
(\divy \pphi)^\al \al + (\curly \pphi)^\al \times \al
=
(\curly \pphi)^\al \times \al,
\end{align*}
since $\divy \pphi = 0$ by assumption. Therefore
\begin{equation*}
2\pi |\al| |\pphi^\al|
=
\frac{|(\curly \pphi)^\al \times \al|}{|\al|}
\leq
|(\curly \pphi)^\al|.
\end{equation*}
Recalling the definition \eqref{eq_fourier_sobolev_norm}, we have
\begin{align*}
|\pphi|_{\HH^{1+s}(Y)}^2
=
(2\pi)^{2(1+s)}
\sum_{\al \in \ZZZ^d} |\al|^{2(1+s)}|\pphi^\al|^2
&=
(2\pi)^{2s} \sum_{\al \in \ZZZ^d} |\al|^{2s} (2\pi|\al||\pphi^\al|)^2\\
&\leq
(2\pi)^{2s} \sum_{\al \in \ZZZ^d} |\al|^{2s} |(\curly \pphi)^\al|^2
=
|\curly \pphi|_{\HH^s(Y)}^2.
\end{align*}
\end{proof}

\subsection{Univariate vector potentials}

\begin{lemma}
\label{lemma_smooth_vector_potential}
For all $\pphi \in \HH_{\sharp}(\ddiv,Y)$ with $\divy \pphi = 0$,
there exists a unique $\qq \in \HH^1_\sharp(Y)$ such that
\begin{equation}
\label{eq_vector_potential}
\curly \qq = \pphi \quad \tand \quad \divy \qq = 0.
\end{equation}
In addition,
\begin{equation}
\label{eq_stability_vector_potential}
\|\qq\|_{\HH^{1}(Y)} \leq \left (1 + \pi^{-1}\right ) \|\pphi\|_{\LL^2(Y)}.
\end{equation}
\end{lemma}

The analogue of Lemma \ref{lemma_smooth_vector_potential} in bounded domains appears in,
e.g., \cite[Theorem 3.4]{GiRa:86}.

\begin{proof}[Proof of Lemma \ref{lemma_smooth_vector_potential}]
Under the assumption that the solution $\qq \in \HH^1_\sharp(Y)$ exists and is unique,
the inequality \eqref{eq_Hs_divy_free} with $s = 0$ shows that
\begin{equation*}
|\qq|_{\HH^{1}(Y)} \leq \|\curly \qq\|_{\LL^2(Y)} = \|\pphi\|_{\LL^2(Y)}.
\end{equation*}
Since the mean value of $\qq$ vanishes and $Y$ is convex, the Poincar\'e inequality
\begin{equation*}
\|\qq\|_{\LL^2(Y)} \leq \pi^{-1}|\qq|_{\HH^{1}(Y)},
\end{equation*}
holds, and \eqref{eq_stability_vector_potential} immediately follows.
We now show that if the solution to \eqref{eq_stability_vector_potential}
exists, it is necessarily unique. It is sufficient to prove that if $\pphi = \boldsymbol 0$, then $\qq=\boldsymbol0$. If $\pphi = \boldsymbol 0$ then 
$\curly \qq= \boldsymbol 0$, and 
$|\qq|_{\HH^{1}(Y)} = 0$ by \eqref{eq_Hs_divy_free}; therefore
$\qq$ is constant. But since $\qq \in \HH^1_\sharp(Y)$, its mean value
vanishes, showing that $\qq = \boldsymbol 0$.
It remains to prove existence; given $\pphi \in \HH_\sharp(\ddiv,Y)$,
define the vector $\qq^\al$ by 
\begin{equation}\label{eq_Fourier_coeff}
\qq^{\boldsymbol 0} := \boldsymbol 0, \qquad
\qq^\al \eq -\frac{1}{2 \pi \ri }\frac{\al \times \pphi^{\al}}{|\al|^2}, \quad 
\al \in \ZZZ^d \setminus \{\boldsymbol 0\}.
\end{equation}
Since $\sum_{\al \in \ZZZ^d} |\qq^\al|^2 \leq (2\pi)^{-1} \sum_\al |\pphi^\al|^2<\infty$,
the results recapped in \S\ref{sec:FS} imply that
\begin{equation*}
\qq \eq \sum_{\al \in \ZZZ^d} \qq^\al \re^\al \in \LL^2(Y).
\end{equation*}
Since $(\qq,1)_Y = \qq^{\boldsymbol 0} = \boldsymbol 0$,
$\qq \in \LL^2_\sharp(Y)$. It remains to show that $\qq$ satisfies \eqref{eq_vector_potential},
since then the fact that $\qq \in \HH^1_\sharp(Y)$ will then follow from $\pphi\in\LL^2(Y)$
and the relation \eqref{eq_H1}.

First, since $\al \cdot (\al \times \vv) = 0$ for all $\vv \in \mathbb{C}^d$,
$\al\cdot \qq^\al = 0$ for all $\al \in \ZZZ^d$; therefore \eqref{eq_Fourier_div}
implies that $\divy \qq = 0$. 
Next, the fact that $\divy \pphi = 0$ implies by \eqref{eq_Fourier_div} that $\al \cdot \pphi^\al = 0$ 
for all $\al \in \ZZZ^d \setminus \{\boldsymbol 0\}$. Therefore,
\begin{equation*}
|\al|^2 \pphi^\al
=
(\al \cdot \pphi^\al) \al + (\al \times \pphi^\al) \times \al
=
(\al \times \pphi^\al) \times \al
=
2 \pi \ri |\al|^2 (\al \times \qq^\al),
\end{equation*}
and therefore
\begin{equation}\label{eq:EAS1}
2 \pi \ri  (\al \times \qq^\al)=\pphi^\al 
\end{equation}
for all $\al \in \ZZZ^d \setminus \{\boldsymbol 0\}$.
For $\al = \boldsymbol 0$, 
\begin{equation}\label{eq:EAS2}
2 \pi \ri  (\al \times \qq^\al) = \boldsymbol 0 = \pphi^\al,
\end{equation}
since $\pphi \in \HH_\sharp(\ddiv,Y)$.
Combining \eqref{eq:EAS1}, \eqref{eq:EAS2}, and 
\eqref{eq_Fourier_curl}, we find that $\curly \qq = \pphi$,
which concludes the proof.
\end{proof}

The next lemma involves the $\LL^\infty(Y)$ and $L^\infty(Y)$ norms defined by 
\beq\label{eq:Linftynorm}
\N{\qq}_{\LL^\infty(Y)}:= \esssup_{\yy \in Y} |\qq(\yy)|_2
\quad
\tand
\quad
\N{q}_{L^\infty(Y)}:= \esssup_{\yy \in Y} |q(\yy)|.
\eeq

\begin{theorem}
\label{theorem_vector_potential}
Let $\pphi \in \HH_\sharp(\ddiv,Y)$ with $\divy \pphi = 0$,
and assume that $\pphi$ does not depend on the $y_d$ variable.

(i) If $\qq \in \HH^1_\sharp(Y)$ is the unique solution
to \eqref{eq_vector_potential}, then $\qq$ does not depend on
the $y_d$ variable.

(ii) If, in addition, $\pphi \in \HH^s(Y)$ for some $s > 0$,
then $\qq \in \LL^\infty(Y)$ and there exists $\Cemb>0$ such that
\begin{equation}
\label{eq_vector_potential_Linfty}
\|\qq\|_{\LL^\infty(Y)} \leq \Cemb \left (1 + \pi^{-1}\right )\|\pphi\|_{\HH^s(Y)}.
\end{equation}
\end{theorem}

\begin{proof}
(i) 
We consider the case $d=3$; the case $d=2$ is very similar (and easier).
For $\yy \in \Rea^3$, we write $\yy = (\yy',y_3)$ for $\yy'\in Y' = (0,1)^2$.
Similarly, for $\ppsi\in \Rea^3$ we write $\ppsi = (\ppsi',\psi_3)$.
We introduce the notation that, for $\ppsi'(\yy')$ a two-dimensional vector independent of $y_3$,
$\ddiv_{\yy'} \ppsi' := \partial_{y_1} \psi'_1 + \partial_{y_2} \psi'_2$
and $\scurl_{\yy'} \ppsi' := \partial_{y_1} \psi'_2 - \partial_{y_2} \psi'_1$.
Furthermore, if $\psi_3(\yy')$ is a scalar field that does not depend on $y_3$, then
$\ccurl_{\yy'} \psi_3 := (\partial_{y_2} \psi_3,-\partial_{y_1} \psi_3)$ (i.e.~$\ccurl_{\yy'}$
is the rotation of the gradient of $\psi$).  These definitions imply that, if
$\ppsi = (\ppsi',\psi_3)$ is a vector field that does not depend on $y_3$, then 
\beq\label{eq:EAS3}
\ccurl_{\yy} \ppsi = (\ccurl_{\yy'} \psi_3,\scurl_{\yy'} \ppsi').
\eeq
Now, since $\pphi$ is independent of $y_3$ and  $\divy \pphi = 0$, $\ddiv_{\yy'} \pphi' =0$.
Observe that, with the notation $\pphi =(\pphi',\phi_3)$, both $\pphi'$ and $\phi_3$ have
mean value zero. Using the lower-dimensional analogues of the Fourier-series argument
used to prove Lemma \ref{lemma_smooth_vector_potential}, one can show that there exists
a unique $p_d \in H_\sharp(\ccurl, Y')$ and a unique $\pp' \in \HH_\sharp(\scurl, Y')$
such that
\begin{equation}
\label{eq_Sunday1}
\vcurl_{\yy'} p_3 = \pphi',
\qquad
\left \{
\begin{array}{rcl}
\scurl_{\yy'} \pp' &=& \phi_3,
\\
\ddiv_{\yy'} \pp' &=& 0.
\end{array}
\right .
\end{equation}
Indeed, the Fourier coefficients of $p_3$ and $\pp'$ are given by 
\beqs
(p_3)^\al =
\begin{cases} 
|\al|^{-2}  (\alpha_2, -\alpha_1)\cdot (\pphi')^\al, & \al \neq {\bf 0},\\
0, & \al  ={\bf 0},
\end{cases}
\quad
\tand
\quad
(\pp')^\al =
\begin{cases} 
|\al|^{-2}  (\alpha_2, -\alpha_1)(\phi_d)^\al, & \al \neq {\bf 0},\\
0, & \al  ={\bf 0},
\end{cases}
\eeqs
(compare to \eqref{eq_Fourier_coeff}). We then set $\pp \eq (\pp',p_3)$.
Since $\pp'$ and $p_3$ both have mean value zero, so does $\pp$.
Since $\pp$ is independent of $y_3$, \eqref{eq:EAS3} implies that
\begin{equation*}
\curly \pp = (\vcurl_{\yy'} p_3,\scurl_{\yy'} \pp'),
\end{equation*}
which equals $\pphi$ by \eqref{eq_Sunday1}; therefore $\pp\in \HH_\sharp(\ccurl,Y)$. Since
$\divy \pp = \ddiv_{\yy'} \pp'+\partial p_3/\partial y_3 = 0,$ $\pp\in \HH_\sharp(\ddiv,Y)$
and then the relation \eqref{eq_H1} implies that $\pp\in \HH^1_\sharp(Y)$. Since the solution
to \eqref{eq_vector_potential} is unique by Lemma \ref{lemma_smooth_vector_potential},
$\qq = \pp$, so that $\qq$ does not depend on $y_3$.

(ii)
Since $\curl \qq = \pphi \in \HH^s(Y)$, $\qq \in \HH^{1+s}(Y)$ by Lemma \ref{lem:A1}, with 
$|\qq|_{\HH^{1+s}(Y)} \leq |\pphi|_{\HH^s(Y)}.$
Using this with \eqref{eq_stability_vector_potential}, we find that
\begin{align}
\|\qq\|_{\HH^{1+s}(Y)}^2
=
\|\qq\|_{\HH^{1}(Y)}^2 + |\qq|_{\HH^{1+s}(Y)}^2
\label{tmp_linfty_2}
\leq
\left (1 + \pi^{-1}\right )^2\|\pphi\|_{\LL^2(Y)}^2 + |\pphi|_{\HH^s(Y)}^2
\leq
\left (1 + \pi^{-1}\right )^2\|\pphi\|_{\HH^s(Y)}^2,
\end{align}
Our goal is now to use Sobolev embedding to bound $\|\qq\|_{L^\infty(Y)}$
by $\|\qq\|_{\HH^{1+s}(Y)}$.  Recall that $\|v\|_{L^\infty}\lesssim \|v\|_{H^{d/2+s}}$ for
$0<s<1$ (see, e.g., \cite[Theorem 3.26]{Mc:00}), indicating that for $d=3$ we require the
$\HH^{3/2+s}$ norm of $\qq$. However, since $\qq$ and its components are only functions of
$\yy'\in Y' \subset \Rea^{d-1}$, and $d-1 \leq 2$, we only require the $\HH^{1+s}$ norm of $\qq$.
Given $0<s<1$, there exists a constant $\Cemb$ such that
\begin{equation}\label{eq_Sobolev_emb}
\|\psi\|_{\LL^\infty(Y')} \leq \Cemb\|\psi\|_{H^{1+s}(Y')}
\quad
\tfa \psi \in H^{1+s}(Y').
\end{equation}
Using (in this order) the definition of the $\LL^\infty(Y)$ norm \eqref{eq:Linftynorm},
the bound \eqref{eq_Sobolev_emb}, the fact that (by Part (i)) $\qq \in \HH^{1+s}(Y)$
and does not depend on the $y_d$ variable, and the property \eqref{eq_norm_scalar_vector2},
we have
\begin{equation}
\label{tmp_linfty_1}
\|\qq\|_{\LL^\infty(Y)}^2
\leq
\sum_{j=1}^d \|q_j\|_{L^\infty(Y')}^2
\leq
\Cemb^2 \sum_{j=1}^d \|q_j\|_{H^{1+s}(Y')}^2
=
\Cemb^2 \|\qq\|_{\HH^{1+s}(Y)}^2.
\end{equation}
The result \eqref{eq_vector_potential_Linfty} then follows from combining \eqref{tmp_linfty_1} and \eqref{tmp_linfty_2}
\end{proof}

\subsection{Bivariate vector potentials}

We now consider functions $\pphi: \Din \times Y \to \mathbb C$
such that
\begin{subequations}
\label{eq_assumption_pphi}
\begin{equation}
\pphi(\xx,\yy) = \hM(\yy) \vv(\xx),
\end{equation}
where $\hM$ is a $Y$-periodic matrix-valued function with zero mean value that does not depend on
$y_d$ and is such that
\begin{equation}
\hM_{j \ell} \in H^s(Y) \cap L^\infty(Y)
\end{equation}
for some fixed $s > 0$, and $\vv \in \HH^2(\Din)$. We further assume that
$\divy \pphi(\xx,\cdot) = 0$
in the sense of distribution for a.e.~$\xx \in \Din$.
\end{subequations}

\begin{lemma}\label{lem:A5}
Assume that $\pphi$ satisfies \eqref{eq_assumption_pphi} and additionally
that $\vv \in \CC^\infty(\overline{\Din})$. Then, for all $\xx \in \overline{\Din}$,
there exists a unique $\qq(\xx,\cdot) \in \HH^1_\sharp(Y)$ such that
\begin{equation}
\label{eq_bivariate_vector_potential}
\curly \qq(\xx, \cdot) = \pphi(\xx,\cdot)
\quad\tand\quad
\divy \qq(\xx, \cdot) = 0
\end{equation}
for all $\yy \in Y$. Furthermore, $\qq \in C^\infty(\overline{\Din}, \HH^1_\sharp(Y))$
with
\beq\label{eq:EAS4}
\N{(\partial^\bt_{\xx} \qq)(\xx,\cdot)}_{\HH^1(Y)} \leq \left( 1+ \pi^{-1}\right) \N{\hM(\cdot) \partial^\bt_{\xx} \vv(\xx,\cdot)}_{\LL^2(Y)}
\eeq
and  $\qq(\xx,\cdot) \in C^\infty(\overline{\Din}, \LL^\infty(Y))$ with 
\begin{equation}
\label{eq_bivariate_vector_potential_Linfty}
\|(\partial^\bt_{\xx} \qq)(\xx,\cdot)\|_{\LL^\infty(Y)}
\leq
\Cemb\left(1+ \pi^{-1}\right) \|\hM\|_{\HH^s(Y)} |(\partial^\bt \vv)(\xx)|
\end{equation}
for all $\bt \in \ZZZ^d_+$ and $\bx \in \overline{\Din}$. In particular, for all $\bt \in \ZZZ^d_+$,
\begin{equation}
\label{eq_bivariate_sobolev_Din}
\|(\partial^\bt_{\xx} \qq)^\eps\|_{\LL^2(\Din)}
\leq
\Cemb \left(1+ \pi^{-1}\right)\|\hM\|_{\HH^s(Y)} \N{\partial^\bt \vv}_{\LL^2(\Din)}, \quad\tand
\end{equation}
\begin{equation}
\label{eq_bivariate_sobolev_Gamma}
\|(\partial^\bt_{\xx} \qq)^\eps\|_{\LL^2(\Gamma)}
\leq
\Cemb\left(1+ \pi^{-1}\right)\|\hM\|_{\HH^s(Y)} \N{\partial^\bt \vv}_{\LL^2(\Gamma)}.
\end{equation}
\end{lemma}

\begin{proof}
The existence and uniqueness of the solution $\qq(\xx,\cdot) \in \HH^1_\sharp(Y)$
for all $\xx \in \overline{\Din}$ is a direct consequence of Theorem \ref{theorem_vector_potential}.
Now, set $\delta > 0$ and define $D_{\rm in}^\delta  \subset \Din$ as the subset of $\Din$ of
distance $\geq \delta$ from $\Gamma$. With $\xx \in D_{\rm in}^\delta $ and $\xx' \in \Din$,
\begin{equation*}
\curly \left (\qq(\xx,\cdot)-\qq(\xx',\cdot)\right )
=
\pphi(\xx,\cdot)-\pphi(\xx',\cdot) \quad\tand\quad
\divy \big(\qq(\xx,\cdot)-\qq(\xx',\cdot)\big)
=
0,
\end{equation*}
so that, by \eqref{eq_stability_vector_potential},
\begin{align*}
\|\qq(\xx,\cdot)-\qq(\xx',\cdot)\|_{\HH^{1}(Y)}
&\leq\left(1+  \pi^{-1}
\right)
\|\pphi(\xx,\cdot)-\pphi(\xx',\cdot)\|_{\LL^2(Y)}\\
&\leq\left(1+ \pi^{-1}
\right)
|\xx-\xx'| \sup_{\xx \in \Din}\|\hM(\cdot) \gradx \vv(\xx,\cdot)\|_{\LL^2(Y)},
\end{align*}
and, by \eqref{eq_vector_potential_Linfty},
\begin{align*}
\|\qq(\xx,\cdot)-\qq(\xx',\cdot)\|_{\LL^\infty(Y)}
&\leq\Cemb\left(1+ \pi^{-1}\right)
\|\pphi(\xx,\cdot)-\pphi(\xx',\cdot)\|_{\HH^s(Y)},\\
&\leq
\Cemb\left(1+ \pi^{-1}
\right) |\xx-\xx'| \sup_{\xx \in \Din}\|\hM(\cdot) \gradx \vv(\xx,\cdot)\|_{\HH^s(Y)}.
\end{align*}
Since the above inequalities hold for any $\xx'\in \Din$, we deduce that
\begin{equation*}
\sup_{\xx \in D_{\rm in}^\delta } \|\gradx \qq(\xx,\cdot)\|_{\HH^{1}(Y)}
\leq \left(1+ \pi^{-1}\right)
\sup_{\xx \in \Din} \|\hM(\cdot) \gradx \vv(\xx,\cdot)\|_{\LL^2(Y)},
\end{equation*}
and
\begin{equation*}
\sup_{\xx \in D_{\rm in}^\delta } \|\gradx \qq(\xx,\cdot)\|_{\LL^\infty(Y)}
\leq \Cemb\left(1+ \pi^{-1}\right)
\sup_{\xx \in \Din} \|\hM(\cdot) \gradx \vv(\xx,\cdot)\|_{\HH^s(Y)}.
\end{equation*}
But now, since the upper bounds are independent of $\delta$, we obtain
\eqref{eq:EAS4} and \eqref{eq_bivariate_vector_potential_Linfty} 
for $|\bt| = 1$ by letting $\delta \to 0$. We have there proved
that $\qq \in C^1(\overline{\Din},\HH^1(Y)) \cap C^1(\overline{\Din},\LL^\infty(Y))$;
we can therefore differentiate \eqref{eq_bivariate_vector_potential} with respect to $x_j$,
to obtain that
$\partial_{x_j} \qq \in C^0(\overline{\Din},\HH^1(Y)) \cap C^0(\overline{\Din},\LL^\infty(Y))$.
By linearity, we can then repeating the above argument with
$\widetilde \qq \eq \partial_{x_j} \qq$ and $\widetilde \pphi = \partial_{x_j} \pphi$.
We then obtain \eqref{eq:EAS4} and \eqref{eq_bivariate_vector_potential_Linfty}
by induction on $|\bt|$.

By integrating over $\Din$ and using \eqref{eq_bivariate_vector_potential_Linfty}, we have
\begin{align}
\label{eq_Sunday2}
\|(\partial^\bt_{\xx} \qq)^\eps\|_{\LL^2(\Din)}^2
=
\int_{\Din} \left |(\partial^\bt_{\xx} \qq)\left (\xx,\frac{\xx}{\eps}\right )\right |^2 \rd\xx
&\leq
\int_{\Din} \left \|(\partial^\bt_{\xx} \qq)
\left (\xx,\cdot\right )\right \|_{\LL^\infty(Y)}^2 \rd\xx,
\\
&\leq
(\Cemb)^2  \left(1+ \pi^{-1}\right)^2\|\hM\|_{\HH^s(Y)}^2
\int_{\Din}  |(\partial^\bt \vv)(\xx)|^2 \rd\xx,\nonumber
\end{align}
and \eqref{eq_bivariate_sobolev_Din} follows.
The bound \eqref{eq_bivariate_sobolev_Gamma} follows similarly, but
integrating over $\Gamma$.
\end{proof}

\begin{theorem}
\label{theorem_bivariate_vector_potential}
Assume that $\pphi: \Din \times Y \to \CCC$ satisfies the three conditions in \eqref{eq_assumption_pphi} with
$\vv \in \HH^2(\Din)$. Then there exists a unique vector potential
$\qq \in H^2(\Din,\HH^1_\sharp(Y))$ such that
\begin{equation*}
\curly \qq(\xx,\cdot) = \pphi(\xx,\cdot)
\quad\tand\quad
\divy \qq(\xx,\cdot) = 0
\end{equation*}
for a.e.~$\xx \in \Din$.
Futhermore,
\begin{equation}
\label{eq_curlx_qq_eps_Din}
\|(\curlx \qq)^\eps\|_{\LL^2(\Din)} \leq \Cemb\left(1+ \pi^{-1}\right) \|\hM\|_{\HH^s(Y)}|\vv|_{\HH^1(\Din)},
\end{equation}
\begin{equation}
\label{eq_qq_eps_Gamma}
\N{\curl \qq^\eps  \cdot \nnu}_{\HH^{-1}(\Gamma)}
\leq
\N{\qq^\eps \times \nnu}_{\LL^2(\Gamma)}
\leq
\Cemb\left(1+ \pi^{-1}\right)
\N{\hM}_{\HH^s(Y)}\N{\vv}_{\LL^2(\Gamma)},
\end{equation}
and
\begin{equation}
\label{eq_curl_qq_eps_Gamma}
\N{\curl \qq^\eps  \cdot \nnu}_{\LL^2(\Gamma)}
\leq
\frac{1}{\eps} \|\hM\|_{\LL^\infty(Y)} \N{\vv}_{\LL^2(\Gamma)}
+
\Cemb\left(1+ \pi^{-1}\right) \|\hM\|_{\HH^s(Y)} |\vv|_{\HH^1(\Gamma)}.
\end{equation}
In addition, if $kR_0 \gtrsim 1$,
\begin{equation}
\label{eq_curl_qq_eps_Gamma_half}
\N{\curl \qq^\eps  \cdot \nnu}_{H^{-1/2}(\Gamma)}
\lesssim
(1 + (k\eps)^{-1/2})
\left ( k\N{\vv}_{\LL^2(\Din)} + |\vv|_{\HH^{1}(\Din)} \right )
\end{equation}
(where the omitted constants depend on $\|\hM\|_{\HH^s(Y)}$ and $\|\hM\|_{\LL^\infty(Y)}$).
\end{theorem}

\begin{proof}
Once we have shown that $\qq$ exists, is unique, and satisfies the bounds 
\eqref{eq_bivariate_sobolev_Din} and \eqref{eq_bivariate_sobolev_Gamma} for $|\beta|\leq 1$
the bound \eqref{eq_curlx_qq_eps_Din} is a direct consequence of \eqref{eq_bivariate_sobolev_Din}, and 
the bound on $\|\qq^\eps \times \nnu\|_{\LL^2(\Gamma)}$
in \eqref{eq_qq_eps_Gamma}
is a direct consequence of
\eqref{eq_bivariate_sobolev_Gamma}.
We now prove the bound 
\begin{equation}
\label{eq:Monday1}
\N{\curl \qq^\eps  \cdot \nnu}_{\HH^{-1}(\Gamma)}
\leq
\N{\qq^\eps \times \nnu}_{\LL^2(\Gamma)}
\end{equation}
to complete the proof of  \eqref{eq_qq_eps_Gamma}.
Consider a test function $\phi \in C^\infty(\overline{\Din})$. Observe that
\begin{equation*}
\big(\curl \qq^\eps  \cdot \nnu,\phi\big)_\Gamma
=
\big(\curl \qq^\eps,\grad \phi\big)_{\Din}
=
-\big(\qq^\eps \times \nnu, \grad \phi\big)_\Gamma,
\end{equation*}
with the first equality holding by the divergence theorem and the fact that
$\div (\curl \qq^\eps) = 0$, and the second equality holding by the Green-type
identity in, e.g., \cite[Theorem 3.29]{Mo:03} and the fact that $\curl (\grad \phi)={\bf 0}$.
The bound \eqref{eq:Monday1} then follows since $\gamma(C^\infty(\overline{\Din}))$ is
dense in $H^1(\Gamma)$ (see, e.g., \cite[Page 276]{ChGrLaSp:12}).
To prove \eqref{eq_curl_qq_eps_Gamma}, we observe that, by \eqref{eq_eps_diff},
\begin{equation*}
\curl \qq^\eps
=
(\curlx \qq)^\eps + \frac{1}{\eps}(\curly \qq)^\eps =
(\curlx \qq)^\eps + \frac{1}{\eps} \hM^\eps \vv.
\end{equation*}
Since $(\hM)_{j\ell} \in L^\infty(Y)$ for $1 \leq j,\ell \leq d$,
\eqref{eq_curl_qq_eps_Gamma} follows from this last equality by taking norms and using
\eqref{eq_bivariate_sobolev_Gamma}.
To prove \eqref{eq_curl_qq_eps_Gamma_half}, 
we let $T: \LL^2(\Gamma) \to \HH^{-1}(\Gamma)$ be the operator
$T: \vv \to \curl \qq^{\eps} \cdot \nnu$. 
By interpolation (see, e.g., \cite[Appendix B]{Mc:00})
using the bounds \eqref{eq_qq_eps_Gamma} and \eqref{eq_curl_qq_eps_Gamma}
(similar to in the proof of Lemma \ref{lemma_u1_eps}), we have 
\begin{equation*}
\N{\curl \qq^\eps  \cdot \nnu}_{\HH^{-1/2}(\Gamma)}
\lesssim
\eps^{-1/2} \N{\vv}_{\LL^2(\Gamma)} + |\vv|_{\HH^{1/2}(\Gamma)}\quad\tfa \vv \in \HH^1(\Din).
\end{equation*}
By the multiplicative trace inequality \eqref{eq:multtrace},
\begin{align*}
&\eps^{-1} \N{\vv}_{\LL^2(\Gamma)}^2
\lesssim
\frac{1}{R \eps}
\N{\vv}_{\LL^2(\Din)}^2 + \eps^{-1} \N{\vv}_{\LL^2(\Din)}|\vv|_{\HH^1(\Din)}\\
&\qquad\lesssim
\frac{1}{R \eps}
\N{\vv}_{\LL^2(\Din)}^2 + \frac{k}{\eps} \N{\vv}_{\LL^2(\Din)}^2 + \frac{1}{k\eps}|\vv|_{\HH^1(\Din)}^2
\lesssim
\frac{1}{k\eps} \left (1 + \frac{1}{kR}\right )k^2
\N{\vv}_{\LL^2(\Din)}^2 + \frac{1}{k\eps}|\vv|_{\HH^1(\Din)}^2,
\end{align*}
and then \eqref{eq_curl_qq_eps_Gamma_half} follows from the assumption that $kR_0 \gtrsim 1$.

It therefore remains to show that $\qq$ exists, is unique, and satisfies the bounds 
\eqref{eq_bivariate_sobolev_Din} and \eqref{eq_bivariate_sobolev_Gamma} for $|\beta|\leq 1$.
Given $\vv \in \HH^2(\Din)$ there exist $\vv_n \in \CC^\infty(\overline{\Din})$ such that $\vv_n \rightarrow \vv$ in $\HH^2(\Din)$ as $n\tendi$. Let $\qq_n\in C^\infty(\overline{\Din}, \HH^1_\sharp(Y))$ be the solution of 
\beqs
\curly \qq_n(\xx,\yy) = \hM(\yy)\vv_n(\xx), \quad\divy \qq_n(\xx,\yy) =0,
\eeqs
which exists by Lemma \ref{lem:A5}.
By linearity,
\beqs
\curly \big(\qq_n-\qq_m\big)(\xx,\yy) = \hM(\yy)\big(\vv_n-\vv_m\big)(\xx), \quad\divy \big(\qq_n-\qq_m\big)(\xx,\yy) =0.
\eeqs
The definition of the Bochner norm
\beqs
\N{\qq}^2_{H^2(\overline{\Din}, \HH^1_\sharp(Y))}:= 
\sum_{|\bt|\leq 2}\N{\partial^{\bt}_{\xx} \qq}^2_{L^2(\overline{\Din}, \HH^1_\sharp(Y))}
\eeqs
and the bound \eqref{eq:EAS4} 
imply that
\beqs
\N{\qq_n-\qq_m}_{\HH^2(\overline{\Din}, \HH^1_\sharp(Y))}\leq \left(1+\pi^{-1}\right) \big\|\hM\big\|_{L^{\infty}(Y)}\N{\vv_n-\vv_m}_{\HH^2(\overline{\Din})}
\eeqs
for all $n, m$. Therefore $\qq_n$ is Cauchy sequence in $H^2(\overline{\Din}, \HH^1_\sharp(Y))$
and converges to some $\qq \in H^2(\overline{\Din}, \HH^1_\sharp(Y))$. By continuity of weak
derivative operators, $\curly \qq = \pphi$ and $\divy \qq = 0$ for a.~e.~$\bx \in \Din$.
Furthermore, by \eqref{eq_bivariate_vector_potential_Linfty},  
\beqs
\int_{\Din}\|(\partial^\bt_{\xx} \qq_n)(\xx,\cdot)\|_{\LL^\infty(Y)}^2\, \rd \bx
\leq
(\Cemb)^2  \left(1+\pi^{-1}\right)^2\|\hM\|_{\HH^s(Y)}^2 \int_{\Din}|(\partial^\bt \vv_n)(\xx)|^2\,\rd \bx
\eeqs
and
\beqs
\int_{\Din}\|\partial^\bt_{\xx} (\qq_n-\qq_m)(\xx,\cdot)\|_{\LL^\infty(Y)}^2\, \rd \bx
\leq
(\Cemb)^2  \left(1+\pi^{-1}\right)^2\|\hM\|_{\HH^s(Y)}^2 \int_{\Din}|\partial^\bt (\vv_n-\vv_m)(\xx)|^2\,\rd \bx
\eeqs
for $|\beta|\leq 2$ and $n,m \in \mathbb N$. Therefore $\qq_n$ is a Cauchy sequence in $H^2(\Din, \LL^\infty(Y))$ and 
\beq\label{eq:EAS6}
\int_{\Din}\|(\partial^\bt_{\xx} \qq)(\xx,\cdot)\|_{\LL^\infty(Y)}^2\,\rd \bx
\leq
(\Cemb)^2 \left(1+\pi^{-1}\right)^2 \|\hM\|_{\HH^s(Y)}^2 \int_{\Din} |(\partial^\bt \vv)(\xx)|^2 \, \rd \bx
\eeq
for $|\beta|\leq 2$.  The bounds \eqref{eq_bivariate_sobolev_Din} and
\eqref{eq_bivariate_sobolev_Gamma} for $|\beta|\leq 2$ therefore follow from 
\eqref{eq:EAS6} using similar reasoning to in \eqref{eq_Sunday2}.
\end{proof}

\begin{theorem}
\label{theorem_vector_potential_main}
Let $\Gamma$ be $C^{1,1}$. Assume that $u_0 \in H^2(\Din)$, 
$\hM\in H_\sharp(Y) \cap L^\infty(Y)$, for some $s>0$, and that $\hM$ does not depend on $y_d$.
Assume further that $\divy (\hM \bw)=0$ for all $\bw \in \Com^d$ independent of $\by$.
Then there exists a unique vector potential $\qq \in H^1(\Din,\HH^1_\sharp(Y))$ such that
\begin{equation*}
\curly \qq(\xx,\cdot) = \hM(\yy)\nabla u_0(\xx) \quad\tand\quad \divy \qq(\xx,\cdot) = 0
\end{equation*}
for a.e.~$\xx \in \Din$. Futhermore,
\begin{equation}
\label{eq_C1}
\|(\curlx \qq)^\eps\|_{\LL^2(\Din)}
\lesssim
|u|_{H^2(\Din)},
\qquad
\N{\qq^\eps \times \nnu}_{\LL^2(\Gamma)}
\lesssim
|u_0|_{H^1(\Gamma)},
\end{equation}
and, if $kR_0\gtrsim 1$,
\begin{equation}
\label{eq_C2}
\N{\curl \qq^\eps  \cdot \nnu}_{\HH^{-1/2}(\Gamma)}
\lesssim
(1+(k\eps)^{-1/2}) \left (k\|u_0\|_{H^1(\Din)} + \|u_0\|_{H^2(\Din)}\right )
\end{equation}
(where the omitted constants depend on $\|\hM\|_{\HH^s(Y)}$ and $\|\hM\|_{\LL^\infty(Y)}$).
\end{theorem}

\bpf
By applying Theorem \ref{theorem_bivariate_vector_potential} with $\vv=\nabla u_0$,
we see that $\qq$ exists and the bounds in \eqref{eq_C1}, \eqref{eq_C2} hold when $u_0\in H^3(\Din)$.
By the density of $H^3(\Din)$ in $H^2(\Din)$, and the fact that right-hand sides
of the bounds \eqref{eq_C1} are controlled by $\|u_0\|_{H^2(\Din)}$, $\qq$ exists and the bounds in \eqref{eq_C1}, \eqref{eq_C2} hold for $u_0\in H^2(\Din)$.
\epf

\section{Bounding $\Cse$ for $kR$ sufficiently small}\label{app:small_k}

As discussed in \S\ref{sec:homo_discuss}, in \cite[Bottom of Page 2539 and top of Page 2540]{CaGuMo:16} $\Cse$ is assumed to be independent of $\eps$ and to only depend on $n_{\min}, n_{\max}, A_{\min},$ and $A_{\max}$. We now show that this is true if
$kR$ is sufficiently small.

By \cite[Lemma 3.3]{MeSa:10}, given $k_0, R_0>0$, there exist $\CTR= \CTR(k_0 R_0)$ such that 
\beq\label{eq:CDtN}
- \Re \big\langle \DtN \phi,\phi\big\rangle_{\GR} \geq \CTR R^{-1}\N{\phi}^2_{L^2(\GR)}  \quad\tfa \phi \in H^{1/2}(\GR), \,\,k\geq k_0, \tand R\geq R_0.
\eeq
By, e.g., \cite[Corollary A.15]{ToWi:05}, there exists $\CPF = \CPF(\Omega_-)$ (`PF' standing for `Poincar\'e--Friedrichs') such that
\beq\label{eq:CPF}
R^{-2} \N{v}^2_{L^2(B_R)} \leq \CPF \left(R^{-1}\N{ v}^2_{L^2(\GR)} + \N{\nabla v}^2_{L^2(B_R)}\right) \quad\tfa v\in H^1(B_R).
\eeq

\ble[Coercivity of $b(\cdot,\cdot)$ for $kR$ sufficiently small]\label{lem:coercivity}
If 
\beq\label{eq:lower_coercivity}
kR \leq 3 n_{\max} \CPF\big( \min\{ A_{\min}, \CTR\}\big)^{-1},
\eeq
then
\beq\label{eq:coercivity}
\Re b(v,v) \geq \frac{ \min\{ A_{\min}, \CTR\}}{2} \N{\nabla v}^2_{L^2(B_R)} +\frac{n_{\max}}{2} k^2\N{v}^2_{L^2(B_R)} \quad \tfa v\in H^1(B_R).
\eeq
\ele
The bound 
$\Cse \leq (
\min\{
\min\{ A_{\min}, \CTR\}\, , \, n_{\max}
\}
)^{-1}$
under the assumption \eqref{eq:lower_coercivity} then immediately follows from the Lax--Milgram lemma and the definition of $\Csol$ \eqref{eq_Csol_def}

\bpf[Proof of Lemma \ref{lem:coercivity}]
Using the definition of $b(\cdot,\cdot)$ \eqref{eq:sesqui} and the inequalities \eqref{eq:CDtN} and \eqref{eq:CPF}, we have
\begin{align*}
&\Re b(v,v) \geq A_{\min} \N{\nabla v}^2_{L^2(B_R)} - k^2 n_{\max} \N{v}^2_{L^2(B_R)} + R^{-1}\CTR \N{ v}^2_{L^2(\GR)},\\
& \geq \min\{ A_{\min}, \CTR\} \Big (\N{\nabla v}^2_{L^2(B_R)} + \frac{1}{R}\N{ v}^2_{L^2(\GR)}\Big) - k^2 n_{\max}  \N{v}^2_{L^2(B_R)},\\
& \geq  \frac{\min\{ A_{\min}, \CTR\}}{2}\Big(\N{\nabla v}^2_{L^2(B_R)} + \frac{1}{R}\N{ v}^2_{L^2(\GR)}\Big) + n_{\max} \left( - 1 + \frac{\min\{ A_{\min}, \CTR\}}{2n_{\max} (kR)^2 \CPF}\right)k^2 \N{v}^2_{L^2(B_R)} ,
\end{align*}
and the result \eqref{eq:coercivity} follows.
\epf

\section*{Acknowledgements} 

The authors thank Kirill Cherednichenko (University of Bath),
Daniel Peterseim (Universit\"at Augsburg), and Barbara Verf\"urth (KIT)
for useful discussions about the homogenisation literature.
The authors also thank the anonymous referees for their constructive comments on the paper.
EAS was supported by EPSRC grant EP/R005591/1. 

\bibliographystyle{siamplain}
\bibliography{bibliography}

\end{document}